%% file: main_arxiv_Bouchra.tex
\documentclass[final]{amsart}
\usepackage[utf8x]{inputenc}
\usepackage{graphicx}
\usepackage{amsmath,mathrsfs,amsthm,amssymb}
\usepackage[a4paper, total={6.2in, 9.2in}]{geometry}
\usepackage{bm}
\usepackage{color}
\usepackage[colorlinks,citecolor=blue]{hyperref}
\usepackage[comma, authoryear]{natbib}
\usepackage{natbib}
\usepackage{cleveref}
\usepackage{subcaption}
\usepackage{multirow}
\usepackage[color,notref,notcite]{showkeys}
\usepackage{enumerate}
\usepackage{float} 
\definecolor{labelkey}{rgb}{0.6,0,1}

\newcommand{\ds}{\displaystyle}
\newcommand{\N}[1]{\lfloor N{#1}\rfloor}
\definecolor{darkblue}{rgb}{0.0, 0.0, 0.55}
\newcommand{\BB}[1]{\textcolor{red}{#1}}

\def\btheta{{\boldsymbol\theta}}
\def\bSigma{{\boldsymbol\Sigma}}
\def\bGamma{{\boldsymbol\Gamma}}
\def\bLambda{{\boldsymbol\Lambda}}
\def\bTheta{{\boldsymbol\Theta}}

\def\bmu{{\boldsymbol\mu}}
\def\bepsilon{{\boldsymbol\epsilon}}
\def\bvarphi{{\boldsymbol\varphi}}

\def\cC{\mathcal{C}}
\def\dB{\mathbb{B}}
\def\dD{\mathbb{D}}

\def\dF{\mathbb{F}}

\def\dM{\mathbb{M}}
\def\dP{\mathbb{P}}
\def\dR{\mathbb{R}}
\def\dS{\mathbb{S}}

\def\dT{\mathbb{T}}
\def\dV{\mathbb{V}}

\def\bc{\mathbf{c}}

\def\bQ{\mathbf{Q}}
\def\br{\mathbf{r}}
\def\bR{\mathbf{R}}

\def\bY{\mathbf{Y}}
\def\bZ{\mathbf{Z}}
\def\cO{\mathcal{O}}

\newcommand{\I}{\ensuremath{\mathbb{I}}}

\newtheorem{theorem}{Theorem}[section]
\newtheorem{lemma}{Lemma}[section]
\newtheorem{proposition}{Proposition}[section]
\newtheorem{corollary}{Corollary}[section]
\newtheorem{assumption}{Assumption}
\newtheorem{remark}{Remark}

\numberwithin{equation}{section}
\begin{document}
\title[Sequential change-point detection for GOU processes]{Sequential change-point detection for generalized Ornstein-Uhlenbeck processes}

\author{Yunhong Lyu}
\address{Department of Mathematics and Statistics, Trent University}
\email{yunhonglyu@trentu.ca}

\author{Bouchra R. Nasri}
\address{Département de médecine sociale et préventive, École de santé publique, Université de Montréal}
\email{bouchra.nasri@umontreal.ca}

\author{Bruno N. Rémillard}
\address{Department of Statistics and Business Analytics, United Arab Emirates University and Department of Decision Sciences, HEC Montréal}
\email{bruno.remillard@hec.ca}

\thanks{Funding in partial support of this work was provided by the Fonds
québécois de la recherche en santé and the Natural Sciences and Engineering Research Council of Canada.}



\begin{abstract}
In this article, we study sequential change-point methods for discretely observed generalized Ornstein-Uhlenbeck processes with periodic drift. Two detection methods are proposed, and their respective performance is studied through numerical experiments for several choices of parameters.
\end{abstract}

\maketitle
\section{Introduction}

The main focus of sequential change-point detection is the identification of shifts in distribution of data collected sequentially over time, where one of the ear;lier reference is  \cite{page1954continuous}. In this work, we are interested in Ornstein-Uhlenbeck (O-U)  and its recent extensions called generalized O-U (GOU) processes \citep{nkurunziza2019improved,lyu2023inference}. The O-U process is known for its mean-reverting property, meaning that its long-term expectation is constant. This is no longer true for the GOU process, whose main drift term can be periodic. In fact, we are focusing in detecting change-points in the parameters of a GOU process in a sequential way, by looking at the process observed at discrete times. Unlike \cite{Berkes/Gombay/Horvath/Kokoszka:2004}, who considered GARCH processes,  we propose two sequential monitoring schemes to detect parameter changes in discretely observed GOU processes: one based on residuals and one based on estimators. In addition, we also establish the asymptotic results of the detectors under both the null and alternative hypotheses. Note that in the particular case of a GARCH(0,0), the methodology proposed by \cite{Berkes/Gombay/Horvath/Kokoszka:2004} is based on cumulative sums of residuals. Motivated by \cite{Berkes/Gombay/Horvath/Kokoszka:2004}, we will consider sequential based on a statistic $\Gamma(N,K)$ of observations $X_1,\ldots, X_{N+K}$, together with a boundary function $g(N,K)$. The decision rule is to stop and reject the null hypothesis $H_0$ of no change-point at
\begin{equation*}
    \tau_\Gamma(N)=\begin{cases}
        \inf\{K\ge 1: \Gamma(N,K)\ge g(N,K)\};\\
        \infty\quad \mbox{if}~\Gamma(N,K)< g(N,K), \text{ for all } K\ge 1.
    \end{cases}
\end{equation*}
Note that if there is an instantaneous change immediately after the first observation of the new data, this leads to the fact that $K_*=1$. We choose the statistics and boundaries so that
\begin{equation}
\label{FalseAlarm}
\lim_{N\to\infty}\dP_{H_0}\left(\tau_\Gamma(N)<\infty\right)=\alpha,
\end{equation}
where $0<\alpha<1$ is the prescribed type-I error. We will also choose the boundary so that under the alternative hypothesis $H_1$ of a change-point,
\begin{equation}
\label{CPHa}
\lim_{N\to\infty}\dP_{H_1}\left(\tau_\Gamma(N)<\infty\right)=1.
\end{equation}
Condition \eqref{FalseAlarm} ensures that the probability of a false alarm is asymptotically $\alpha$, while condition
\eqref{CPHa} means that under $H_1$, a change-point is detected with probability approaching $1$.

This paper is structured as follows: the GOU model and some preliminary results for the estimation of its parameters are presented in \Cref{sec:Estimation}, while the discretized process is defined in \Cref{sec:discrete}.  The limiting behaviour of the estimated drift parameters is presented together with a non-parametric goodness-of-fit test in Section \ref{ssec:gof}. In \Cref{sec:Asymptotic}, we propose two processes to detect the sequential change-point in the GOU process. We then prove the asymptotic property of the two detectors under both the null and alternative hypotheses. Numerical experiments for GOU processes with drift
$\mu_1 +\mu_{2} \sqrt{2}\cos{(2\pi  t)}-a x
$ are considered in \Cref{sec:Simulation} to assess the performance of the two statistics. All simulations were performed using the CRAN package GenOU \citep{Lyu/Nasri/Remillard:2025a}.

\section{Parameter estimation}\label{sec:Estimation}
In what follows, we adopt the following notations. Consider a probability space $\left(\Omega, \mathscr{F}, \rm{P}\right)$ with $\mathscr{F}$ as a $\sigma-$field on the sample space $\Omega$, and $\rm{P}$ as a probability measure. Furthermore, let $L^p$ represent the space of measurable functions that are $ p$-integrable, where $p\ge 1$. For mathematical convenience, we assume that $\mathscr{F}$ is complete. In addition, $O_{\rm{P}}(1)$ denotes a random quantity such that  is bounded in probability, and $o_{\rm{P}}(1)$ indicates convergence to 0 in probability. We define $\mathbb{R}^{+}=[0,+\infty)$. The notation ${I}_{n}$ is used to represent the $n$-dimensional identity matrix.
Let $\mathbb{I}_{A}$ denote the indicator function of the event $A$,  $\top$ denote the transpose of a matrix, and let
$\btheta = (\bmu,a)^{\top}=(\mu_{1},~~\mu_{2},\dots,~~\mu_{p},a)^{\top}\in\cO$, where $\cO \subset\mathbb{R}^{p}\times(0,\infty)$ is the parameter space and $\bvarphi{(t)}=[\varphi_{1}(t),~\varphi_{2}(t),~\varphi_{3}(t),~\ldots~,~\varphi_{p}(t)]^\top,$ where for each $i=1,2,\dots,p,$ the function $\varphi_{i}(t)$ is a real-valued function of $t$. The \textbf{absolute maximum norm} of vectors and matrices is denoted by
$\|\cdot \|$.
We consider the following Stochastic Differential Equation (SDE):
\begin{equation}\label{eq2}
dX_{t}=S({\btheta},t,X_{t})dt+\sigma dB_{t},
\end{equation}
using notations similar to \cite{lyu2023inference}, where $\ds S({\btheta},t,x)=\bmu^\top \bvarphi(t)-a x$, $\{B_{t}, t\ge 0\}$ is a standard Brownian motion, and $\sigma$ is a positive  constant.
\begin{remark}\label{remark1}
In contrast to estimating drift parameters for continuous-time diffusion processes, the assumption of a known diffusion parameter $\sigma$ is a prevalent practice. This choice comes from the singularity of measures corresponding to various diffusion parameters, which facilitates the direct computation of $\sigma$ from a solitary continuous-time observation path, obviating the need for estimation procedures.
\end{remark}

In order to describe the solution of the SDE, we need the following regularity conditions:
\begin{assumption}\label{hyp:1}
The distribution of the initial value, $X_{0}$, of the SDE in \eqref{eq2} does not depend on the drift parameter $\btheta.$ Furthermore, $X_{0}$ is independent of $\{B_{t}: t\ge 0\}$ and $E[|X_{0}|^{d}]<\infty,$ for some $d\ge 2.$
\end{assumption}
\begin{assumption}\label{hyp:2}
For any $T>0$, the base function $\{\bvarphi(t), t\in[0,T]\}$ is Riemann-integrable on [0,T] and possess the following properties:
\begin{enumerate}
\item $\varphi_1(t)\equiv 1$.
    \item Periodicity: $\bvarphi(t+\upsilon)=\bvarphi(t)$, for all $t\in[0,T]$, where $\upsilon>0$ is the period.
    \item Orthogonality in $L^{2}\left([0,\upsilon],\ds\frac{1}{\upsilon}d\lambda\right):\int_{0}^{\upsilon}\bvarphi(t)\bvarphi^\top(t)dt=\upsilon {I}_{p}$
\end{enumerate}
\end{assumption}
\begin{remark}\label{rem:a}
  It follows from part 3 of Assumption \ref{hyp:2} that $\ds \bar \bvarphi= \frac{1}{\upsilon}\int_0^{\upsilon} \bvarphi(s)ds = e_1 = (1,0\ldots,0)^\top$. In addition, since the base function $\bvarphi$ is bounded on [0,T] and $\upsilon-$periodic, this implies that $\bvarphi$ is bounded on $ \mathbb{R}_{+}=[0,+\infty).$ Furthermore, there exists a positive constant $K_{\bvarphi}$, such that
\begin{equation}\label{kphi}
\left\|\bvarphi\right\|_\infty = \sup_{t\ge 0}\left\|\bvarphi(t)\right\|\le K_{\bvarphi}.
\end{equation}
Without loss of generality, we assume $\upsilon=1$ in the following sections.
\end{remark}
\subsection{Maximum likelihood estimation}

First, we introduce the log-likelihood function, if the process is observed continuously on the interval $[0,T]$.
From \cite[Theorem 7.7]{liptser2001statistics}, the log-likelihood function of the process defined by \eqref{eq2}   is
\begin{equation}\label{eq2.3}
{\rm log}\mathcal{L}([0,T],\btheta)=\frac{1}{\sigma^{2}}\int_{0}^{T}S({\btheta,t,X_{t}})dX_{t}-\frac{1}{2\sigma^{2}}\int_{0}^{T}S^{2}({\btheta,t,X_{t}})d{t},
\end{equation}
where $T>0$; see also \cite{dehling2010drift, nkurunziza2018estimation}. Furthermore, set
\begin{equation}
\bQ_{[t_0,t_1]}=\begin{bmatrix} \ds\int_{t_0}^{t_1}\bvarphi(t)\bvarphi^\top(t)dt&-\ds\int_{t_0}^{t_1}\bvarphi(t)X_{t}dt\\
-\ds\int_{t_0}^{t_1}\bvarphi^\top(t)X_tdt&\ds\int_{t_0}^{t_1}X_{t}^{2}dt
\end{bmatrix},
\label{Qmatrix}
\end{equation}
\begin{equation}
\tilde{\bR}_{[t_0,t_1]}=\begin{bmatrix} \ds\int_{t_0}^{t_1}\bvarphi^\top(t)dX_{t}&-\ds\int_{t_0}^{t_1}X_{t}dX_{t}\\
\end{bmatrix}^{\top},\quad{} \bR_{[t_0,t_1]}=\begin{bmatrix} \ds\int_{t_0}^{t_1}\bvarphi^\top(t)dB_{t}&-\ds\int_{t_0}^{t_1}X_{t}dB_{t}\\
\end{bmatrix}^{\top}
\label{Rmatrix},
\end{equation}
for $0\le t_0<t_1 \le T$. Then,
\begin{equation}
{\rm log}\mathcal{L}([0,T],\btheta)=\frac{1}{\sigma^{2}}\btheta^{\top}\tilde{\bR}_{[0,T]}-\frac{1}{2\sigma^{2}}\btheta^{\top}\bQ_{[0,T]}\btheta.\nonumber
\end{equation}
It is then easy to check that the maximum likelihood estimator (MLE)  is
$\check\btheta_{T}=\bQ_{[0,T]}^{-1}\tilde{\bR}_{[0,T]}$.
\begin{proposition}
\label{MLE_0206}
Suppose that Assumptions \ref{hyp:1}-\ref{hyp:2} hold. Then,
$\check\btheta_{T}=\btheta+\sigma \bQ_{[0,T]}^{-1}\bR_{[0,T]}$.
\end{proposition}
The proof can be found in Proposition 4.1 from \cite{dehling2010drift}.

Finally, one can look at the asymptotic behaviour of this estimator, which will be used when studying the asymptotic behaviour of the change-point statistics.

To simplify some  expressions, set
\begin{equation}\label{sigma}
\bSigma=\begin{bmatrix}
    I_p & \bLambda\\
    \bLambda^{\top} & \omega
\end{bmatrix},\quad \bLambda = -\int_{0}^{1} \widetilde{h}(t)\bvarphi(t)dt, \quad \omega = \int_{0}^{1} \widetilde{h}^{2}(t)dt + \frac{\sigma^2}{2a}.
\end{equation}
Based on Proposition 3.3 in \cite{lyu2023inference}, we have the following result, proven in Appendix \ref{app:pf1}, which plays an important role in obtaining the asymptotic property of test statistics.
\begin{proposition}
\label{ConvQSigma}
Suppose that Assumptions \ref{hyp:1}-\ref{hyp:2} hold. Then, as $T\to\infty$, we have
\begin{equation}
\label{Qconverg}
\frac{1}{T}\bQ_{[0,T]}\stackrel{a.s. \text{ \rm and } L^{1} }{\longrightarrow} \bSigma, \text{ and } T\bQ_{[0,T]}^{-1}\stackrel{a.s.}{\longrightarrow} \bSigma^{-1}.
\end{equation}
\end{proposition}
\begin{remark}\label{rem:Sigma}
From \Cref{rem:a}, $\bar \bvarphi= \int_0^1 \bvarphi(s)ds = (1,0,\ldots,0)^\top$,
and set $b= \frac{\bmu^\top \bar \bvarphi}{a} = \frac{\mu_1}{a}$.
Then the first column of $\bSigma$ is
$\mathcal{C}_1 = (\bar \bvarphi^\top, -b)^\top$.
\end{remark}
As an immediate result of \Cref{ConvQSigma}, we have the following corollary.
\begin{corollary}
\label{Cbound_0114_2024}
Suppose that Assumptions \ref{hyp:1}-\ref{hyp:2} hold. Then,
\begin{equation*}
\left|\left|\frac{1}{T}\bQ_{[0,T]}-\bSigma \right|\right|=O_{\rm{P}}(T^{-1}),
\end{equation*}
as $T\to\infty$.
\end{corollary}
\begin{proof}[Proof of \Cref{Cbound_0114_2024}]
First,
  \begin{equation*}
  \frac{1}{T}\int_{0}^{T}\bvarphi(t)\bvarphi^\top(t)dt -  I_{p}=  \frac{\lfloor T\rfloor}{T}I_{p}-I_{p}+\frac{1}{T}\int_{\lfloor T\rfloor}^{T}\bvarphi(t)\bvarphi^\top(t)dt.
  \end{equation*}
  Furthermore, we get
   \begin{equation*}
T\left|\left|\frac{1}{T}\int_{0}^{T}\bvarphi(t)\bvarphi^\top(t)dt-  I_{p}\right|\right| \le T\left(\frac{\lfloor T\rfloor-T}{T} \left|\left| I_{p}\right|\right|+\frac{K_{\bvarphi}^2}{T}\right)\le T\frac{1+K_{\bvarphi}^2}{T}=1+K_{\bvarphi}^2.
  \end{equation*}
The rest of the proof is similar to the proof of \Cref{ConvQSigma}.
\end{proof}

The next result, proven in Appendix \ref{app:pf2}, shows that the MLE $\check\btheta_{T}$ is consistent.
\begin{lemma}\label{lem:mle}
Suppose that Assumptions \ref{hyp:1}-\ref{hyp:2} hold. Then, $\check\btheta_{T}$ is a  strongly consistent estimator of $\btheta$. Furthermore,
$\check\btheta_{T}$ is asymptotically normal, i.e. $T^{1/2} \left(\check\btheta_{T}-\btheta\right)\xrightarrow[T\to\infty]{{ D}}\mathcal{N}_{p+1}(0,\sigma^{2}{\bSigma}^{-1})$, and
$||\bR_{[0,T]}||={\rm O}_{\rm{P}}(T^{1/2} )$.
\end{lemma}

\section{Asymptotic behaviour of estimators for the discretized process}\label{sec:discrete}

In most applications, a continuous process is only observed at finite times. To this end, consider a partition $0=t_{0}<t_{1}<\dots<t_{N}=T$ on a given period $[0,T]$,  with $\ds \Delta_{N}= \frac{T}{N} = t_{1}-t_{0}=\cdots=t_{N}-t_{N-1}$.

For all $i\in \{1,\ldots, N\}$, set $Y_{N,i}=\frac{X_{t_{i}}-X_{t_{i-1}}}{\Delta_N^{1/2}}$,
\begin{equation}\label{eq:Z}
\bZ_{N,i}= \Delta_N^{1/2} (\varphi_{1}(t_{i-1}),\varphi_{2}(t_{i-1}),\dots,\varphi_{p}(t_{i-1}),-X_{t_{i-1}})^\top,
\end{equation}
and
\begin{equation}\label{eq:eps}
\epsilon_{N,i}=\frac{\sigma}{\Delta_N^{1/2}}(B_{t_{i}}-B_{t_{i-1}}) = \sigma \tilde \epsilon_{N,i}, \quad i\ge 1.
\end{equation}
Then, let
\begin{eqnarray}\label{eq:rn}
r_{N,i} &=& Y_{N,i} -  \bZ_{N,i}^\top \btheta - \epsilon_{N,i} \\
&=&  \frac{1}{\Delta_N^{1/2}} \int^{t_{i}}_{t_{i-1}} \left[\bmu^\top\left\{\bvarphi(s)-\bvarphi(t_{i-1})\right\}-a(X_s-X_{t_{i-1}}) \right]ds ,\nonumber
\end{eqnarray}
so
\begin{equation}\label{DiscYZ}
Y_{N,i}=\bZ_{N,i}^\top \btheta+\epsilon_{N,i} +r_{N,i}, \qquad i\ge 1.
\end{equation}
Based on the previous notations, set
\begin{equation*}
\bY_N=\begin{bmatrix}
Y_{N,1}\\
Y_{N,2}\\
\vdots\\
Y_{N,N}\\
\end{bmatrix}, \quad \bZ_N=\begin{bmatrix}
\bZ_{N,1}^\top \\
\bZ_{N,2}^\top \\
\vdots\\
\bZ_{N,N}^\top
\end{bmatrix},
\quad \btheta=\begin{bmatrix}
\mu_1\\
\mu_2\\
\vdots\\
\mu_p\\
a
\end{bmatrix},
\quad \bepsilon_N=\begin{bmatrix}
    \epsilon_{N,1}\\
    \epsilon_{N,2}\\
    \vdots\\
    \epsilon_{N,N}
\end{bmatrix},
\quad \br_N=\begin{bmatrix}
    r_{N,1}\\
    r_{N,2}\\
    \vdots\\
    r_{N,N}
\end{bmatrix}.
\end{equation*}
Then,
$ \bY_N=\bZ_N\btheta+\bepsilon_N +\br_N$ and the estimator based on the discretized regression is given $\hat\btheta_N=\left(\bZ_N^{\top}\bZ_N\right)^{-1}\bZ_N^{\top}\bY_N$.
Hence,
\begin{equation}\label{eq:thetaN}
\hat\btheta_N=\left(\bZ_N^{\top}\bZ_N\right)^{-1}\bZ_N^{\top}(\bZ_N\btheta+\bepsilon_N+ \br_N)=
\btheta+\left(\bZ_N^{\top}\bZ_N\right)^{-1}\bZ_N^{\top}(\bepsilon_N  +\br_N).
\end{equation}
In passing, let us first note a relationship between the complete sample $\{X_t:0\le t\le T\}$ and $\left\{X^{\Delta_{N}}(t_{i}):i=0,1,2,\dots\right\}$ where $X^{\Delta_{N}}(t_{i})$ satisfies the Euler-Maruyama discretized version of the SDE of $X(t)$. To this end, let $i_{t}=\max\{i=0,1,2,\dots:t_{i}\le t\}$ i.e. $i_{t}$ is the maximum of integer less than or equal to $t$, and suppose that
$X^{\Delta_{N}}(0)=X_{0}$.
From  \cite[Theorem 9.6.2, p. 324]{kloeden1999numerical},
there exists a constant
$\bf{C}^*$, such that
\begin{equation*}
\sup_{0\le t\le T}E\left[\left| X_t- X^{\Delta_{N}}(i_{t})\right|\right]\le {\bf C}^*\sqrt{\Delta_{N}+C(\Delta_N)},
\end{equation*}
where $C(\Delta_N)$ is a non-negative function with $\ds \inf_{\Delta_N > 0} C(\Delta_N) = 0$. The existence of such a function follows from Proposition \ref{prop:CN}.

\subsection{Sequential change-point detection}
We  are now in a position to define what is meant by change-point. To this end,
extend \eqref{DiscYZ}, by letting $\btheta$ depend on $i$, viz.,
\begin{equation*}\label{DiscYZ}
Y_{N,i}=\bZ_{N,i}^\top \btheta_i+\epsilon_{N,i} +r_{N,i}, \qquad i\ge 1.
\end{equation*}

First, we suppose that there is no change in the parameter $\btheta$ in the historical data of size $N$:
\begin{equation}
\label{FirstAssum}
\btheta_{i}=\btheta_0,\quad i=1,2,\cdots,N,
\end{equation}where $\btheta_0= \left(\bmu^\top,a\right)^{\top}=(\mu_{1},~~\mu_{2},\dots,~~\mu_{p},a)^{\top}$.
This can be written as
\begin{equation}
\label{eq:H0}
H_0: \btheta_{i}=\btheta_0,\qquad i\ge 1,
\end{equation}
with the alternative hypothesis
\begin{equation}
\label{eq:H1}
\begin{aligned}
 H_1:& ~\text{there is }K_* \ge 0,\text{such that }\btheta_{i}=\btheta_0,\quad i \le N+K_*, ~\text{and}\\
 & \btheta_{i}=\btheta_*,\quad i > N+K_*, \text{with }\btheta_0\neq \btheta_*,
\end{aligned}
\end{equation}
with $\btheta_*= (\bmu_*,a_*)^{\top}=(\mu_{1*},~~\mu_{2*},\dots,~~\mu_{p*},a_*)^{\top}$.

Let
$X_{t_{N+i}}^{(2)}$ denote  the value of the process $\{X_t,t\ge 0\}$ after change-point, for  $i > K_*$.
Furthermore, for $i\ge K_*+1$, set
\begin{equation}\label{eq:Z_1}
\bZ_{N,i}^{(2)}= \Delta_N^{1/2}\left(\varphi_{1}(t_{N+i-1}),\varphi_{2}(t_{N+i-1}),\dots,\varphi_{p}(t_{N+i-1}),-X_{t_{N+i-1}}^{(2)}\right)^\top,
\end{equation}
and the discretized version of the process is given by
\begin{equation}\label{DiscYZ_1}
Y_{N,i}^{(2)} =\frac{X_{t_{i}}^{(2)}-X_{t_{i-1}}^{(2)}}{\Delta_N^{1/2}}=\left(\bZ_{N,i}^{(2)}\right)^\top \btheta_*+\epsilon_{N,i}
+r_{N,i}^{(2)} , \qquad i\ge  K_*+1,
\end{equation}
where
\begin{equation}
\epsilon_{N,i}=\frac{\sigma}{\Delta_N^{1/2}}(B_{t_{i}}-B_{t_{i-1}}), \quad i \ge K_*+1,
\label{eq:eps_1}
\end{equation}
and
\begin{equation}\label{eq:rn2}
r_{N,i}^{(2)} =
 \frac{1}{\Delta_N^{1/2}} \int_{t_{i-1}}^{t_{i}} \left[\bmu^\top\left\{\bvarphi(s)-\bvarphi(t_{i-1})\right\}-a\left(X_s^{(2)}-X_{t_{i-1}}^{(2)}\right) \right]ds.
\end{equation}

\subsection{Asymptotic behaviour}

Now, set $\ds \bSigma_N= \frac{1}{T}{\bZ_N^{\top}\bZ_N} = \frac{1}{T}\sum_{i=1}^N \bZ_{N,i}\bZ_{N,i}^\top$.
Before starting the next result, we need the following proposition proven in Appendix \ref{app:pf0}.

\begin{proposition}\label{prop:CN}
 Suppose $\bvarphi$ is $1$-periodic and Riemann integrable, and define
 \begin{equation}\label{eq:CN}
C(\Delta_N) = \sup_{i\ge 1} \frac{1}{\Delta_N}\int_{t_{i-1}}^{t_{i}}\|\bvarphi(s)-\bvarphi(t_{i-1})\|ds.
 \end{equation}
 Then, $C(\Delta_N)\to 0$ as $N\to\infty$.
\end{proposition}

Before stating the next results, we need to had extra assumptions on $N$ and $T$.
\begin{assumption}\label{hyp:N}
Suppose that $N=\lfloor T^{\delta}\rfloor$, with  {$2<\delta <1/\gamma$}, and $0\leq\gamma<1/2$.
The partition $0=t_{0}<t_{1}<\dots<t_{N}=T$ on a given period $[0,T]$ has an equal increment $\Delta_{N}=T/N \to 0$, as $T\to\infty$.
\end{assumption}

The next result, proven in Appendix \ref{app:pf3}, relates $\bSigma_N$ to $\frac{1}{T}\bQ_{[0,T]}$. Let $C_3=2K_{\bvarphi}+C_1, C
_4=\frac{2}{3}K_{\bvarphi}\sigma+2\sqrt{C_2}\sigma$.

\begin{proposition}\label{prop:convdiscont}
Suppose that Assumptions \ref{hyp:1}--\ref{hyp:N} hold. Then,
\begin{equation*}
E\left[\left\|\bSigma_N-\frac{1}{T}\bQ_{[0,T]}\right\|\right]\leq C_3C(\Delta_N)+C_4\Delta_N^{1/2} +O(\Delta_N),
\end{equation*}
$$
E\left[\left\|\frac{1}{T^{1/2}}\sum_{i=1}^{N}\bZ_{N,i}\epsilon_{N,i} -\frac{\sigma}{T^{1/2}}\bR_{[0,T]}\right\|\right]
\leq \sigma \sqrt{2 K_\varphi C(\Delta_N)+  \frac{1}{2}\sigma^2 \Delta_N^{1/2}+o(\Delta_N)},
$$
and
\begin{equation}\label{eq:rN20}
E\left[\left\|\frac{1}{T^{1/2}}\sum_{i=1}^{N}\bZ_{N,i} r_{N,i}\right\|\right]
= T^{1/2}\left\{O\left(C(\Delta_N)\right)+O\left(\Delta_N^{1/2}\right)\right\} .
\end{equation}
\end{proposition}

As a by-product of \eqref{eq:thetaN}, \Cref{prop:convdiscont}, and \Cref{lem:mle}, we have the following corollary, if one assumes a little smoothness of $\bvarphi$.
\begin{assumption}\label{hyp:phiextra}
$C(\Delta)=O\left(\Delta^{1/2}\right)$.
\end{assumption}
Note that if $\bvarphi$ is Lipschitz, then $C(\Delta)=O(\Delta)$, so \Cref{hyp:phiextra} is trivially satisfied in this case.
\begin{corollary}\label{cor:est}
Under Assumptions \ref{hyp:1}--\ref{hyp:phiextra},
    $T^{1/2}\left(\check\btheta_T-\hat\btheta_N\right) \stackrel{\rm{P}}{\longrightarrow}0$, as $T\to\infty$.
\end{corollary}
This result follows from the fact that under Assumptions \ref{hyp:N}--\ref{hyp:phiextra}, the right-hand side of \eqref{eq:rN20} is $O\left(T^{-(\delta-2)/2}\right)=o(1)$, as $T\to\infty$.

Finally, we end this section with a goodness-of-fit for GOU processes.

\subsection{Goodness-of-fit test for a GOU process}\label{ssec:gof}

The goodness-of-fit test is based on the residuals defined by
\begin{equation}\label{eq:res1}
\hat{\epsilon}_{N,i}=Y_{N,i}-\bZ_{N,i}^\top\hat\btheta_N = \epsilon_{N,i}  +r_{N,i}  - \bZ_{N,i}^\top \left(\hat\btheta_N-\btheta\right),\qquad i\ge 1.
\end{equation}
For any $y\in\dR$, set $\ds \hat F_N(y) = \frac{1}{N}\sum_{i=1}^N \I\left( \hat \epsilon_{N,i}\le y\right)$, and $F(y) = \Phi(y/\sigma)$, where $\Phi$ is the cdf of the standard Gaussian, with density $\phi$.  Further, set $\dF_N(y)=N^{1/2}\left\{\hat F_N(y)-F(y)\right\}$, $y\in \dR$, and let  $ \cC_1 = \left(1,0,\ldots,0,-\mu_1/a\right)^\top$.
\begin{theorem}\label{thm:FN}
$\dF_N$ converges in $C[-\infty,+\infty]$ to $\dB\circ F+ F' \bTheta^\top \cC_1$, where $\dB$ is a Brownian bridge and $\bTheta_N=T^{1/2}\left(\btheta_N-\btheta_0\right)$ converges to $\bTheta$.
\end{theorem}
\begin{proof}
It follows from \cite{Ghoudi/Remillard:1998} or \cite{Nasri/Remillard:2019},
Assumptions \ref{hyp:1}--\ref{hyp:phiextra}, \eqref{eq:res1}, and \eqref{eq:rN20}, that
$\ds \frac{1}{N^{1/2}}\sum_{i=1}^N r_{N,i} =o_{\rm{P}}(1)$, while $\ds \frac{1}{N^{1/2}T^{1/2}}\sum_{i=1}^N \bZ_{N,i} $  is the first column of the matrix $\bSigma_N$, by Remark \ref{rem:Sigma}. As a result, $\ds \frac{1}{N^{1/2}T^{1/2}}\sum_{i=1}^N \bZ_{N,i}\stackrel{\rm{P}}{\to} \cC_1$. It follows that  for any $\bc\in \dR^{p+1}$,
$$
\frac{1}{N^{1/2}}\sum_{i=1}^N \left\{ F\left(y+ r_{N,i}+ \bc^\top \bZ_{N,i}/T^{1/2} \right) -F(y)\right\} \stackrel{\rm{P}} {\longrightarrow}F'(y)\bc ^\top \cC_1 = \frac{1}{\sigma}\phi\left(\frac{y}{\sigma}\right)\bc^\top \cC_1.
$$
\end{proof}
Next, set $\dD_N(u) = N^{1/2}\{D_N(u)-u\}$, where
$\ds D_N(u) = \frac{1}{N}\sum_{i=1}^N \I\left\{\Phi\left(\hat e_{N,i}/\hat\sigma_N\right) \le u\right\}$, $ u\in [0,1]$, and
$\ds \hat{\sigma}^2_N= \frac{1}{N}\sum_{i=1}^N Y_{N,i}^2$.
\begin{corollary}\label{cor:DN}
$\dD_n$ converges in $C[0,1]$ to $\dD$, where
\begin{equation}\label{eq:DLIM}
     \dD(u) = \beta(u) + \phi\circ \Phi^{-1}(u)\left\{\dM + \dS \Phi^{-1}(u)\right\},
 \end{equation}
  $\dM_N = N^{-1/2}\sum_{i=1}^N \tilde\epsilon_{N,i}$ converges in law to $\dM\sim N(0,1)$,  and
 $\dS_N = N^{1/2}\left(\frac{\hat\sigma_N}{\sigma}-1\right)$ converges in law to $\dS\sim N(0,1/2)$, independent of $\dM$. In addition, $E\{\dM \beta(u)\}= -\phi\circ \Phi^{-1}(u)$, and $E\{\beta(u)\dS\} = -\frac{1}{2}H(u)$, with $H(u) = \Phi^{-1}(u) \phi\circ \Phi^{-1}(u)$.
\end{corollary}
\begin{remark}
The limiting process $\dD$ is the same as the one used to test for Gaussianity with unknown mean and variance \citep{Ghoudi/Remillard:2015}.
In particular, the law of $\dD$ does not depend on any parameters.
\end{remark}
\begin{proof}
It follows from Theorem \ref{thm:FN} that $\dD_n$ converges in $C[0,1]$ to $\dD$, where
$\dD(u) = \dB(u) + \phi\circ \Phi^{-1}(u) \cC_1^\top \bTheta/\sigma + \dS \Phi^{-1}(u) \phi\circ \Phi^{-1}(u) $, and
 $\dS_N = N^{1/2}\left(\frac{\hat\sigma_N}{\sigma}-1\right)$ converges in law to $\dS$.
Now, it follows from \eqref{eq:thetaN}, Proposition \ref{prop:convdiscont} and Assumption \ref{hyp:phiextra} that $ \bTheta/\sigma $ is the limit in law
$\ds \bSigma^{-1} \frac{1}{T^{1/2}}\sum_{i=1}^N \bZ_{N,i}\tilde\epsilon_{N,i}$.
 Since $\bSigma^{-1}\cC_1 = (1,0,\ldots,0)^\top$, it follows that $\dM = \bTheta^\top\cC_1 \sigma$  is the limit in law of $\dM_N = N^{-1/2}\sum_{i=1}^N \tilde\epsilon_{N,i}$, which converges to a standard Gaussian.
Next, under Assumptions \ref{hyp:1}--\ref{hyp:phiextra}, and using Proposition \ref{prop:convdiscont}, one gets
that $\frac{1}{N^{1/2}}\sum_{i=1}^N r_{N,i}^2 = o_P(1)$, so
\begin{eqnarray*}
N^{1/2}\left(\frac{\hat\sigma_N^2}{\sigma^2}-1\right)&=&
\frac{1}{N^{1/2}}\sum_{i=1}^N \left(\frac{\epsilon_{N,i}^2}{\sigma^2}-1\right)
+ \frac{1}{\sigma^2 N^{1/2}}\sum_{i=1}^N r_{N,i}^2 + \frac{T}{\sigma N^{1/2}}\btheta^\top \bSigma_N\btheta \\
&&\quad +2\btheta^\top \frac{1}{\sigma^2 N^{1/2}} \sum_{i=1}^N \left\{\bZ_{N,i}(\epsilon_{N,i}+r_{N,i})+ \epsilon_{N,i} r_{N,i} \right\}\\
&=& \frac{1}{N^{1/2}}\sum_{i=1}^N \left(\tilde \epsilon_{N,i}^2 -1\right)+o_{\rm{P}}(1).
\end{eqnarray*}
As a result, $\dS_N = \left(\frac{\hat\sigma_N^2}{\sigma^2}-1\right)\Big{/}\left(1+\frac{\hat\sigma_N}{\sigma}\right) = \frac{1}{2} \frac{1}{N^{1/2}}\sum_{i=1}^N \left(\tilde \epsilon_{N,i}^2 -1\right)+o_P(1) $ converges in law to
$\dS \sim N(0,1/2)$, with $E\{\beta(u)\dS\} = -\frac{1}{2}H(u)$, with $H(u) = \Phi^{-1}(u) \phi\circ \Phi^{-1}(u)$.
\end{proof}
Since the law of $\dD$ does not depend on any parameters, any continuous functional of $\dD_N$ like Kolmogorov-Smirnov or Cram\'er-von Mises can be easily tabulated. First, note that if $u_{N:1}< \cdots < u_{N:N}$ are the order statistics of $u_{N,i} = \Phi\left(\frac{\hat e_{N,i}}{\hat\sigma_N}\right)$, $i\in \{1,\ldots, N\}$, it follows that the Kolmogorov-Smirnov statistic is
$$
\dT_N = \sup_{u\in [0,1]}|\dD_N(u)| = N^{1/2} \max_{1\le i\le N}\max(|u_{N:i}-i/N|,|u_{N:i}-(i-1)/N|,
$$
while the Cram\'er-von Mises statistic is
$$
\dV_N = \int_0^1 \{\dD_N(u)\}^2 du = \frac{1}{12N}+\sum_{i=1}^N \left\{u_{N:i} - \frac{(i-1/2)}{N}\right\}^2.
$$
To estimate the quantiles of $\dT_N$ and $\dV_N$, one can proceed the following way: For $k\in\{1,\ldots, B\}$, generate $\varepsilon_i\sim N(0,1)$, $i\in \{1,\ldots,N\}$,
$u_{N,i} = \Phi\left(\frac{\varepsilon_i-m_N}{s_N}\right)$, where $m_N = \frac{1}{N}\sum_{i=1}^N \varepsilon_i$, and $s_N^2 = \frac{1}{N-1}\sum_{i=1}^N \varepsilon_i^2$. Then, compute $\dT_N^{(k)}$ and $\dV_N^{(k)}$, $k\in \{1,\ldots,B\}$ and calculate the quantiles. Examples of these calculations are given in \Cref{tab:gof}. Note that the limiting distribution of $\dT$ of $\dT_N$ is the Lilliefors statistic \citep{Lilliefors:1967}.
\begin{table}[ht!]
\caption{Estimated quantiles of size $\alpha\in\{90\%,95\%,97.5\%,99\%\}$ for Kolmogorov-Smirnov (ks) and Cram\'er-von Mises (cvm) based on $B=500,000$ replications for sample size $n\in\{100,250,500,750,1000\}$. }\label{tab:gof}
\centering
\begin{tabular}{ccccccccccc}
\hline
&        \multicolumn{10}{c}{$n$} \\
$a$ & \multicolumn{2}{c}{$100$} & \multicolumn{2}{c}{$250$} &\multicolumn{2}{c}{$500$} &\multicolumn{2}{c}{$750$} &\multicolumn{2}{c}{$1000$} \\
&      ks      & cvm   &   ks  & cvm   &   ks  & cvm   &   ks  & cvm   &   ks  & cvm \\
\hline
90\%   & 0.817 & 0.103  & 0.825  &0.103 & 0.828   & 0.103   & 0.829 & 0.103   & 0.830 &  0.103     \\
95\%   & 0.890 & 0.125  & 0.898  &0.126 & 0.901   & 0.126   & 0.902 & 0.126   & 0.904 &  0.126     \\
97.5\% & 0.956 & 0.148  & 0.966  &0.148 & 0.967   & 0.149   & 0.970 & 0.149   & 0.972 &  0.149     \\
99\%   & 1.037 & 0.179  & 1.047  &0.179 & 1.049   & 0.178   & 1.052 & 0.179   & 1.053 &  0.179     \\
\hline
\end{tabular}
\end{table}
It follows from  Table \ref{tab:gof} that the Cram\'er-von Mises statistic seems to converge faster than the Kolmogorov-Smirnov statistic.  For example, for any sample size, the 95\% quantile for the
Cram\'er-von Mises statistic is $0.149$, while it is $0.904$ the Kolmogorov-Smirnov statistic.
The speed of convergence might depend on the number of estimated parameters, even if the limit is parameter free. This will be studied in a forthcoming paper.

\section{Asymptotic property of the test statistics}\label{sec:Asymptotic}

In what follows, we define two processes to detect change-point.  In this section, we will develop the test statistics for detecting structural changes and examine their asymptotic distributions under both the null and alternative.

\subsection{Change-point detection based on residuals}\label{sec:DetRes}
Let $\hat{Q}(N,K)$ be the stochastic process
\begin{equation}
\label{statistic1}
\hat{Q}(N,K)= \sum_{N<i\le N+K}\hat{\epsilon}_{N,i}, \qquad K\ge 1,
\end{equation}
where
\begin{equation}
    \label{eq:epsilon}
    \hat\epsilon_{N,i} =Y_{N,i}-\bZ_{N,i}^\top \hat\btheta_N= \epsilon_{N,i} +r_{N,i}- \bZ_{N,i} \left(\bZ_N^{\top}\bZ_N\right)^{-1}\bZ_N^{\top}(\bepsilon_N+\br_N).
\end{equation}
The boundary function is defined  as $\ds g(N,K)=c \sigma g_1(N,K)$, $c = c(\alpha) > 0$,
where $g_1(N,K)=N^{1/2}\left(1+\frac{K}{N}\right)\ell_\gamma\left(\frac{K}{N}\right)$, where $\ell_\gamma(s)=\left(\frac{s}{1+s}\right)^{\gamma}$, with $0\le \gamma<1/2$,
 and $\alpha$ is the type-I error defined in \eqref{FalseAlarm}. The associated stopping time is then given by
\begin{equation}
\label{location1}
\tau_{\hat{Q}}(N) = \inf\{K \geq 1 : |\hat{Q}(N,K)| \geq g(N,K)\}.
\end{equation}
The proof of the next result, given in Appendix \ref{app:pf-thm1}, is similar to the proof of Theorem 2.1 in \cite{horvath2004monitoring}, although there are extra terms
that need to be taken care of. These extra terms appear because we work with the process observed at discrete times $t_{i-1} = i\Delta_N$, $i\ge 0$.
\begin{theorem}
\label{thm:main1}
Suppose that Assumptions \ref{hyp:1}-\ref{hyp:N} hold. Then, under the null hypothesis \eqref{eq:H0}, we have
\begin{multline*}
\lim_{N\to\infty}\dP\left(\tau_{\hat{Q}}(N)<\infty\right)=\lim_{N \rightarrow \infty}
\dP\left\{\sup_{1 \le K < \infty}\frac{|\hat{Q}(N,K)|}{\sigma g_1(N,K)} \ge c \right\}
=\dP \left\{ \sup_{0 < t \le 1} \frac{|\dB_1(t)|}{t^\gamma} \ge c\right\},
\end{multline*}
where $\dB_1(\cdot)$ is a Brownian motion on $[0,1]$. 
\end{theorem}
\begin{remark}
Suppose that there is no change in the
historical data and the natural estimator of $\sigma^2$  is the so-called realized quadratic variation of the process given by
$$
\hat{\sigma}^2_N= \frac{1}{N}\sum_{i=1}^N Y_{N,i}^2 = \frac{1}{N}\sum_{i=1}^{N}\frac{(X_{t_{i}}-X_{t_{i-1}})^2}{t_{i}-t_{i-1}}.
$$
It follows from \Cref{cor:DN} that $\hat{\sigma}^2_N$ converges in probability to $\sigma^2$, so \Cref{thm:main1} remains true when $\sigma$ is replaced by
$\hat{\sigma}_N$.
\end{remark}
The next assumption is needed to prove the result under the alternative \eqref{eq:H1}, namely that there exists $K_* \ge 0$
such that
$$
\btheta_{i}=\btheta_0,\quad 1\le i\le N+K_*, \text{ and }
 \btheta_{i}=\btheta_*,\quad i > N+K_*, \text{with }\btheta_0\neq \btheta_*=(\mu_*^\top, a_*)^\top.
$$
\begin{assumption}
\label{hyp:theta_star}
Suppose that
$K_*/N\to t_* \ge 0 $ and
\begin{equation}\label{eq:condalt}
\kappa_{\hat Q,\btheta_0,\btheta_*}
= \mathbf{e}_1^\top \left(\frac{a}{a_*}\bmu_*-\bmu\right)  = a\left(\frac{\mu_{1*}}{a_*}-\frac{\mu_{1}}{a}\right)\neq 0,
\end{equation}
\end{assumption}
The next result, showing the behavior of the statistic under the alternative hypothesis, is proven in Appendix \ref{app:pf-thm1-alt}.
\begin{theorem}
\label{ThmAlt}
Suppose that Assumptions \ref{hyp:1}-\ref{hyp:theta_star} hold. Then, under the alternative hypothesis \eqref{eq:H1}, as $N\to\infty$,
\begin{equation*}
\sup_{1 \le K < \infty}\frac{\left|\hat{Q}(N,K)\right|}{g_1(N,K)}\stackrel{\rm{P}}{\longrightarrow}\infty.
\end{equation*}
In particular, $\ds \lim_{N \to \infty} \dP \left\{ \tau_{\hat{Q}}(N) < \infty \right\} = 1$.
\end{theorem}

\begin{remark}\label{rem:kappaQ}
Note that one can have $\btheta_*\neq \btheta_0$ and $\kappa_{\hat Q,\btheta_0,\btheta_*}=0$, showing that the test statistic is not consistent under all alternative hypotheses. This is illustrated by a numerical experiment in Section \ref{sec:Simulation}.
\end{remark}

\subsection{Sequential detection using estimators}\label{sec:DetEst}
One can also detect change-point using a statistic based on estimations at different times. To this end,
let $\hat{\bGamma}(N,K)$ be the process defined by
\begin{equation}\
\label{statistic3}
\hat{\bGamma}(N,K)=T^{1/2} \left(\frac{1}{T}\bZ_N^\top\bZ_N\right)^{1/2}\left(\hat\btheta_{N+K}-\hat\btheta_N\right).
\end{equation}
Further set
\begin{equation}
\label{location2}
\tau_{\hat{\bGamma}}(N) = \inf\left\{ K\ge 1 : \left\|\hat{\bGamma}(N,K)\right\| \geq c {\ell\left(K/N\right)}\right\},
\end{equation}
where $\ell_\gamma(s) =\left(\frac{s}{1+s}\right)^{\gamma}$, with $0\le \gamma<1/2$. The proof of the next result is given in Appendix \ref{app:pf-thm2}.

\begin{theorem}
\label{thm:main12}
Suppose that Assumptions \ref{hyp:1}-\ref{hyp:phiextra} hold. Then, under the null hypothesis \eqref{eq:H0},
\begin{eqnarray*}
\lim_{N \to \infty} \dP\left\{\tau_{\hat{\bGamma}}(N) < \infty\right\} &=&
\lim_{N \rightarrow \infty}
\dP\left\{ \sup_{K\ge 1}\left\|\hat{\bGamma}(N,K)\right\|/{\ell\left(K/N\right)}\ge c \right\}\\
&=&\dP\left\{ \sup_{0 < t \le 1} \frac{\left\|\dB_{p+1}(t)\right\|}{t^\gamma}\ge c \right\},
\end{eqnarray*}
where $\dB_{p+1}(\cdot)$ is a $(p+1)$-dimensional Brownian motion.
\end{theorem}

Finally, we can state the result of the behavior of the statistic $\hat \bGamma(N,K)/h(K/N)$ under the alternative hypothesis. Its proof is given in Appendix \ref{app:pf-thm2-alt}.
\begin{theorem}
\label{ThmAlt2}
Suppose that \Cref{hyp:1}-\ref{hyp:phiextra} hold. Then, under the alternative hypothesis \eqref{eq:H1}, as $N\to\infty$, we have
\begin{equation*}
 \sup_{K\ge 1}\left\{\left\|\hat{\bGamma}(N,K)\right\|/\ell\left(K/N\right)\right\}\stackrel{\rm{P}}{\longrightarrow}\infty.
\end{equation*}
In particular, $\ds \lim_{N \to \infty} \dP \left\{ \tau_{\hat{\bGamma}}(N) < \infty \right\} = 1$.
\end{theorem}

\section{Numerical experiments }\label{sec:Simulation}
In this section, we present the results of several numerical experiments conducted to assess the finite sample performance of the monitoring strategies discussed earlier.
It is worth noting that while our monitoring schemes are designed to detect sudden changes in the parameters of the underlying linear model, this does not necessarily translate in an abrupt alteration in the expectation of the observed data itself.

First, we present some tables giving critical values for the limiting behavior of our statistics. \Cref{tab-critical1}--\Cref{tab-critical5} give respectively the approximate critical values $c(\alpha,\gamma)$ associated with $k$-dimensional Brownian motions, i.e., for which
\begin{equation*}
\rm{P} \left\{ \sup_{0 < t \le 1} \frac{|\dB_k(t)|}{t^\gamma} \ge c(\alpha,\gamma)\right\}=\alpha.
\end{equation*}

\begin{table}[ht!]
\small
\setlength{\abovecaptionskip}{0cm}
\setlength{\belowcaptionskip}{0.25cm}
\caption{Critical values $c(\alpha)$  of $\ds \sup_{0 < t \le 1}\frac{|\dB_1(t)|}{t^\gamma}$, based on 50,000 replications. The Brownian motion is evaluated at 10,000 equidistant points in $[0,1]$.}
\label{tab-critical1}
\centering
\begin{tabular}{ccccccc}
\hline
\multirow{2}*{$\alpha$}&&&$\gamma$&&&\\
  &0&0.1&0.2&0.3&0.4&0.49\\
 \hline
0.1   & 1.9520 & 2.0082 & 2.0703 & 2.1619 & 2.3527 & 2.8296 \\
0.05  & 2.2280 & 2.2933 & 2.3307 & 2.4295 & 2.6056 & 3.0738 \\
0.025 & 2.4947 & 2.5440 & 2.5784 & 2.6687 & 2.8388 & 3.3109 \\
0.01  & 2.8074 & 2.8545 & 2.8833 & 2.9547 & 3.1131 & 3.5775 \\
 \hline
  \end{tabular}
\end{table}

\begin{table}[htbp]
\small
\setlength{\abovecaptionskip}{0cm}
\setlength{\belowcaptionskip}{0.25cm}
\caption{Critical values $c(\alpha)$ of $\ds \sup_{0 < t\le 1} \frac{\left\|\dB_{2}(t)\right\|}{t^\gamma}$, based on 50,000 replications.
The Brownian motion is evaluated at 10,000 equidistant points in $[0,1]$. }
\label{tab-critical2}
\centering
\begin{tabular}{ccccccc}
\hline
\multirow{2}*{$\alpha$}&&&$\gamma$&&&\\
  &0&0.1&0.2&0.3&0.4&0.49\\
 \hline
 0.100   &  2.4165  & 2.4543  & 2.5095  & 2.6087  & 2.7839 &  3.2875     \\
 0.050   &  2.6944  & 2.7231  & 2.7740  & 2.8655  & 3.0354 &  3.5269     \\
 0.025   &  2.9533  & 2.9539  & 3.0157  & 3.0922  & 3.2566 &  3.7328    \\
 0.010   &  3.2625  & 3.2541  & 3.3063  & 3.3661  & 3.5367 &  3.9957   \\
 \hline
  \end{tabular}
\end{table}


\begin{table}[htbp]
\small
\setlength{\abovecaptionskip}{0cm}
\setlength{\belowcaptionskip}{0.25cm}
\caption{Critical values $c(\alpha)$ of $\ds \sup_{0 < t\le 1} \frac{\left\|\dB_{3}(t)\right\|}{t^\gamma}$, based on 50,000 replications.
The Brownian motion is evaluated at 10,000 equidistant points in $[0,1]$. }
\label{tab-critical3}
\centering
\begin{tabular}{ccccccc}
\hline
\multirow{2}*{$\alpha$}&&&$\gamma$&&&\\
  &0&0.1&0.2&0.3&0.4&0.49\\
 \hline
0.100 &  2.7472 &   2.7820  &  2.8379  &  2.9212  &  3.1071  &  3.6085    \\
0.050 &  3.0189 &   3.0502  &  3.1019  &  3.1763  &  3.3522  &  3.8305    \\
0.025 &  3.2640 &   3.2890  &  3.3474  &  3.4233  &  3.5744  &  4.0285   \\
0.010 &  3.5698 &   3.5595  &  3.6272  &  3.7057  &  3.8423  &  4.2816  \\
 \hline
  \end{tabular}
\end{table}


\begin{table}[htbp]
\small
\setlength{\abovecaptionskip}{0cm}
\setlength{\belowcaptionskip}{0.25cm}
\caption{Critical values $c(\alpha)$ of $\ds \sup_{0 < t\le 1} \frac{\left\|\dB_{4}(t)\right\|}{t^\gamma}$, based on 50,000 replications.
The Brownian motion is evaluated at 10,000 equidistant points in $[0,1]$. }
\label{tab-critical4}
\centering
\begin{tabular}{ccccccc}
\hline
\multirow{2}*{$\alpha$}&&&$\gamma$&&&\\
  &0&0.1&0.2&0.3&0.4&0.49\\
 \hline
 0.100 & 3.0243  & 3.0623  & 3.1147  & 3.1955  & 3.3683  & 3.8794  \\
 0.050 & 3.3126  & 3.3318  & 3.3768  & 3.4517  & 3.6109  & 4.1133  \\
 0.025 & 3.5516  & 3.5734  & 3.6188  & 3.6838  & 3.8354  & 4.3205  \\
 0.010 & 3.8403  & 3.8594  & 3.9058  & 3.9691  & 4.1084  & 4.5699  \\
 \hline
  \end{tabular}
\end{table}

\begin{table}[htbp]
\small
\setlength{\abovecaptionskip}{0cm}
\setlength{\belowcaptionskip}{0.25cm}
\caption{Critical values $c(\alpha)$ of $\ds \sup_{0 < t\le 1} \frac{\left\|\dB_{5}(t)\right\|}{t^\gamma}$, based on 50,000 replications.
The Brownian motion is evaluated at  10,000 equidistant points in $[0,1]$. }
\label{tab-critical5}
\centering
\begin{tabular}{ccccccc}
\hline
\multirow{2}*{$\alpha$}&&&$\gamma$&&&\\
  &0&0.1&0.2&0.3&0.4&0.49\\
 \hline
 0.100 & 3.2594  & 3.2885  & 3.3424  & 3.4308  & 3.6075  & 4.1203   \\
 0.050 & 3.5229  & 3.5625  & 3.6014  & 3.6854  & 3.8458  & 4.3372   \\
 0.025 & 3.7643  & 3.7948  & 3.8380  & 3.9314  & 4.0702  & 4.5398   \\
 0.010 & 4.0470  & 4.0763  & 4.1232  & 4.2085  & 4.3348  & 4.7933   \\
 \hline
  \end{tabular}
\end{table}

\subsection{Description of the numerical experiments}\label{ssec:exp}

For the numerical experiments, we use Monte-Carlo simulations along with exact solutions of the SDE, as described in \Cref{app:sim}.
First, we assume that the generalized O-U process solves the following SDEs:
\begin{eqnarray}  \label{CPsimu1}
 dX(t)&=&\left( \mu_{1}  + \mu_{2} \sqrt{2}\cos(2\pi t) -a X(t)  \right)dt+\sigma dW_{t},\quad 0\le t\le T(1+t_*),\\
dX(t)&=&\left( \mu_{1*}  + \mu_{2*} \sqrt{2}\cos(2\pi t) -a_* X(t)  \right)dt+\sigma dW_{t}, \quad t > T(1+t_*).
 \label{CPsimu2}
\end{eqnarray}
Here, the change point occurs after $T(1+t_*)$. Hence, if $\frac{K_*}{N}\to t_*$, as in Assumption \ref{hyp:theta_star}, then $t_{N+K_*} = (N+K_*)\frac{T}{N} \to T(1+t_*)$. For the numerical experiments, we chose $t_* \in \{0,0.1,0.3,0.5,0.7\}$, and we took
$\btheta_0=(\mu_{1},\mu_{2},a)^\top= (1,2,1)$, and $\btheta_*^\top =(\mu_{1*},\mu_{2*},a_*)$ belongs to the following set $\{(2,4,2), (5,3,4), (3,3,1),  (15,3,4), (5,3,1) \}$.
We also take $\sigma = \BB{3}$.  From the computations in Appendix \ref{app:tildeh}, one obtains
\begin{equation*}\label{eq:Sigmaexnum}
\bSigma = \left( \begin{array}{ccc}
  1  & 0 & -1\\
 0  & 1 &  -0.04940905\\
 -1  & -0.04940905 & 5.59881809
\end{array}\right).
\end{equation*}
Note that since $\mu_1=a=1$,
\begin{equation}\label{eq:condalt-sim}
\kappa_{\hat Q,\btheta_0,\btheta_*}  
= {\bar \bvarphi}^\top \left(\frac{a}{a_*}\bmu_*-\bmu\right) = a\left(\frac{\mu_{1*}}{a_*}- \frac{\mu_{1}}{a}\right)
= \frac{\mu_{1*}}{a_*}-1,
\end{equation}
 and by Equation \eqref{eq:newkappa},
 \begin{equation}
     \kappa_{\hat \bGamma,\btheta_0,\btheta_*} = \left\|\bSigma (\btheta_*-\btheta_0)\right\| .
 \end{equation}
 The associated values of $\kappa_{\hat Q,\btheta_0,\btheta_*}$ and $\kappa_{\hat \bGamma,\btheta_0,\btheta_*}$ are given in \Cref{tab:kappa}.
\begin{table}[ht!]
\caption{Values of $\kappa_{\hat Q,\btheta_0,\btheta_*}$ and $\kappa_{\hat \bGamma,\btheta_0,\btheta_*}$ for $\btheta_0=(1,2,1)$, and
$\btheta_* \in \{(2,4,2), (5,3,4), (3,3,1),  (15,3,4), (5,3,1) \}$.}
\label{tab:kappa}
\centering
\begin{tabular}{cccccc}
\hline
$\btheta_*$                       & (2,4,2) & (5,3,4) & (3,3,1) & (15,3,4) & (5,3,1) \\
$\kappa_{\hat Q,\btheta_0,\btheta_*}$  & 0       & 0.25  & 2.0  & 2.75 & 4.0 \\
$\kappa_{\hat \bGamma,\btheta_0,\btheta_*}$         & 4.90    & 12.81  & 3.03 & 11.37 & 5.78\\
\hline
\end{tabular}
\end{table}

Using the formulas in \Cref{app:sim}, we simulated 100 trajectories of $X(t_{i-1})$, $t_{i-1} = i\Delta_N = i \frac{T}{N}$,  $i \in \{0,1,\ldots, 3N\}$, where we took $T=20$, $N\in \{500, 750, 1000, 1500, 2000\}$. For the test statistics, we chose $\gamma=0.1$. Then, defining $\delta$ by $N=T^\delta$,  we see that \Cref{hyp:N} is met, namely $2<\delta<1/\gamma=10$, since $\delta\in\{2.074, 2.210, 2.306,  2.441, 2.537\}$.  We set $\tau_{\hat Q}(N)=2N$ when there is no change-point detected after; we did the same for $\tau_{\hat \bGamma}(N)$. Then, we computed the boxplots for the 100
estimated change-points $\tau_{\hat Q}(N)/N-1$ and $\tau_{\hat \bGamma}(N)/N-1$.

\begin{figure*}[ht!]
    \centering
    \includegraphics[width=6cm,height=3cm]{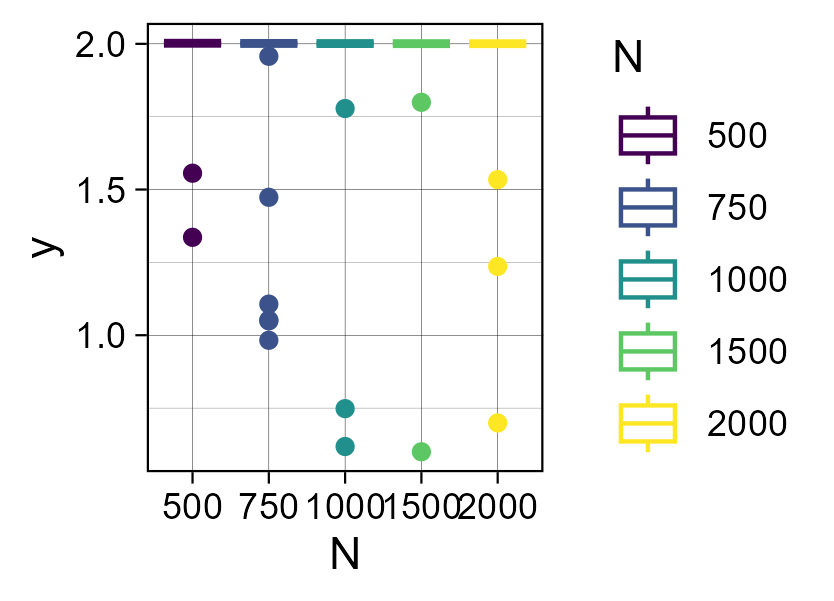}
    \includegraphics[width=6cm,height=3cm]{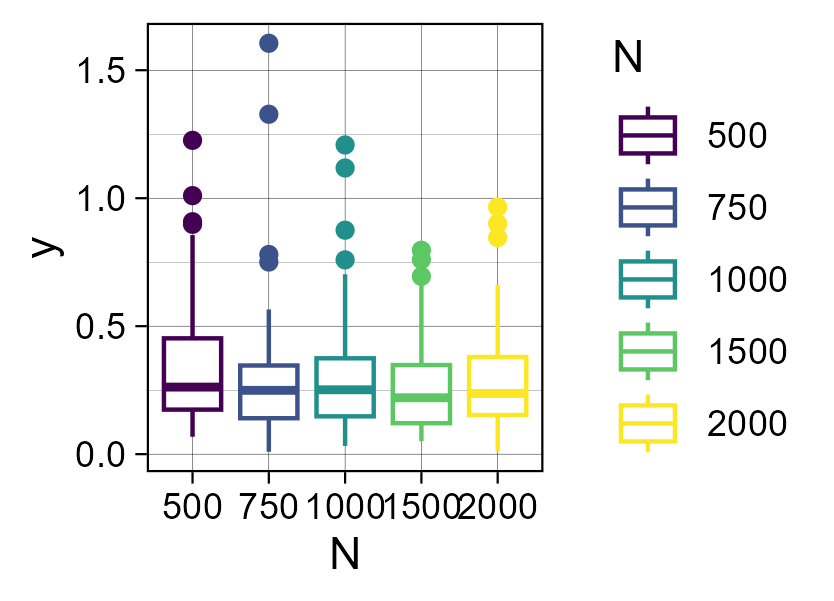}

   \includegraphics[width=6cm,height=3cm]{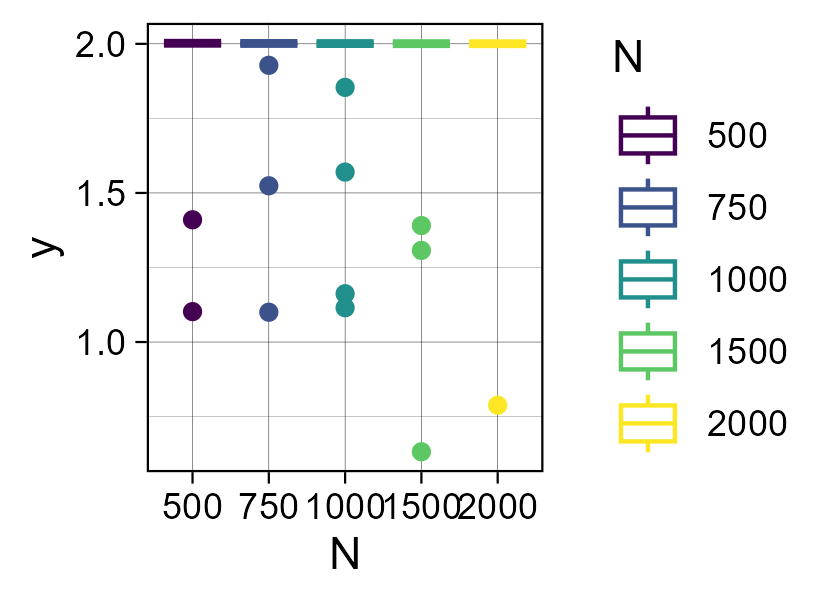}
    \includegraphics[width=6cm,height=3cm]{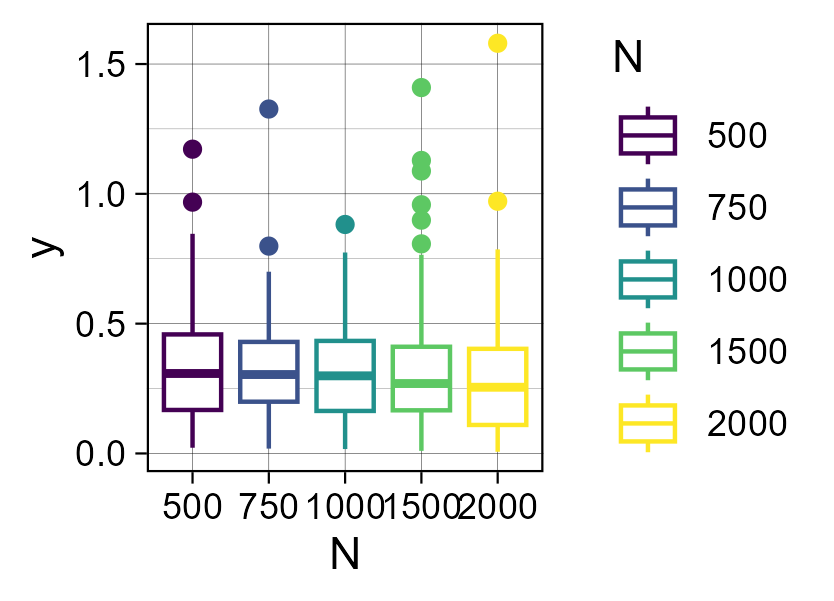}

      \includegraphics[width=6cm,height=3cm]{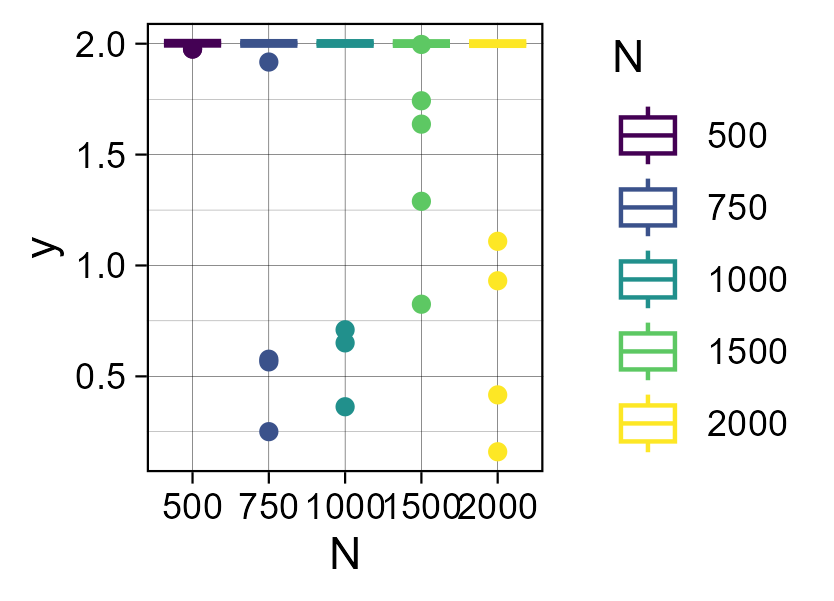}
    \includegraphics[width=6cm,height=3cm]{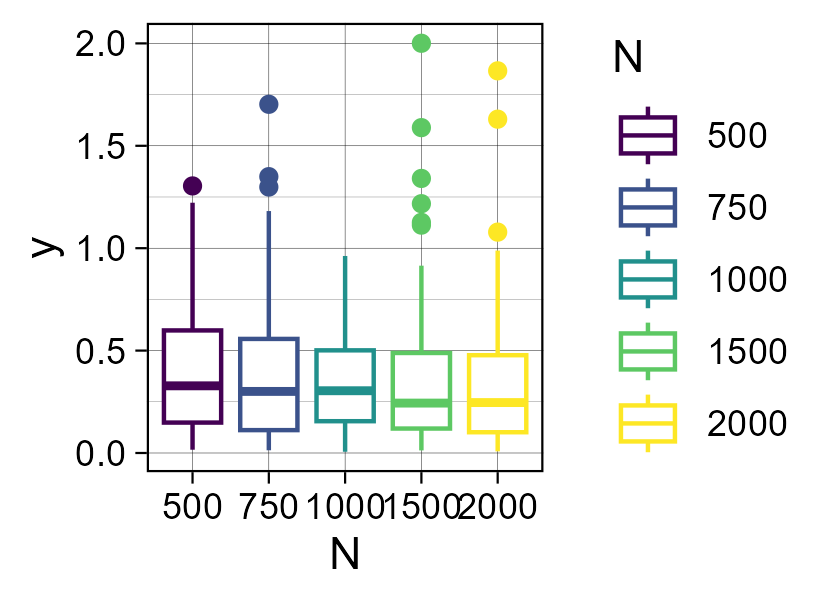}

     \includegraphics[width=6cm,height=3cm]{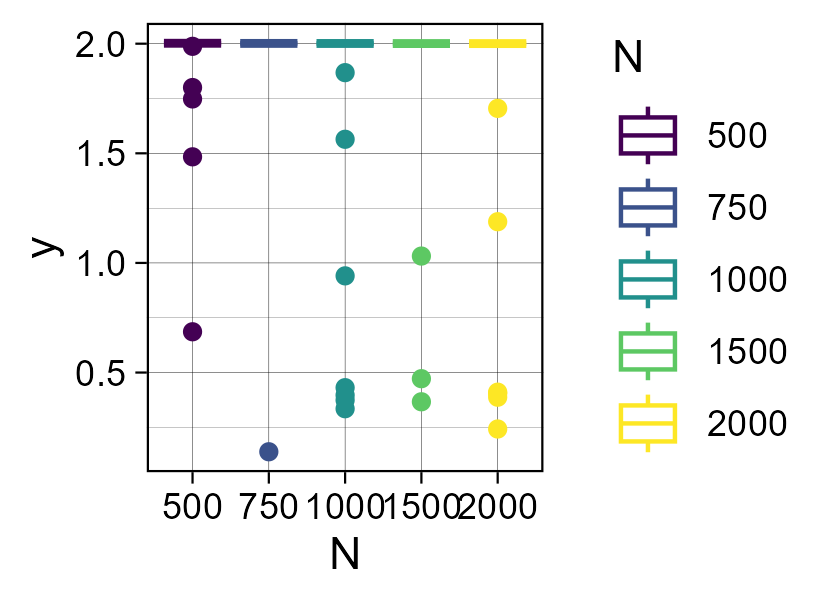}
    \includegraphics[width=6cm,height=3cm]{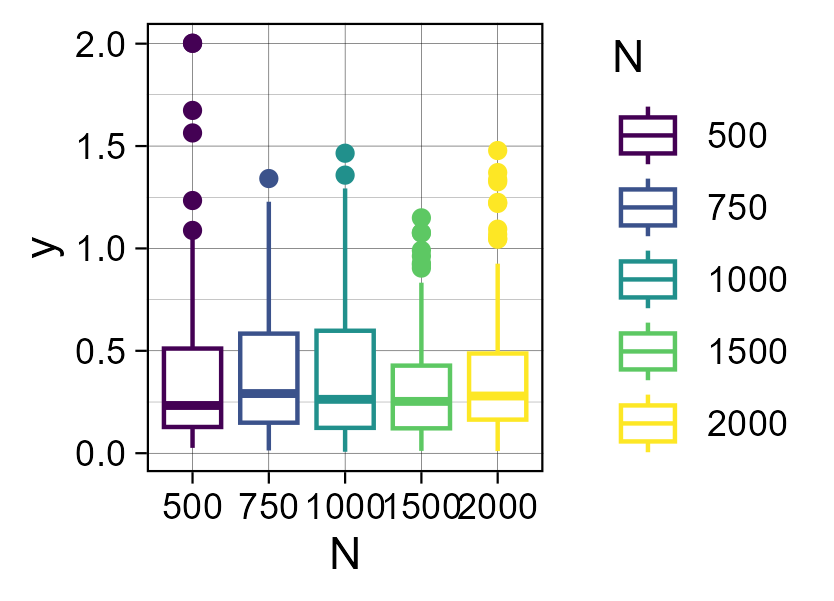}

    \includegraphics[width=6cm,height=3cm]{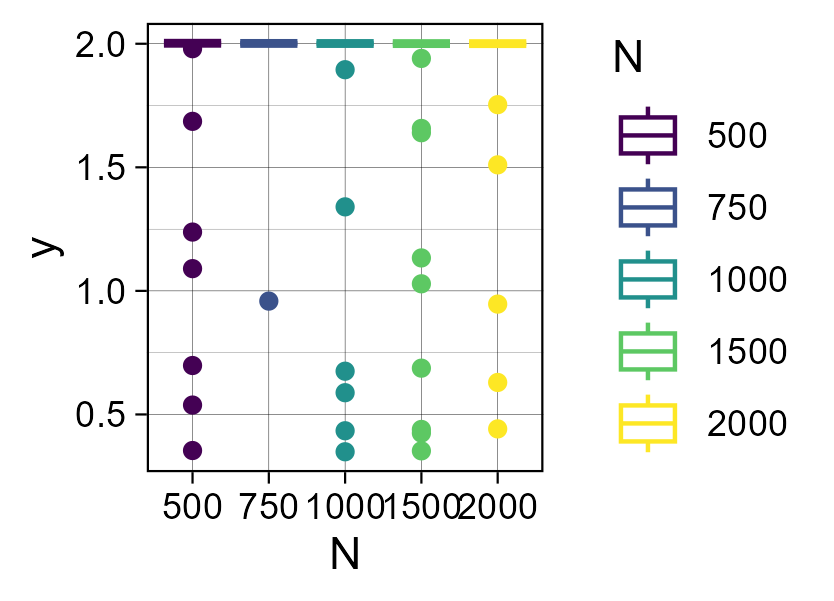}
    \includegraphics[width=6cm,height=3cm]{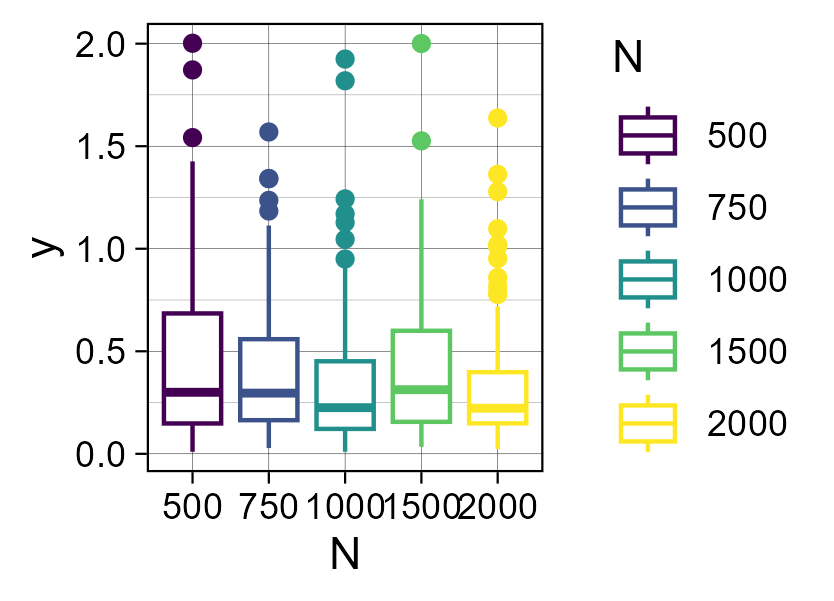}
    \caption{Estimation of $t_*$ when $t_*\in \{0,0.1,0.3,0.5,0.7\}$ (top to bottom),  $\btheta_*=2\btheta_0=(2,4,2)$,  $\kappa_{\hat Q,\btheta_0,\btheta_*}=0$, and $\kappa_{\hat \bGamma,\btheta_0,\btheta_*}=4.90$, based on $\widehat Q$ (left panel) and
     $\widehat \bGamma$ (right panel).}\label{fig:tauStar_t7_242}
\end{figure*}

\begin{figure}[ht!]
    \centering
   \includegraphics[width=6cm,height=3cm]{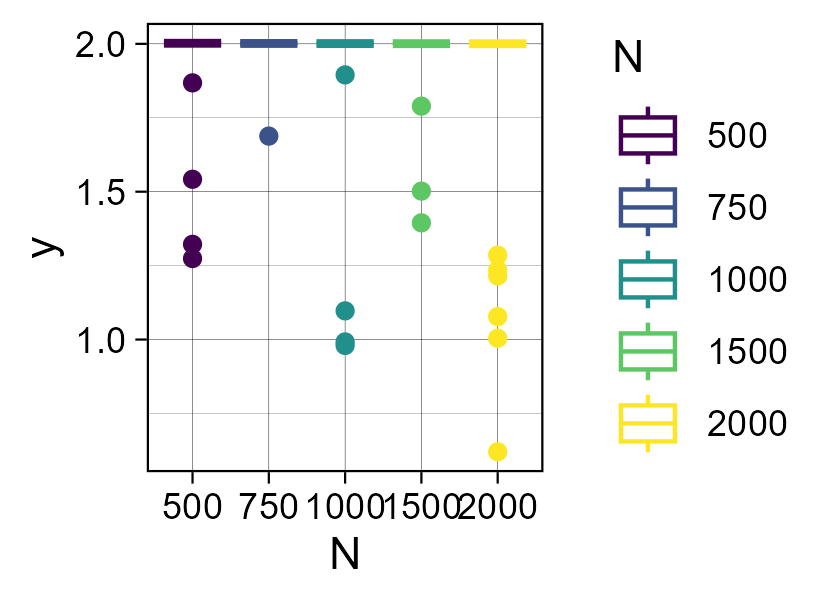}
    \includegraphics[width=6cm,height=3cm]{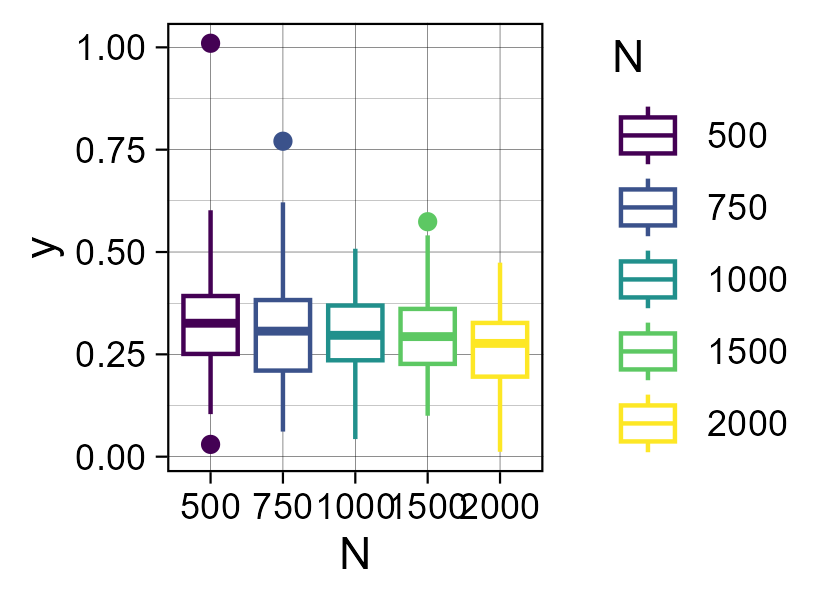}

 \includegraphics[width=6cm,height=3cm]{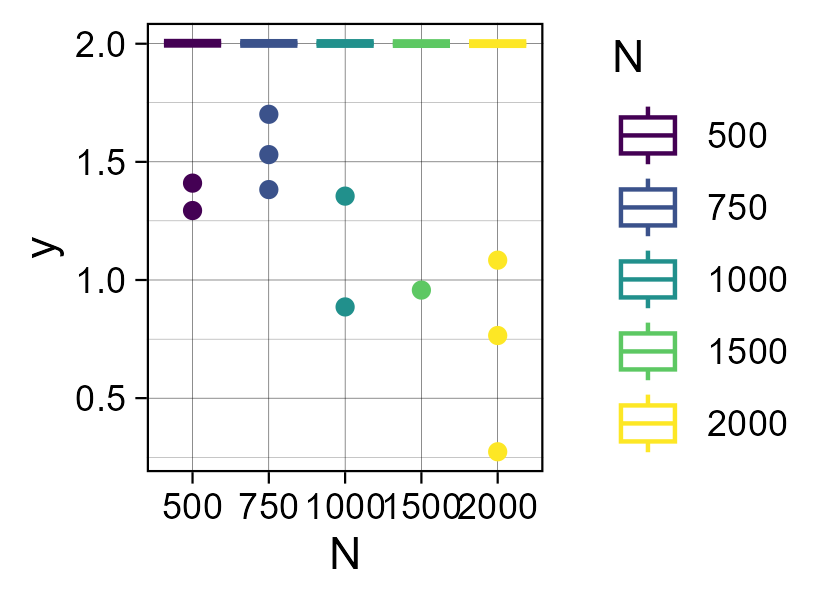}
    \includegraphics[width=6cm,height=3cm]{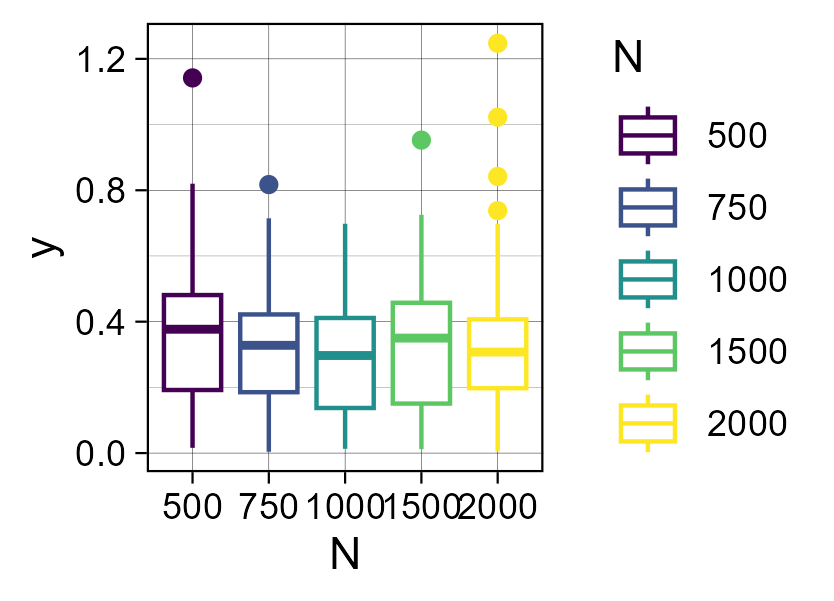}

\includegraphics[width=6cm,height=3cm]{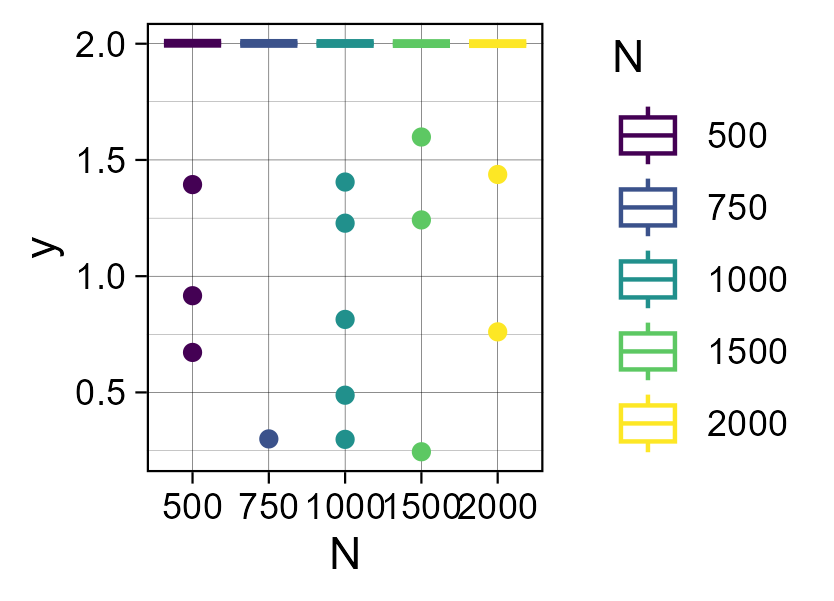}
    \includegraphics[width=6cm,height=3cm]{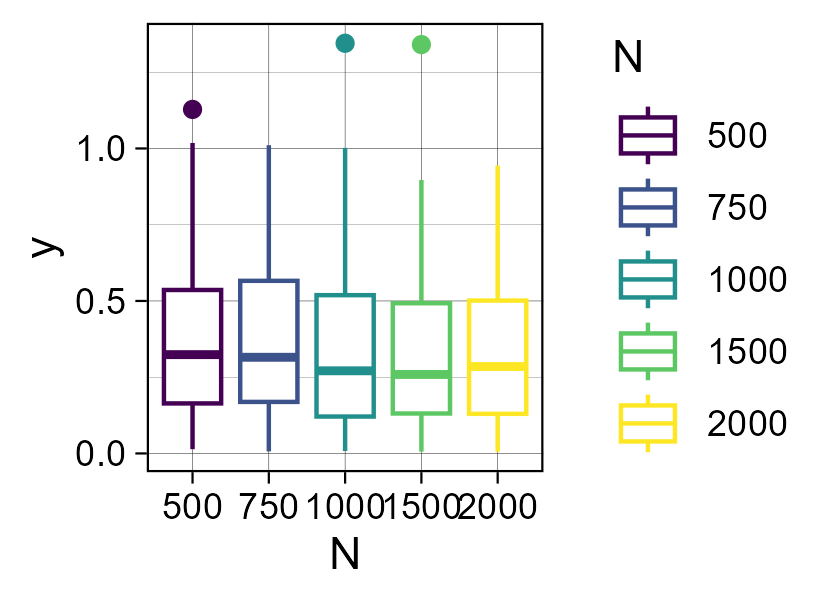}

\includegraphics[width=6cm,height=3cm]{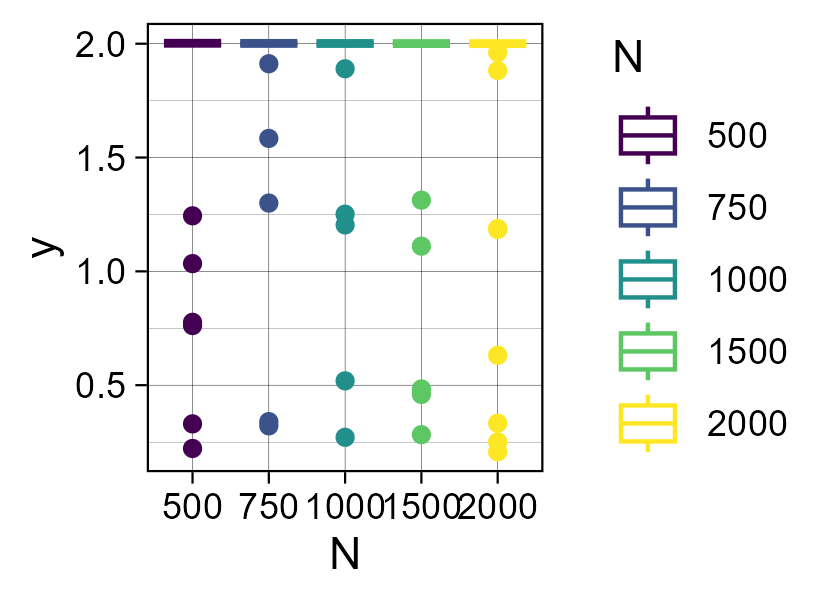}
    \includegraphics[width=6cm,height=3cm]{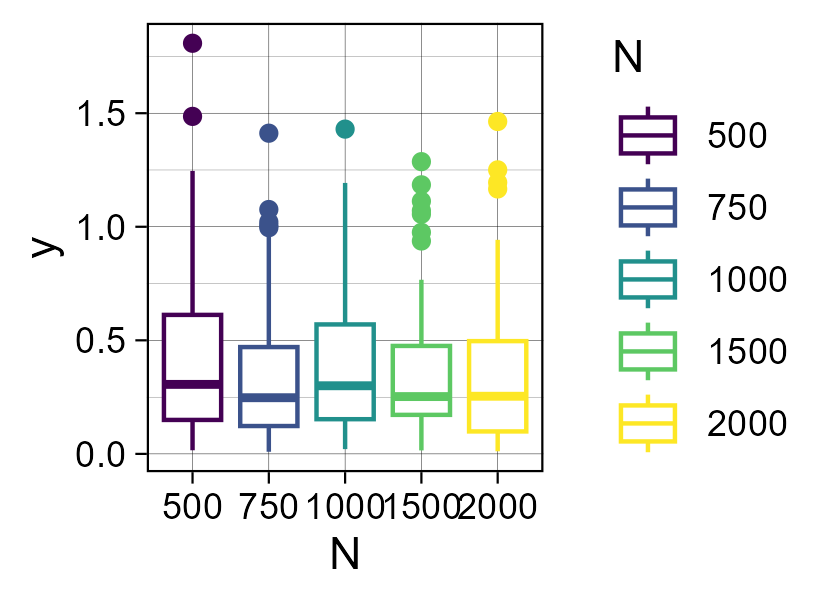}

 \includegraphics[width=6cm,height=3cm]{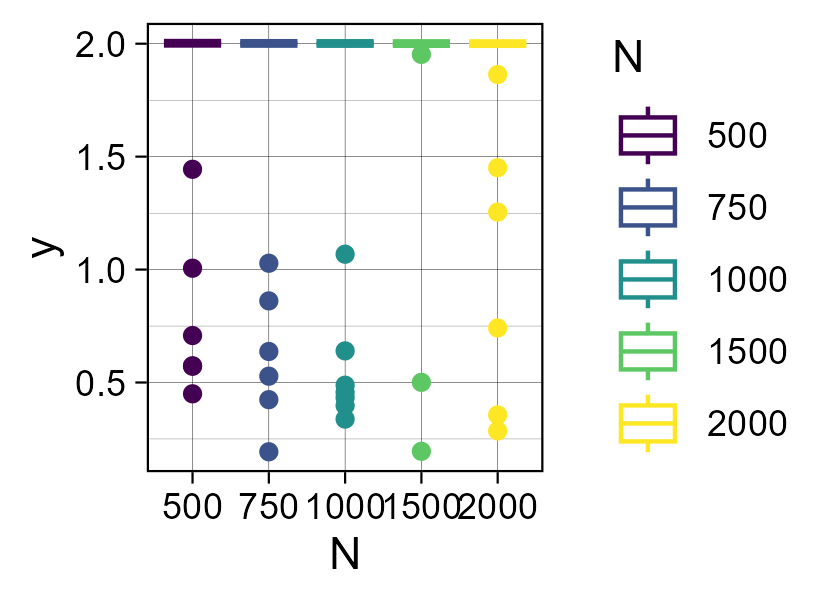}
    \includegraphics[width=6cm,height=3cm]{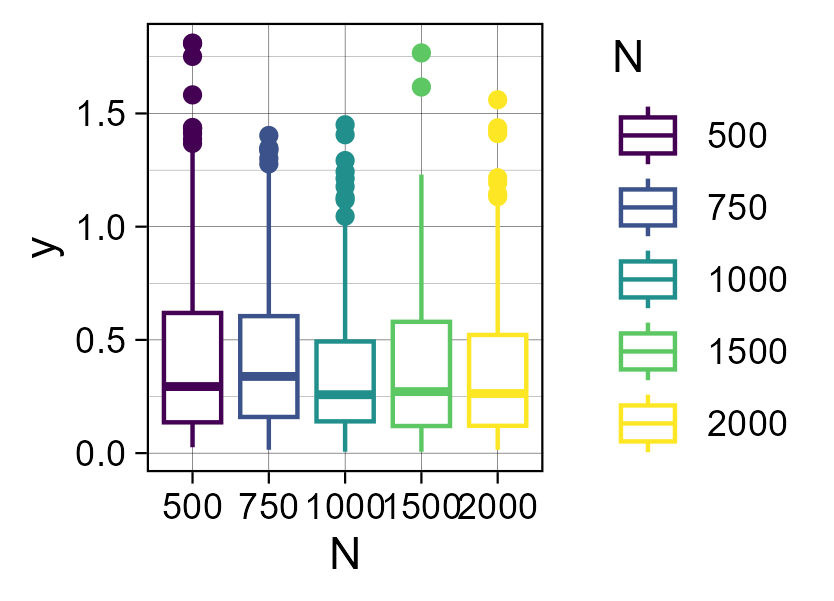}
        \caption{Estimation of $t_*$ when  $t_*\in \{0,0.1,0.3,0.5,0.7\}$ (top to bottom), $\btheta_*=(5,3,4)$,  $\kappa_{\hat Q,\btheta_0,\btheta_*}=0.25$ and $\kappa_{\hat \bGamma,\btheta_0,\btheta_*}=12.81$, based on $\widehat Q$ (left panel) and
     $\widehat \bGamma$ (right panel).}\label{fig:tauStar_t7_534}
\end{figure}


\begin{figure}[ht!]
    \centering

      \includegraphics[width=6cm,height=3cm]{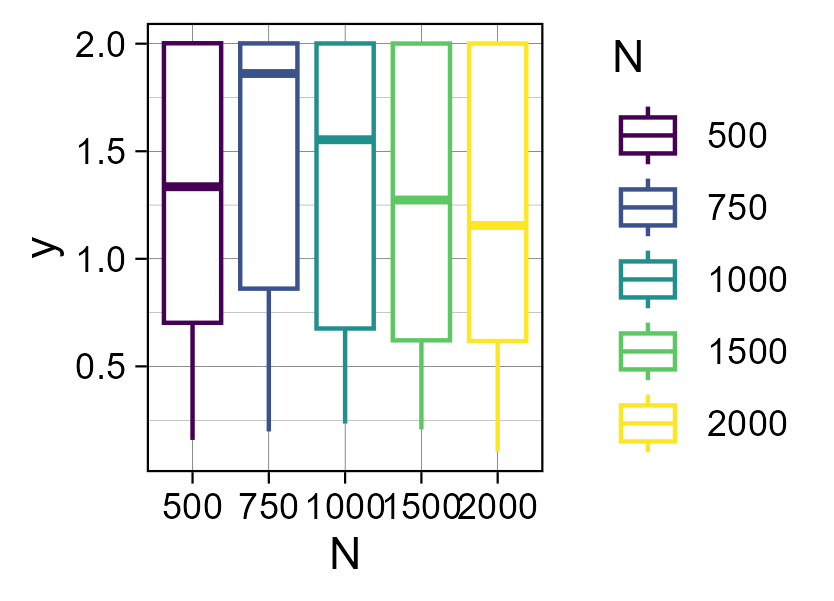}
    \includegraphics[width=6cm,height=3cm]{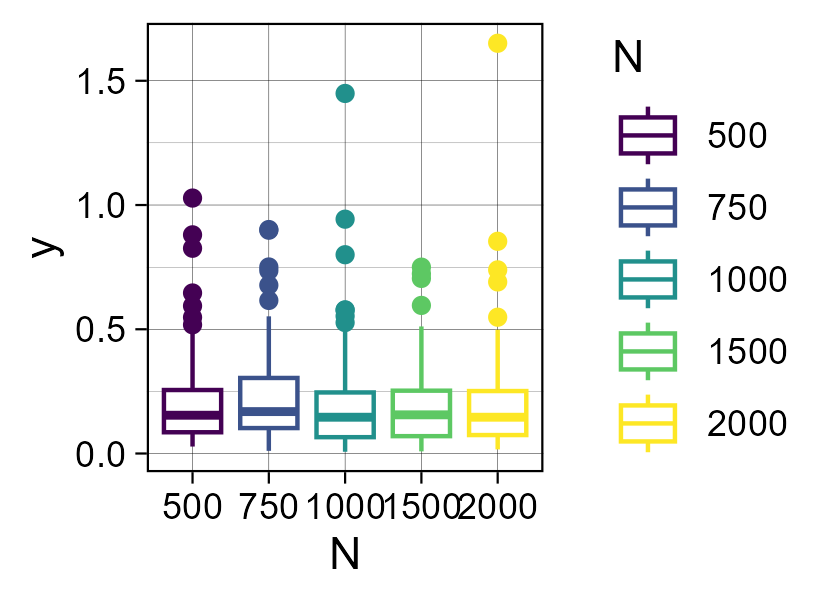}

    \includegraphics[width=6cm,height=3cm]{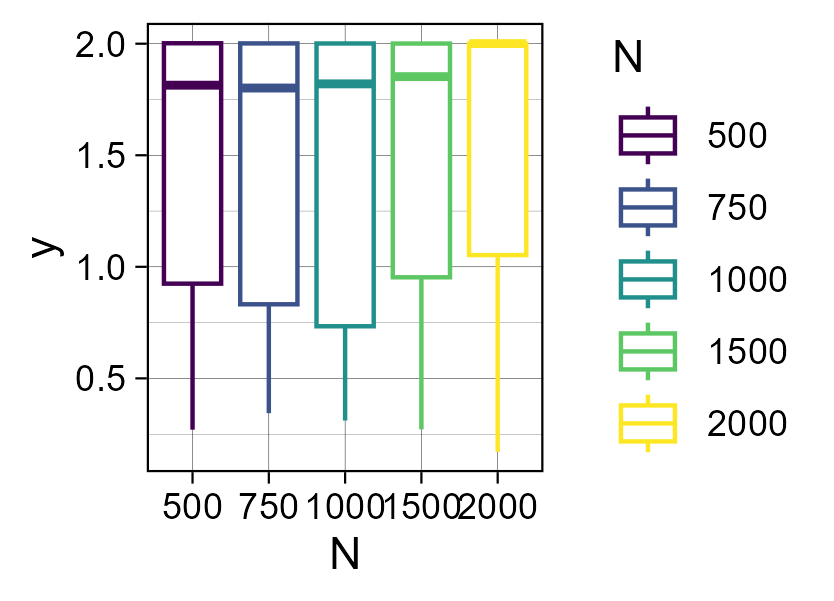}
    \includegraphics[width=6cm,height=3cm]{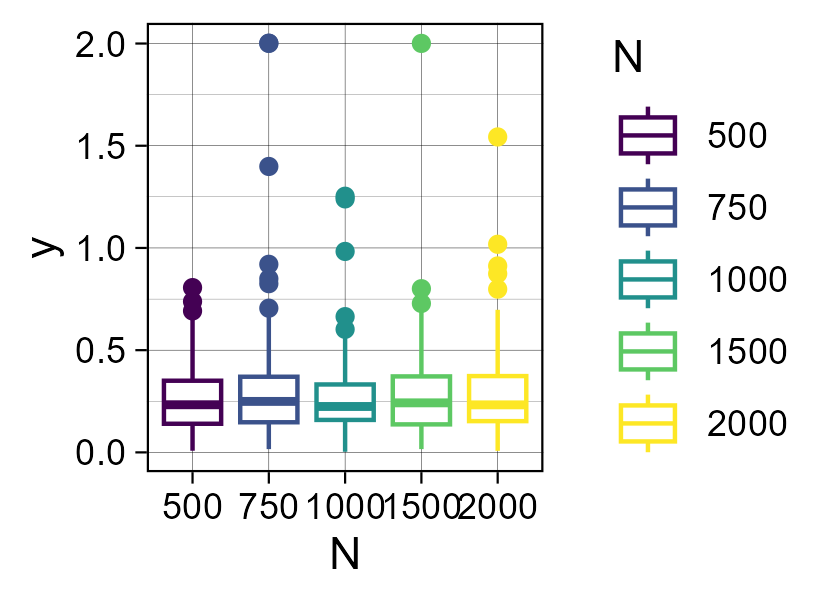}

    \includegraphics[width=6cm,height=3cm]{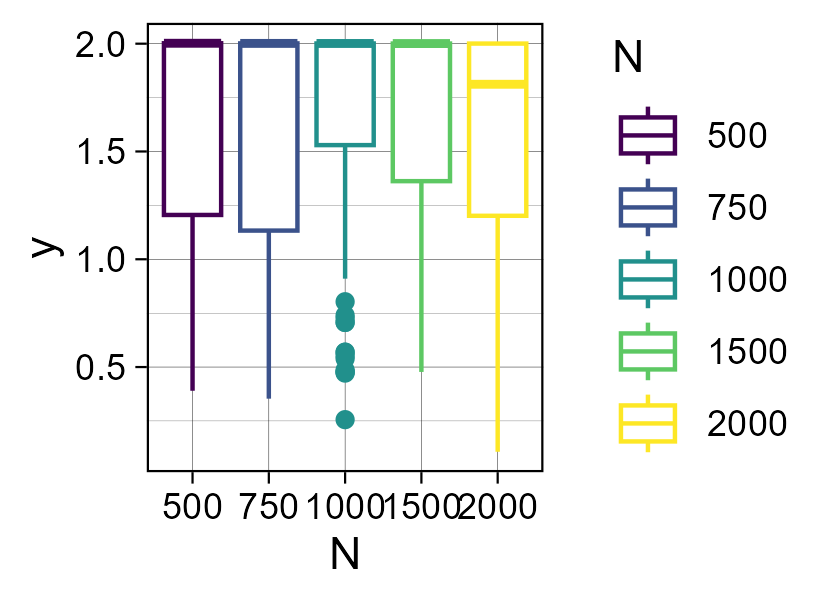}
      \includegraphics[width=6cm,height=3cm]{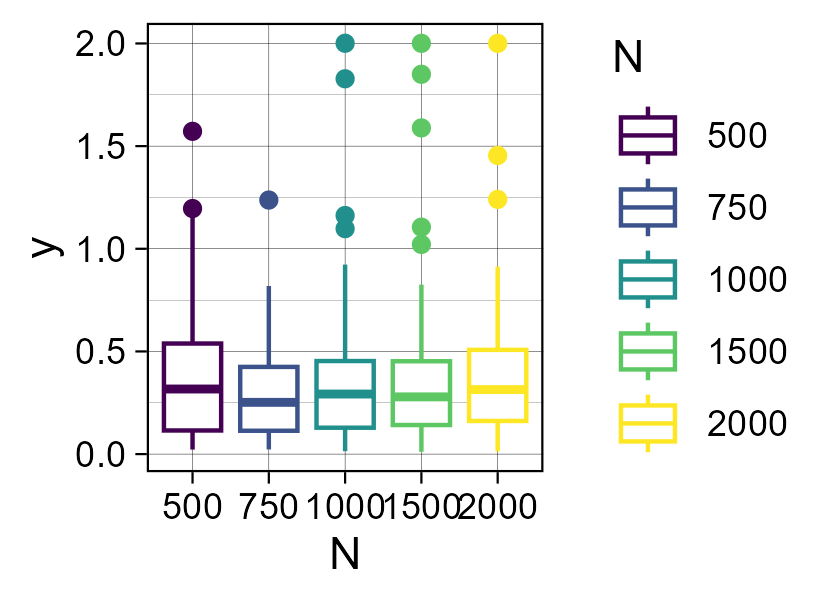}

       \includegraphics[width=6cm,height=3cm]{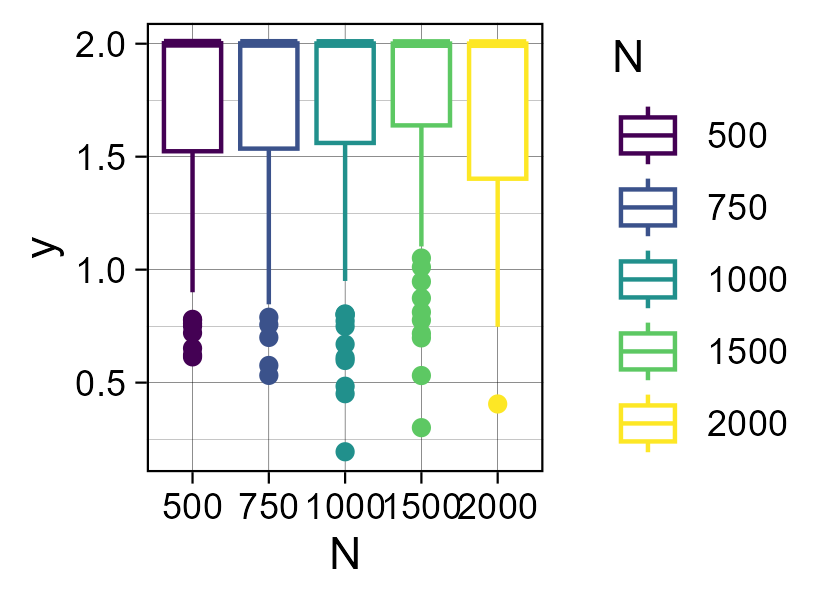}
    \includegraphics[width=6cm,height=3cm]{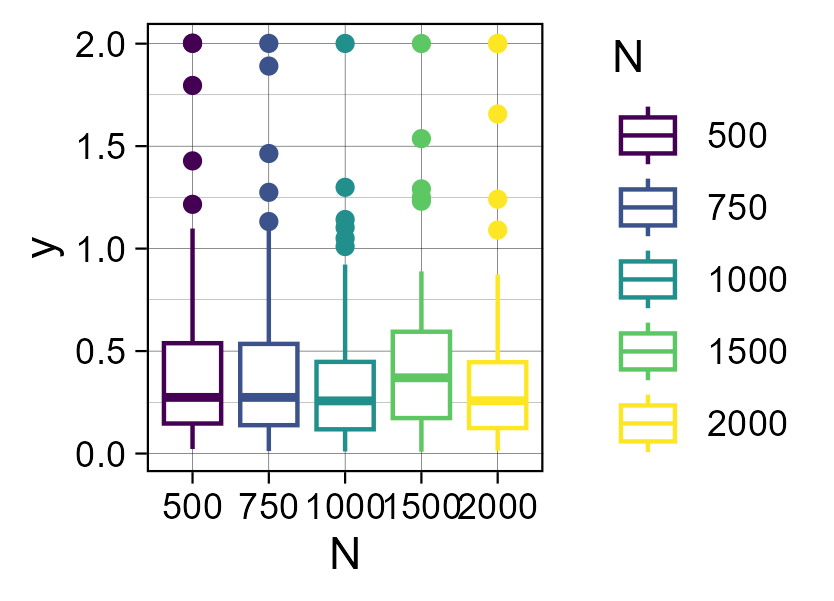}

       \includegraphics[width=6cm,height=3cm]{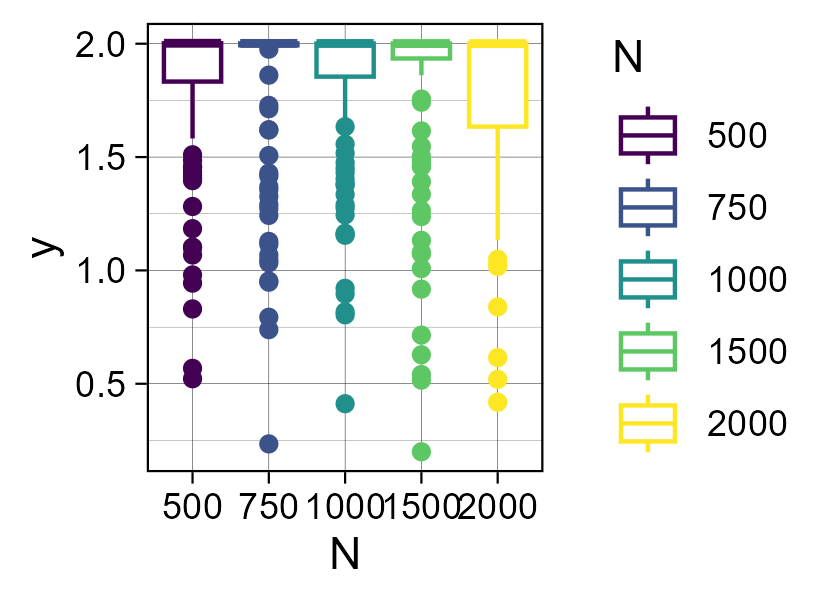}
    \includegraphics[width=6cm,height=3cm]{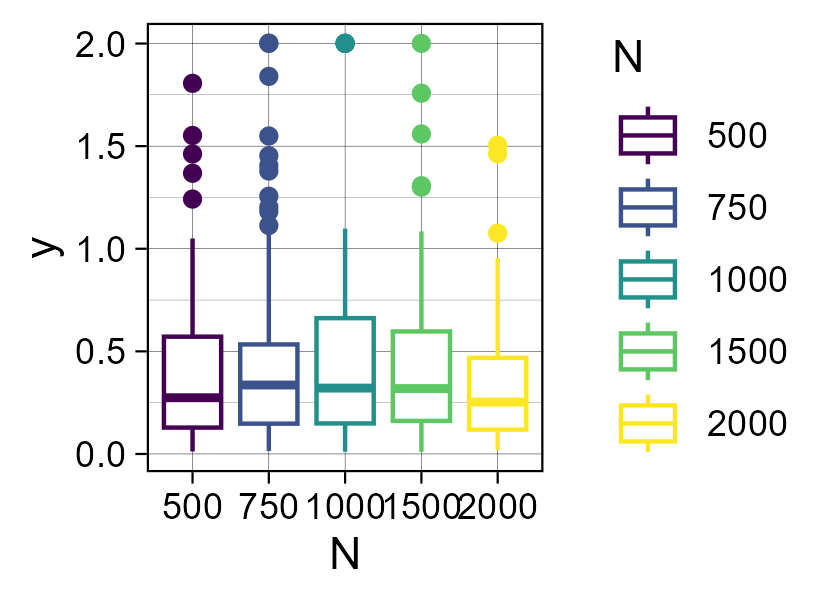}
        \caption{Estimation of $t_*$ when  $t_*\in \{0,0.1,0.3,0.5,0.7\}$ (top to bottom), $\btheta_*=(3,3,1)$,  $\kappa_{\hat Q,\btheta_0,\btheta_*}=2$ and $\kappa_{\hat \bGamma,\btheta_0,\btheta_*}=3.03$, based on $\widehat Q$ (left panel) and
     $\widehat \bGamma$ (right panel).}\label{fig:tauStar_t7_331}
\end{figure}

\subsection{Discussion of the results of the numerical experiments}

First, the case $\btheta_*=2\btheta_0$, is very interesting since Assumption \ref{hyp:theta_star} is not met, because $\kappa_{\hat Q,\btheta_0,\btheta_*}=0$.
The boxplots of $\tau_{\hat Q}$ and $\tau_{\hat\bGamma}$, displayed in Figure \ref{fig:tauStar_t7_242}, illustrates the fact that change-points are rarely detected using $\hat Q$. However, in this case, $\kappa_{\hat \bGamma,\btheta_0,\btheta_*}=4.90$, so the change-points are detected using $\hat\bGamma$.

In the case where $\btheta_*=(5,3,4)$, yielding $\kappa_{\hat Q,\btheta_0,\btheta_*}=0.25$ and $\kappa_{\hat\bGamma,\btheta_0,\btheta_*}=12.81$, the boxplots of $\tau_{\hat Q}$ and $\tau_{\hat\bGamma}$, displayed in Figure \ref{fig:tauStar_t7_534},
      illustrate the fact that it seems very difficult to detect change-points using $\hat Q$, which might be due to the low value of $\kappa_{\hat Q,\btheta_0,\btheta_*}=0.25$. However, in this case, $\kappa_{\hat \bGamma,\btheta_0,\btheta_*}=4.90$, and it seems that using $\hat\bGamma$ works quite well for detecting change-points.

Finally,
   for $\btheta_* \in \{(3,3,1), (15,3,4),(5,3,1)\}$, according to Table \ref{tab:kappa},
   the values of $\kappa_{\hat Q,\btheta_0,\btheta_*}$  and larger, and
   one can see from the boxplots displayed in Figure \ref{fig:tauStar_t7_331}, and Figures
   1-2
   in the Supplementary Material, that the performance of $\hat Q$ increases with $\kappa_{\hat Q,\btheta_0,\btheta_*}$. As for $\hat\bGamma$, the results are quite good, and it seems that $\hat\bGamma$ is generally better than $\hat Q$. Furthermore, in general, the performance of the test statistics seems to increase with $N$.

\subsection{Power curves}

We also calculated a ``power curve'' by computing the probability of rejection of the null hypothesis of no change-point when using N+K points, for 20 values of $K\in \{0.1 jN; 0 \le j\le 19\}$. For example,
if $\left\{\hat\tau_{N,\hat H,i}; i=1,\ldots, B\right\}$ are the possible values for a given $N$, using $B$ simulated trajectories, then the estimated power for statistic $\hat H$, at $\frac{j}{10}$, denoted $Pow\left(N,\frac{j}{10}\right)$, is
\begin{equation}\label{eq:est_power}
Pow_{\hat H}\left(N,\frac{j}{10}\right)=\frac{1}{B}\sum_{i=1}^B \I\left(\hat\tau_{N,G,i} \le \frac{j}{10}\right),\qquad j\in\{0,\ldots,19\}.
\end{equation}
An illustration of the power curves is given in Figure \ref{fig:pow_tauStar_t7_1534}. One can see that in general, $\hat\bGamma$ detects  change-points much earlier than $\hat Q$, and that the power of $\hat\bGamma$ is much better in general than the power of $\hat Q$.
All other curves are displayed in the Supplementary Material. In general, it seems that the test based on $\hat\bGamma$ is much more powerful than the one based on $\hat Q$.
\begin{figure}[ht!]
    \centering

    \includegraphics[width=6cm,height=3cm]{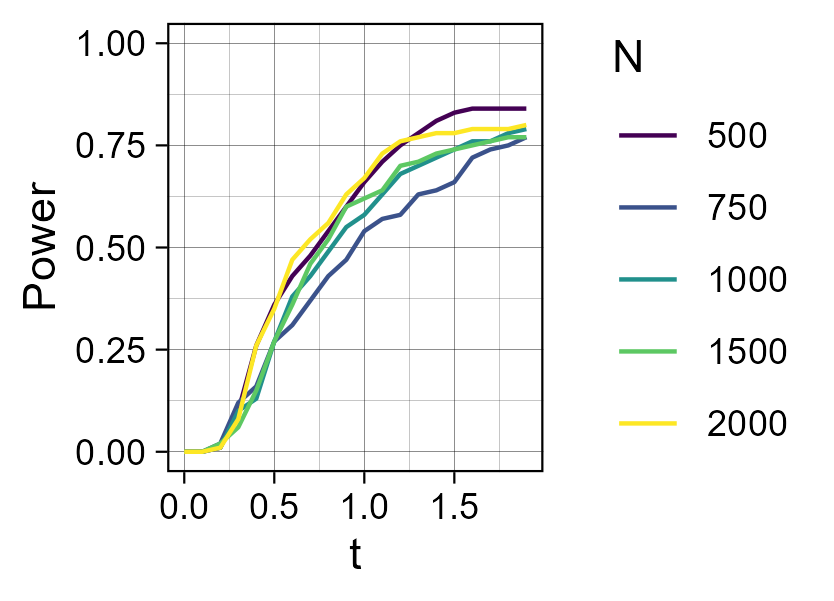}
    \includegraphics[width=6cm,height=3cm]{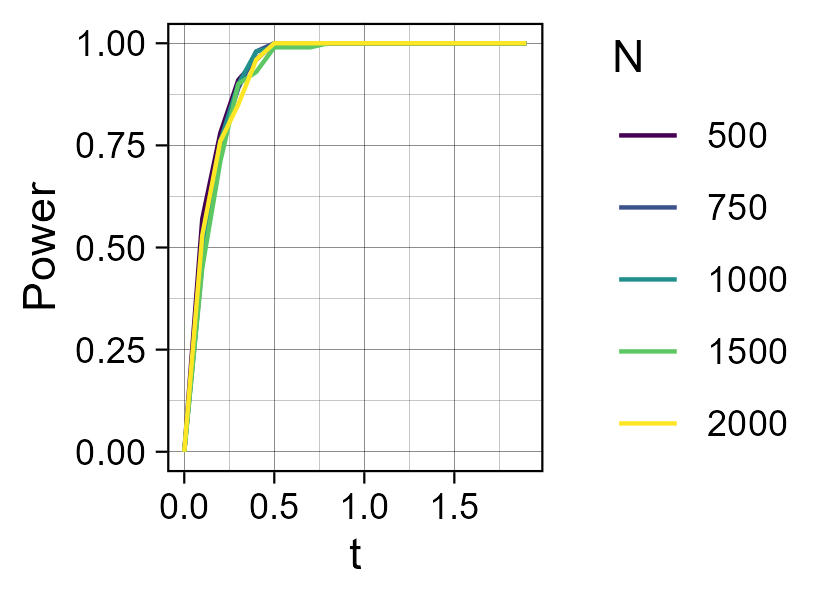}

    \includegraphics[width=6cm,height=3cm]{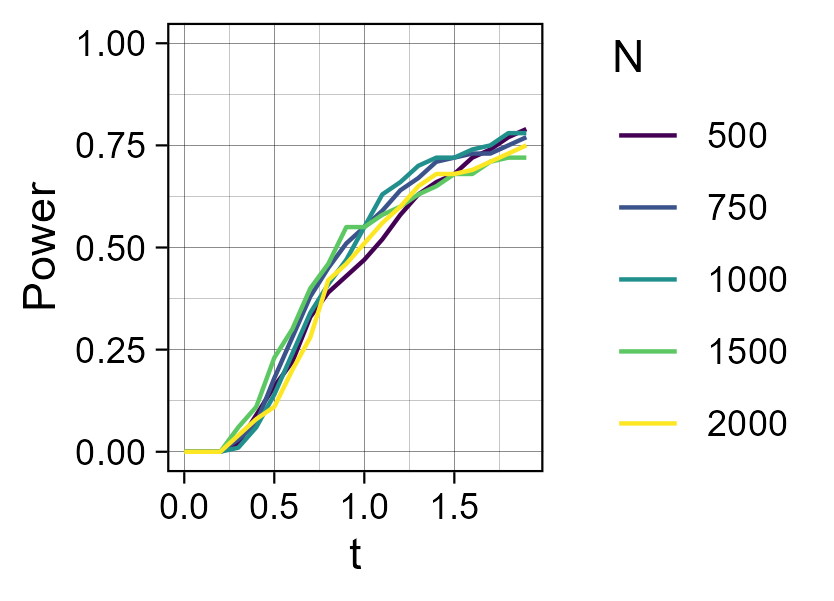}
    \includegraphics[width=6cm,height=3cm]{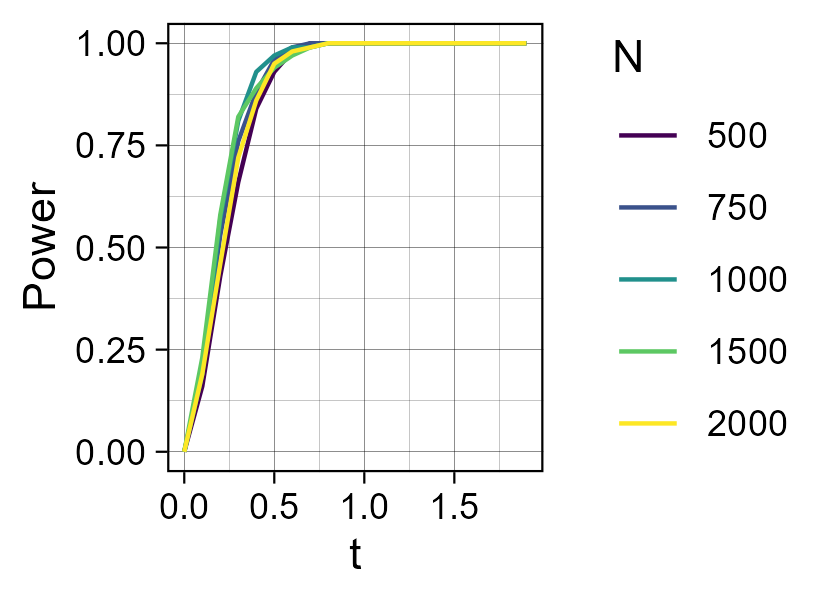}

    \includegraphics[width=6cm,height=3cm]{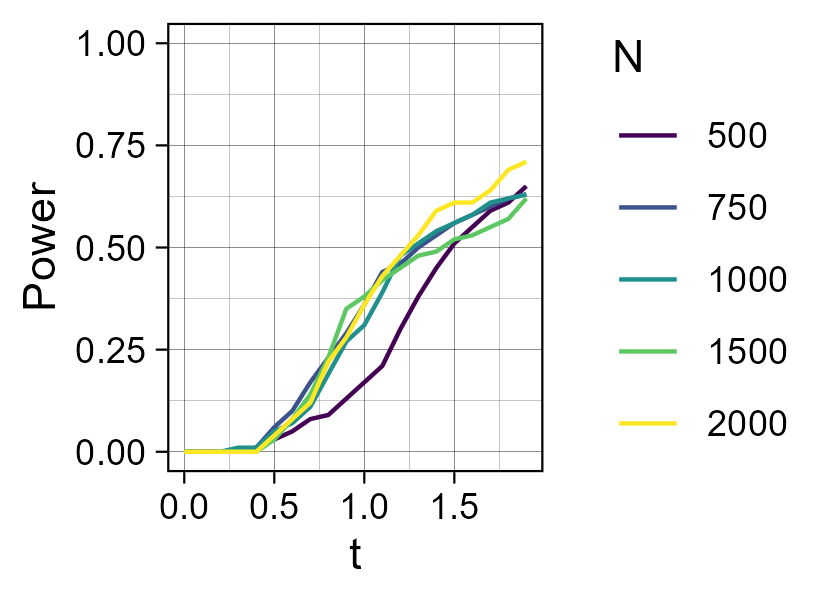}
    \includegraphics[width=6cm,height=3cm]{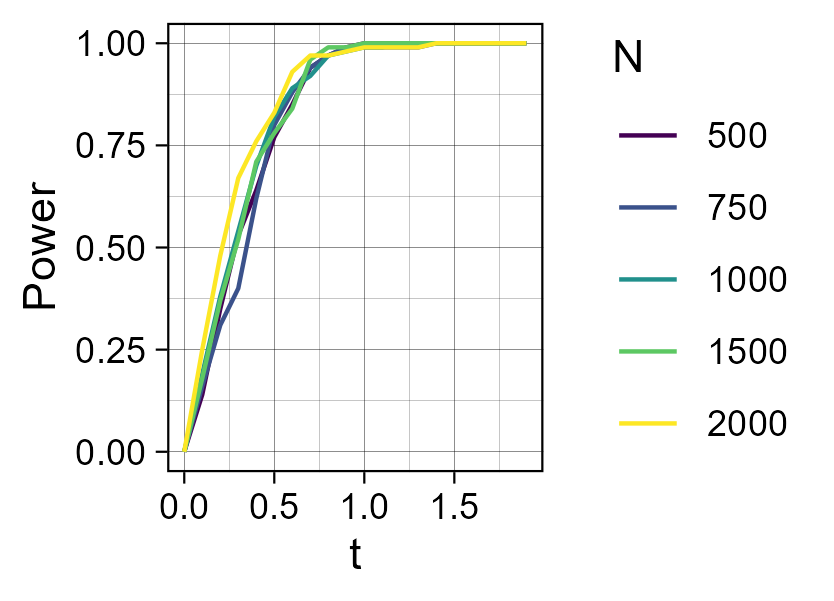}

    \includegraphics[width=6cm,height=3cm]{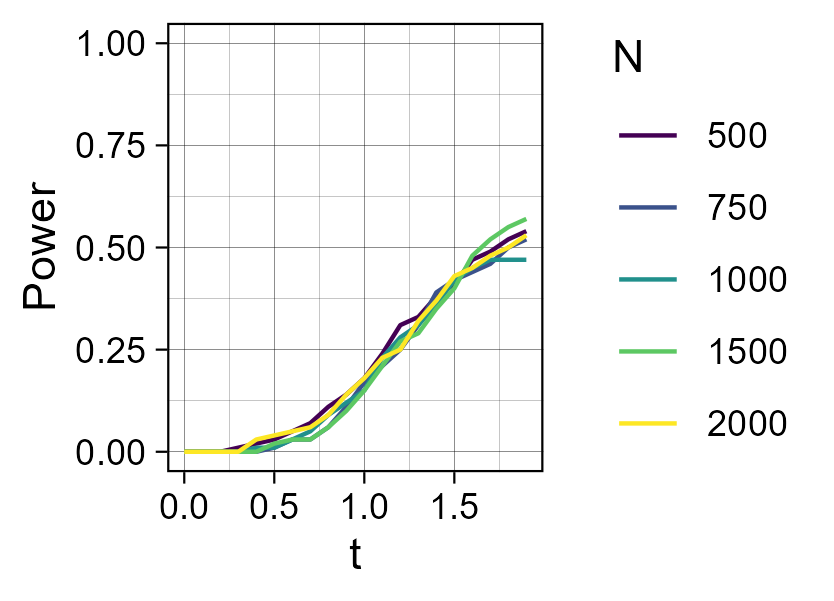}
    \includegraphics[width=6cm,height=3cm]{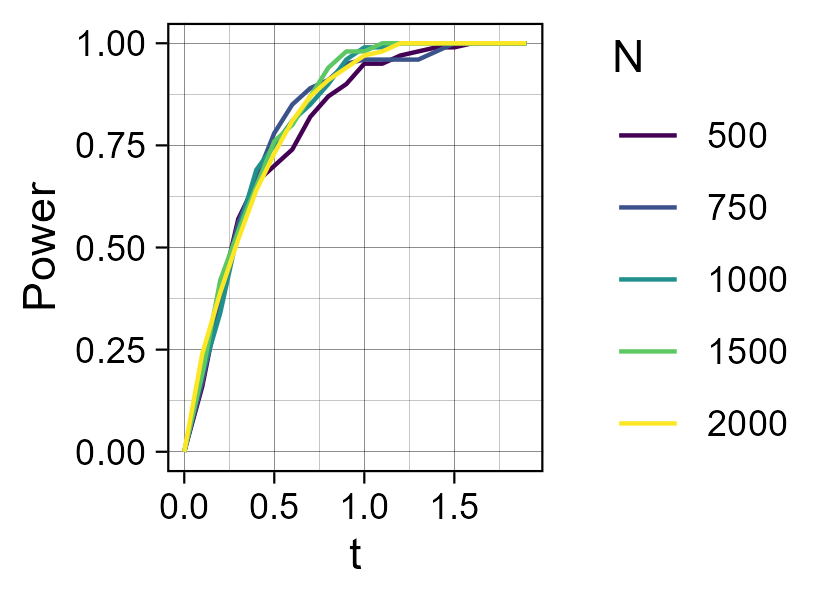}

    \includegraphics[width=6cm,height=3cm]{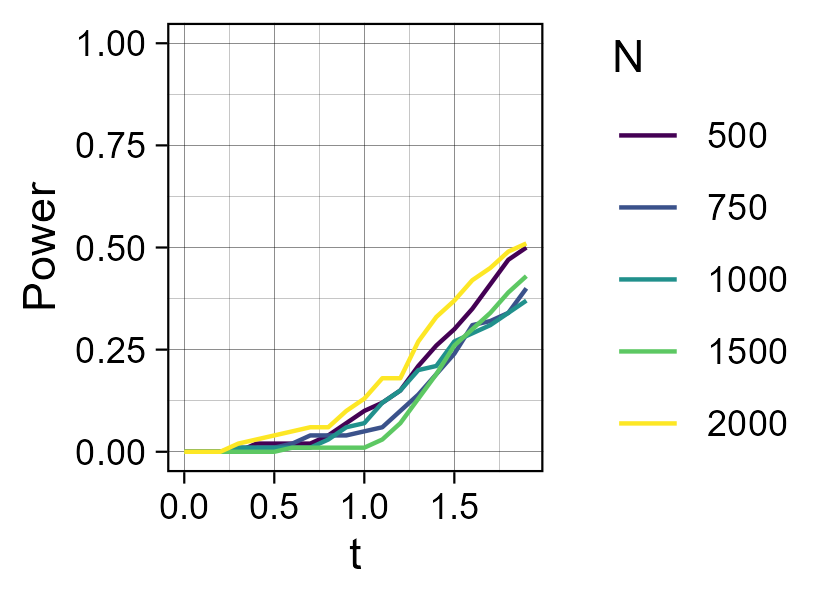}
    \includegraphics[width=6cm,height=3cm]{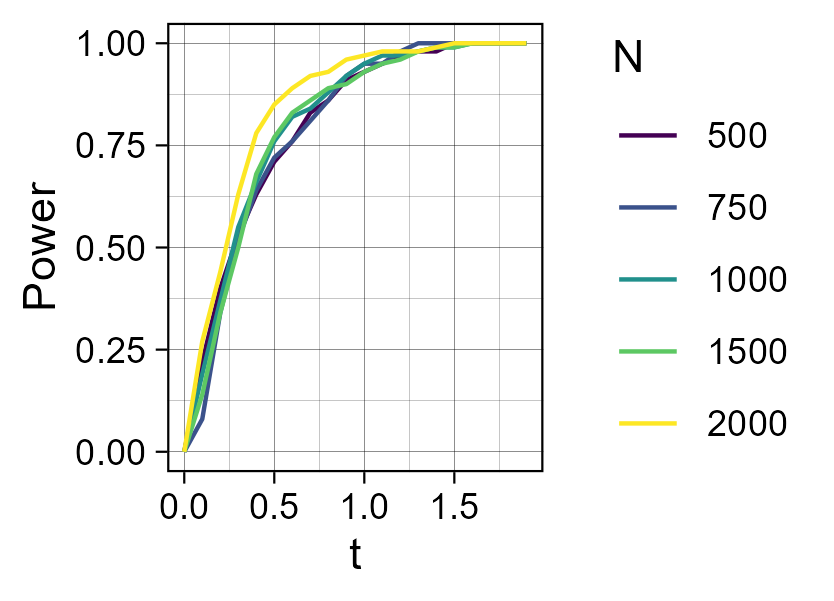}
        \caption{Empirical power of the estimation of  $t_*$ when $t_*\in\{0,0.1,0.3,0.5,0.7\}$ (top to bottom), $\btheta_*=(15,3,4)$, $\kappa_{\hat Q,\btheta_0,\btheta_*}=2.75$ and $\kappa_{\hat \bGamma,\btheta_0,\btheta_*}=11.37$, based on $\widehat Q$ (left panel) and
     $\widehat \bGamma$ (right panel).}\label{fig:pow_tauStar_t7_1534}
\end{figure}

\section{Conclusion}\label{sec:Conclusion}
In this paper, we address the sequential change-point detection problem within the framework of generalized Ornstein-Uhlenbeck (GOU) processes, where the drift is a periodic function of time $t$. First, we recall that the GOU process admits a unique and strong solution with explicit distribution law and we present the maximum likelihood estimator. In addition, for GOU processes observed at discrete times, we develop two types of tests for detecting changes in the parameters: one based on the cumulative sum of residuals, with statistic $\hat Q$, and the other based on sequential estimators, with statistic $\hat\bGamma$. Then, we derive their limiting behaviours under both the null and alternative hypotheses. In particular case, under any alternative hypothesis, the test based on statistic  $\hat \bGamma$ is consistent; however,  the test based on statistic  $\hat Q$ is not consistent in general.
These theoretical results are confirmed in our simulation section. In particular, it seems that the test based on  $\hat\bGamma$ outperforms in general the test based on $\hat Q$.

\appendix\section{Auxiliary results}

\begin{proposition}\label{prop:XT}
Suppose that Assumptions \ref{hyp:1}-\ref{hyp:N} hold. Then, for the partition $0=t_{0}<t_{1}<\dots<t_{N}=T$ on a given period $[0,T]$, with $t_{i}-t_{i-1}\le \Delta$, and $t\in [t_{i-1},t_{i}]$,
\begin{equation*}
\begin{aligned}
E\left( |X_t-X_{t_{i-1}}|\right)& \le \sigma\sqrt{(t-t_{i-1})}+O(\Delta),\\
E\left( |X_t-X_{t_{i-1}}|^2\right) & = \sigma^2 (t-t_{i-1}) + O\left(\Delta^{3/2}\right).
\end{aligned}
\end{equation*}
\end{proposition}

\begin{proof}
From \cite{nkurunziza2019improved}, we have $X_t=e^{-a t}X_{0}+h(t)+z(t)$,
where
$$ h(t)=e^{-a t}\sum_{k=1}^{p}\mu_{k}\int_{0}^{t}e^{a s}\varphi_{k}(s)ds,t\ge 0,
$$
and
$\ds z(t)=\sigma e^{-a t}\int_{0}^{t}e^{a s}dB_{s}.$ Then, for $t\in[t_{i-1},t_{i}]$,
\begin{equation}
\label{Bound_1}
X_t-X_{t_{i-1}}=\left(e^{-a t}X_{0}-e^{-a t_{i-1}}X_{0}\right)+\left(h(t)-h(t_{i-1})\right)+\left(z(t)-z(t_{i-1})\right).
\end{equation}
First,
$$
\left|e^{-a t_{i-1}}-e^{-a t} - a e^{-a t_{i-1}}(t-t_{i-1})\right| \le \frac{a^2}{2}(t-t_{i-1})^2,
$$
so for $t\in [t_{i-1},t_{i}]$,
\begin{equation}
\label{Bound_11}
E[|e^{-a t}X_{0}-e^{-a t_{i-1}}X_{0}|]\leq E[|X_{0}|]\left\{ a (t - t_{i-1})+ \frac{a^2}{2} (t - t_{i-1})^2\right\}.
\end{equation}
Second,
\begin{multline*}
h(t)-h(t_{i-1})=e^{-a t}\sum_{k=1}^{p}\mu_{k}\int_{0}^{t}e^{a s}\varphi_{k}(s)ds-e^{-a t_{i-1}}\sum_{k=1}^{p}\mu_{k}\int_{0}^{t_{i-1}}e^{a s}\varphi_{k}(s)ds\\
=e^{-a t}\sum_{k=1}^{p}\mu_{k}\int_{0}^{t}e^{a s}\varphi_{k}(s)ds-e^{-a t}\sum_{k=1}^{p}\mu_{k}\int_{0}^{t_{i-1}}e^{a s}\varphi_{k}(s)ds\\
+e^{-a t}\sum_{k=1}^{p}\mu_{k}\int_{0}^{t_{i-1}}e^{a s}\varphi_{k}(s)ds-e^{-a t_{i-1}}\sum_{k=1}^{p}\mu_{k}\int_{0}^{t_{i-1}}e^{a s}\varphi_{k}(s)ds\\
=e^{-a t}\sum_{k=1}^{p}\mu_{k}\int_{t_{i-1}}^{t}e^{a s}\varphi_{k}(s)ds+\left(e^{-a t}-e^{-a t_{i-1}}\right)\sum_{k=1}^{p}\mu_{k}\int_{0}^{t_{i-1}}e^{a s}\varphi_{k}(s)ds,
\end{multline*}
so setting $C_1=\sum_{k=1}^{p}|\mu_{k}|K_{\bvarphi}$, the first term is bounded by $C_1(t-t_{i-1})+C_1 a (t-t_{i-1})^2$,
and the second term is also bounded by the same value. As a result,
\begin{equation}
\label{Bound_12}
|h(t)-h(t_{i-1})|\le 2C_1(t-t_{i-1})+C_1 a (t-t_{i-1})^2.
\end{equation}
Third,
\begin{multline*}
z(t)-z(t_{i-1})=\sigma e^{-a t}\int_{0}^{t}e^{a s}dB_{s}-\sigma e^{-a t_{i-1}}\int_{0}^{t_{i-1}}e^{a s}dB_{s}\\
=\sigma e^{-a t}\int_{0}^{t}e^{a s}dB_{s}-\sigma e^{-a t}\int_{0}^{t_{i-1}}e^{a s}dB_{s}+\sigma e^{-a t}\int_{0}^{t_{i-1}}e^{a s}dB_{s}-\sigma e^{-a t_{i-1}}\int_{0}^{t_{i-1}}e^{a s}dB_{s}\\
=\sigma e^{-a t}\int_{t_{i-1}}^{t}e^{a s}dB_{s}+\sigma(e^{-a t}-e^{-a t_{i-1}})\int_{0}^{t_{i-1}}e^{a s}dB_{s}.
\end{multline*}
From the fact $E[|z(t) - z(t_{i-1})|]\leq\left\{E[|z(t) - z(t_{i-1})|^2]\right\}^{1/2}$, we have
\begin{multline*}
E[|z(t) - z(t_{i-1})|^2]=E\left[\left|\sigma e^{-a t}\int_{t_{i-1}}^{t}e^{a s}dB_{s}+\sigma(e^{-a t}-e^{-a t_{i-1}})\int_{0}^{t_{i-1}}e^{a s}dB_{s}\right|^2\right]\\
\leq 2\sigma^2 E\left[\left|e^{-a t}\int_{t_{i-1}}^{t}e^{a s}dB_{s}\right|^2\right]+2\sigma^2 E\left[\left|(e^{-a t}-e^{-a t_{i-1}})\int_{0}^{t_{i-1}}e^{a s}dB_{s}\right|^2\right].
\end{multline*}
Next,
\begin{multline*}
E\left[\left| e^{-a t}\int_{t_{i-1}}^{t}e^{a s}dB_{s}\right|^2\right]=
\frac{1}{2a}(1-e^{-2a (t-t_{i-1})})\le (t-t_{i-1})+ a(t-t_{i-1})^2.
\end{multline*}
Also,
\begin{multline*}
E\left[\left|(e^{-a t}-e^{-a t_{i-1}})\int_{0}^{t_{i-1}}e^{a s}dB_{s}\right|^2\right]=(e^{-a t}-e^{-a t_{i-1}})^2\int_{0}^{t_{i-1}}e^{2a s}d{s}\\
\le \frac{a}{2}(t-t_{i-1})^2\left\{1+a(t-t_{i-1})\right\}^2.
\end{multline*}
As a result,
\begin{equation}
\label{Bound_13}
E[|z(t) - z(t_{i-1})|] \le  \sigma\sqrt{(t-t_{i-1})} \left\{1+a (t-t_{i-1})\right\}.
\end{equation}
\eqref{Bound_1}-\eqref{Bound_13} imply that for $t\in [t_{i-1},t_{i}]$,
\begin{equation}\label{ineqMay10_1}
E\left( |X_t-X_{t_{i-1}}|\right) \le  \sigma\sqrt{(t-t_{i-1})}+O(\Delta).
\end{equation}
Furthermore, for $ t\in [t_{i-1},t_{i}]$, we have
\begin{equation}\label{ineqMay10_2}
E\left( |X_t-X_{t_{i-1}}|^2\right) = \sigma^2 (t-t_{i-1})  + O\left(\Delta^{3/2}\right).
\end{equation}
\end{proof}

\begin{remark}\label{toto}
Note that the process $X_t=e^{-a t}X_{0}+h(t)+z(t)$ is not stationary. However, \cite{dehling2010drift} proposed an auxiliary ergodic and stationary process, namely
\begin{equation}\label{ergodic}
\tilde{X}_t=\tilde{h}(t)+\tilde{z}(t), ~~t\ge 0,
\end{equation}
where the function $\tilde{h}(t):[0,+\infty)\rightarrow\mathbb{R}$ and the process $\{\tilde{z}(t):t\ge 0\}$ are defined by
\begin{equation*}
\tilde{h}(t)=e^{-a t}\sum_{k=1}^{p}\mu_{k}\int_{-\infty}^{t}e^{a s}\varphi_{k}(s)ds,~~\tilde{z}(t)=\sigma e^{-a t}\int_{-\infty}^{t}e^{a s}d\tilde{B}_{s},
\end{equation*}
where $\{\tilde{B}_{s}\}_{s\in\mathbb{R}}$ denotes a bilateral Brownian motion, i.e.
$\tilde{B}_{s}=B_{s}\mathbb{I}_{\mathbb{R}+}(s)+\bar{B}_{-s}\mathbb{I}_{\mathbb{R}-}(s)$,
 $s\in\mathbb{R}$,
with $\{B_{s}\}_{s\ge 0}$ and $\{\bar{B}_{-s}\}_{s\ge 0}$ being two independent standard Brownian motions. For the details of the proof of ergodicity, we refer to \cite{dehling2010drift, nkurunziza2019improved}. From the construction of the auxiliary process, we can prove that
\begin{equation}
\label{ErgodicBound}
\sup_{t\ge 0}E\left[\left| \tilde{X}_{t}\right|^{m}\right]
\le \tilde{m},
\end{equation}
for some positive constant $\tilde{m}$.
The relation between $X_t$ and $\tilde{X}_t$ is given by Theorem 2.1 in \cite{nkurunziza2019improved}. For convenience, we state this theorem here.

\end{remark}

\begin{theorem}[\cite{nkurunziza2019improved}]
\label{CovSolu_Auxi} 
 If Assumptions \ref{hyp:1}-\ref{hyp:2} hold, then we have

$(i)~|X_t-\tilde{X}_t|\le C_{0}e^{-a t}$, \textit{where} $C_{0}$ \textit{is a random variable such that} $E(|C_{0}|^{2})<\infty$,

$(ii)~|X_t-\tilde{X}_t|\xrightarrow[t\rightarrow\infty]{\text{a.s. and } L^{2}}0$ and $|X^{2}_t-\tilde{X}^{2}_t|\xrightarrow[t\rightarrow\infty]{\text{a.s. and } L^{1}}0$,

$(iii) \ds \sup_{0\le t_0 < t_1 \le T}\left|\displaystyle\frac{1}{T}\int_{t_0 }^{t_1}\tilde{X}_t\bvarphi(t)dt-\frac{1}{T}\int_{t_0 }^{t_1}{X}_t\bvarphi(t)dt\right|\xrightarrow[T\rightarrow\infty]{\text{a.s. and } L^{2}}0$,

$(iv)\ds \sup_{0\le t_0 < t_1  \le T}\left|\displaystyle\frac{1}{T}\int_{t_0 }^{t_1}\tilde{X}^{2}_tdt-\frac{1}{T}\int_{t_0}^{t_1}{X}^{2}_tdt\right|\xrightarrow[T\rightarrow\infty]{\text{a.s. and } L^{1}}0.$
\end{theorem}

\begin{proposition}\label{prop:inverse}
If $\tau>0, A_n-A = O_P(n^\tau)$ and $A$ is invertible matrix, then $A_n^{-1}-A^{-1} = O_P(n^\tau)$.
\end{proposition}

\begin{proof}
This follows from $$\|(A+x)^{-1}-A^{-1}+A^{-1}xA^{-1}\|\le \frac{\|I\|\|A^{-1}\|^3 \|x\|^2}{1-\|A^{-1}\|\|x\|}=O\left(\|x\|^2\right),$$
since $\|(I-B)^{-1}\|\le \frac{\|I\|}{1-\|B\|}$, if $\|B\|<1$. In fact, for any $y\neq 0$, $\|(I-B)y\|\ge \|y\|-\|By\|\ge \|y\|(1-\|B\|)$. As a result,
$$
\|I\|=\|(I-B)^{-1}(I-B)\|\ge \|(I-B)^{-1}(I-B)\|-\|(I-B)^{-1}B\| \ge \|(I-B)^{-1}\|(1-\|B\|).
$$
\end{proof}

\section{Proofs for the continuous process}

\subsection{Proof of \Cref{ConvQSigma}}\label{app:pf1}
First, by \Cref{hyp:1}, we have
  \begin{equation*}
     \frac{1}{T}\int_{0}^{T}\bvarphi(t)\bvarphi^\top(t)dt=\frac{1}{T}\int_{0}^{\lfloor T\rfloor}\bvarphi(t)\bvarphi^\top(t)dt+\frac{1}{T}\int_{\lfloor T\rfloor}^{T}\bvarphi(t)\bvarphi^\top(t)dt
  \end{equation*}
  \begin{equation*}
     =\frac{1}{T}\sum_{j=1}^{\lfloor T\rfloor}\int_{j-1}^{j}\bvarphi(t)\bvarphi^\top(t)dt+\frac{1}{T}\int_{\lfloor T\rfloor}^{T}\bvarphi(t)\bvarphi^\top(t)dt
  \end{equation*}
By \Cref{hyp:2},
  \begin{equation*}
     \frac{1}{T}\sum_{j=1}^{\lfloor T\rfloor}\int_{j-1}^{j}\bvarphi(t)\bvarphi^\top(t)dt
     =\frac{1}{T}\sum_{j=1}^{\lfloor T\rfloor}\int_{0}^{1}\bvarphi(u)\bvarphi^{\top}(u)du=\frac{\lfloor T\rfloor}{T}I_{p}\xrightarrow[T\to\infty]{}I_{p}.
  \end{equation*}
Furthermore,
  \begin{equation*}
  \left\|\frac{1}{T}\int_{\lfloor T\rfloor}^{T}\bvarphi(t)\bvarphi^\top(t)dt\right\| \le \frac{K_{\bvarphi}^2}{T},
  \end{equation*}
so
  \begin{equation*}
\left\| \frac{1}{T}\int_{0}^{T}\bvarphi(t)\bvarphi^\top(t)dt -I_p\right\| = O(1/T).
  \end{equation*}
Next, by  \Cref{CovSolu_Auxi}, $\left|\tilde{X}_{t}- X_{t}\right|\xrightarrow[t\to\infty]{a.s.~{\rm and}~L^{2}}0$, so  by using inequality \eqref{kphi}, we have
\begin{equation*}
\bvarphi(t)X_{t}-\bvarphi(t) \tilde{X}_{t}\xrightarrow[t\to\infty]{a.s \text{ \rm and } L^{2}}0.
\end{equation*}
By the continuous version of Ces\`aro's mean theorem, we have
\begin{equation*}
\frac{1}{T}\int_{0}^{T}\bvarphi(t) X_{t}dt-\frac{1}{T}\int_{0}^{T}\bvarphi(t) \tilde{X}_{t}dt\xrightarrow[T\to\infty]{a.s.~{\rm and}~L^{2}}0.
\end{equation*}
Then, we just need to prove
  \begin{equation*}
     \frac{1}{T}\int_{0}^{T}\bvarphi(t) \tilde{X}_{t}dt\xrightarrow[T\to\infty]{a.s.{\rm~and~}L^{2}}\int_{0}^{1}\bvarphi(t)\tilde{h}(t)dt.
  \end{equation*}
 In fact,
  \begin{equation*}
  \begin{aligned}
  \frac{1}{T}\int_{0}^{T}\bvarphi(t) \tilde{X}_{t}dt&=\frac{1}{T}\int_{0}^{\lfloor T\rfloor}\bvarphi(t) \tilde{X}_{t}dt+\frac{1}{T}\int_{\lfloor T\rfloor}^{T}\bvarphi(t) \tilde{X}_{t}dt\\ &=\frac{1}{T}\sum_{j=1}^{\lfloor T\rfloor}\int_{j-1}^{j}\bvarphi(t)\tilde{X}_{t}dt+\frac{1}{T}\int_{\lfloor T\rfloor}^{T}\bvarphi(t) \tilde{X}_{t}dt.
  \end{aligned}
  \end{equation*}
By construction of the auxiliary process given in \eqref{ergodic}, $\{\tilde{X}_{u+k-1},0\le u\le1\}_{k\in\mathbb{N}}$ is stationary and ergodic. Then, $\left\{\int_{j-1}^{j}\bvarphi(t) \tilde{X}_{t}dt\right\}_{j\in\mathbb{N}}$ is stationary and ergodic. This leads to
\begin{equation}\label{G0917_1}
  \begin{array}{ll}
  \ds\frac{1}{T}\sum_{j=1}^{\lfloor T\rfloor}\int_{j-1}^{j}\bvarphi (t)\tilde{X}_{t}dt \xrightarrow[T\to\infty]{a.s.{\rm~and~}L^{1}}\int_{0}^{1}\bvarphi (u)E\left[ \tilde{X}_{u}\right]du=\int_{0}^{1}\bvarphi(u)\tilde{h}(u)dt. \\
  \end{array}
  \end{equation}
 In addition, for $i=1,2,\cdots,p$, let $S_{\lfloor T\rfloor i}=\ds\int_{\lfloor T\rfloor}^{T}\varphi_{i}(t)\tilde{X}_{t}dt$. Then, we have
\begin{equation*}
E[S_{\lfloor T\rfloor i}^{2}]=    E\left[\left(\int_{\lfloor T\rfloor}^{T}\varphi_{i}(t)n \tilde{X}_{t}dt\right)^{2}\right]\le K_{\bvarphi}^{2}E\left[\left(\int_{\lfloor T\rfloor}^{T}\tilde{X}_{t}dt\right)^{2}\right]
\end{equation*}
\begin{equation*}
= K_{\bvarphi}^{2}E\left[\left(\int_{\lfloor T\rfloor}^{T} \tilde{X}_{t}\int_{\lfloor T\rfloor}^{T}\tilde{X}_{u}dudt\right)\right]= K_{\bvarphi}^{2}E\left[\left(\int_{\lfloor T\rfloor}^{T}\int_{\lfloor T\rfloor}^{T} \tilde{X}_{t}\tilde{X}_{u}dudt\right)\right].
\end{equation*}
Hence,
\begin{equation*}
E[S_{\lfloor T\rfloor i}^{2}]= K_{\bvarphi}^{2}\left(\int_{\lfloor T\rfloor}^{T}\int_{\lfloor T\rfloor}^{T}E\left[\tilde{X}_{t} \tilde{X}_{u}\right]dudt\right).
\end{equation*}
By (\ref{ErgodicBound}), for $m=2$, $E\left[\tilde{X}_{t}\tilde{X}_{u}\right]\le \sup_{t\ge 0}E\left[\left|  \tilde{X}_{t}\right|^{2}\right] \le \tilde{m}<\infty $, so by \eqref{eq:XL2}
\begin{equation*}
E[S_{\lfloor T\rfloor i}^{2}]\le K_{\bvarphi}^{2}\sup_{t\ge 0}E\left[\left| X_{t}\right|^{2}\right]\left(\int_{\lfloor T\rfloor}^{T}\int_{\lfloor T\rfloor}^{T}dudt\right)=K_{\bvarphi}^{2}K_1(T-\lfloor T\rfloor)^{2}.
\end{equation*}
Since $T-\lfloor T\rfloor\le 1$, it follows that
\begin{equation*}
\sum_{\lfloor T\rfloor=1}^{\infty}{\rm P}\left(\frac{|S_{\lfloor T\rfloor i}|}{\lfloor T\rfloor}>\delta\right) \le \sum_{\lfloor T\rfloor=1}^{\infty}\frac{E[|S_{\lfloor T\rfloor i}|^{2}]}{\lfloor T\rfloor^{2}\delta^{2}}\le \sum_{\lfloor T\rfloor=1}^{\infty}\frac{K_{\bvarphi}^{2}K_1}{\lfloor T\rfloor^{2}\delta^{2}}
\end{equation*}
\begin{equation*}
=K_{\bvarphi}^{2}K_1\sum_{\lfloor T\rfloor=1}^{\infty}\frac{1}{\lfloor T\rfloor^{2}\delta^{2}}<\infty.
\end{equation*}
By Borel-Cantelli's Lemma, we have
$\ds\frac{|S_{\lfloor T\rfloor i}|}{\lfloor T\rfloor}\xrightarrow[T\to\infty]{a.s.}0$, for $i=1,2,\cdots,p$,
which implies that
\begin{equation}\label{G0917_2}
\frac{1}{T}\int_{\lfloor T\rfloor}^{T}\bvarphi(t) \tilde{X}_{t}dt\xrightarrow[T\to\infty]{a.s.}0.
  \end{equation}
Furthermore,
\begin{equation*}
\begin{aligned}
&E\left[\left|\left|\ds\frac{1}{T}\int_{\lfloor T\rfloor}^{T}\bvarphi(t) \tilde{X}_{t}dt\right|\right|^{m}\right]=E\left[\left|\left|\ds\frac{1}{T}\int_{\lfloor T\rfloor}^{T}\bvarphi(t)\tilde{X}_{t}dt\right|\right|^{m}\right]\\
&\le E\left[\left(\ds\frac{1}{T}\int_{\lfloor T\rfloor}^{T}\left|\left|\bvarphi(t)\right|\right|| \tilde{X}_{t}|dt\right)^{m}\right]
\le \left(\ds\frac{1}{T}\right)^{m}\left(pK_{\bvarphi}\right)^{m/2}E\left[\left(\ds\int_{\lfloor T\rfloor}^{T}| \tilde{X}_{t}|dt\right)^{m}\right]\\
&\le \left(\ds\frac{1}{T}\right)^{m}\left(pK_{\bvarphi}\right)^{m/2}(T-\lfloor T\rfloor)^{m-1}\left(\ds\int_{\lfloor T\rfloor}^{T}E\left[| \tilde{X}_{t}|^{m}\right]dt\right)\\
&\le \left(\ds\frac{1}{T}\right)^{m}\left(pK_{\bvarphi}\right)^{m/2}\left(\ds\int_{\lfloor T\rfloor}^{T}E\left[| \tilde{X}_{t}|^{m}\right]dt\right).
\end{aligned}
\end{equation*}
By (\ref{ErgodicBound}), $\ds\sup_{t\ge 0}E\left[\left| \tilde{X}_{t}\right|^{m}\right]\le\tilde{m} <\infty $, so
\begin{equation}\label{G0917_3}
E\left[\left|\left|\ds\frac{1}{T}\int_{\lfloor T\rfloor}^{T}\bvarphi(t) \tilde{X}_{t}dt\right|\right|^{m}\right]\le\left(\ds\frac{1}{T}\right)^{m}\left(pK_{\bvarphi}\right)^{m/2}\tilde{m}\xrightarrow[T\to\infty]{}0.
\end{equation}
\eqref{G0917_1}, \eqref{G0917_2} and \eqref{G0917_3} imply that
  \begin{equation*}
     \frac{1}{T}\int_{0}^{T}\bvarphi(t) X_{t}dt\xrightarrow[T\to\infty]{a.s.~{\rm and}~L^{m}}\int_{0}^{1}\bvarphi(t)\tilde{h}(t)dt.
  \end{equation*}
Since in \Cref{CovSolu_Auxi}, we already have $ X_{t}^{2}-\tilde{X}_{t}^{2}\xrightarrow[T\to\infty]{a.s.~{\rm and}~L^{1}}0$, by the continuous version of Ces\`aro's mean theorem, we have
 \begin{equation*}
\frac{1}{T}\int_{0}^{T}(X_{t})^{2}dt-\frac{1}{T}\int_{0}^{T}(\tilde{X}_{t})^{2}dt\xrightarrow[T\to\infty]{a.s.~{\rm and}~L^{1}}0.
\end{equation*}
In addition, we just need to prove that
\begin{equation*}
  \frac{1}{T}\int_{0}^{T}( \tilde{X}_{t})^{2}dt\xrightarrow[T\to\infty]{a.s.~{\rm and}~L^{1}}\int_{0}^{1}(\tilde{h}(t))^{2}dt+\frac{\sigma^{2}}{2a}.
\end{equation*}
In fact,
\begin{equation}\label{tilXt2inequality}
 \frac{1}{T}\int_{0}^{\lfloor T\rfloor}(\tilde{X}_{t})^{2}dt \le \frac{1}{T}\int_{0}^{T}(\tilde{X}_{t})^{2}dt\le \frac{1}{T}\int_{0}^{\lfloor T\rfloor+1}( \tilde{X}_{t})^{2}dt.
\end{equation}
Then,
\begin{equation*}
 {\rm LHS~of~}(\ref{tilXt2inequality})=\frac{1}{T}\int_{0}^{\lfloor T\rfloor}( \tilde{X}_{t})^{2}dt= \frac{1}{T}\sum_{j=1}^{\lfloor T\rfloor}\int_{j-1}^{j}(\tilde{X}_{t})^{2}dt.
 \end{equation*}
Since $\left\{\tilde{X}_{u+(k-1)}\right\}_{k\in\mathbb{N}}$ is stationary and ergodic, and $\left(\tilde{X}_{u+(k-1)}\right)^{2}$ is a measurable function of the stationary and ergodic process $\left\{{\tilde{X}_{u+(k-1)}}\right\}_{i\in\mathbb{N}}$, by Theorem 3.5.8 in \cite{Stout1974}, $\left\{( \tilde{X}_{u+(k-1)})^{2}\right\}_{k\in\mathbb{N}}$ is stationary and ergodic. So \emph{the  pointwise ergodic theorem for stationary sequences} (Theorem 3.5.7 in \cite{Stout1974} can be applied to the sequence $\left\{\ds\int_{j-1}^{j}( \tilde{X}_{t})^{2}dt\right\}_{j\in\mathbb{N}}$. Thus,
\begin{equation*}
 {\rm LHS~of~}(\ref{tilXt2inequality})\xrightarrow[n\to\infty]{a.s.~{\rm and}~L^{1}}\int_{0}^{1}\left(\tilde{h}^{2}(u)+\frac{\sigma^{2}}{2a}\right)du=\int_{0}^{1}(\tilde{h}(t))^{2}dt+\frac{\sigma^{2}}{2a}.
 \end{equation*}
 Similarly,
 \begin{align*}
 {\rm RHS~of~}(\ref{tilXt2inequality})&=\frac{1}{T}\int_{0}^{\lfloor T\rfloor+1}(\tilde{X}_{t})^{2}dt\xrightarrow[n\to\infty]{a.s.~{\rm and}~L^{1}}\int_{0}^{1}(\tilde{h}^{2}(u)+\frac{\sigma^{2}}{2a})du=\int_{0}^{1}(\tilde{h}(t))^{2}dt+\frac{\sigma^{2}}{2a}. \end{align*}
 This implies that
\begin{equation*}
  \frac{1}{T}\int_{0}^{T}( \tilde{X}_{t})^{2}dt\xrightarrow[T\to\infty]{a.s.~{\rm and}~L^{1}}\int_{0}^{1}(\tilde{h}(t))^{2}dt+\frac{\sigma^{2}}{2a} .
\end{equation*}
In addition, since the matrix $\Sigma$ is invertible, we can apply the continuous mapping theorem to get that $T\bQ_{[0,T]}^{-1}\to \bSigma^{-1}$ a.s.
This completes the proof. \qed

\subsection{Proof of \Cref{lem:mle}}\label{app:pf2}
\Cref{MLE_0206} indicates that $\check\btheta_{T}-\btheta=\sigma \bQ_{[0,T]}^{-1}\bR_{[0,T]}$, which implies that $T^{1/2} \left(\check\btheta_{T}-\btheta\right)=\sigma T\bQ_{[0,T]}^{-1}\frac{1}{T^{1/2} }\bR_{[0,T]}$. From \Cref{ConvQSigma}, $T\bQ_{[0,T]}^{-1}\to \Sigma^{-1}$ almost surely, as $T\to\infty$. Furthermore, we prove that $\ds\frac{1}{T}\bR_{[0,T]}$ is $L^{2}$ bounded.
For $i=1,\cdots,p$, $\ds\int_{0}^{T}\varphi_{i}(t)dB_{t}$ is a martingale, and
\begin{equation*}
\sup_{T\geq 0}E\left[\left|\frac{1}{T^{1/2} }\int_{0}^{T}\varphi_{i}(t)dB_{t}\right|^{2}\right]=\sup_{T\geq 0}E\left[\frac{1}{T}\int_{0}^{T}\varphi_{i}^{2}(t)dt\right]\leq K_{\bvarphi}^{2},
\end{equation*}
and from (\ref{ErgodicBound}), $\sup_{t\geq 0}E\left[(X_{t})^{2}\right]<\infty$,
which completes the $L^{2}$ boundedness.
Furthermore, $\bR_{[0,T]}$ is a martingale, and by using Doob’s maximal inequality for
submartingales, for any $\delta>0$, we have
\begin{equation*}
{\rm P}\left(\sup_{2^{k}\leq T\leq 2^{k+1}}\frac{1}{T}|\bR_{[0,T]}|>\delta\right)\leq {\rm P}\left(\sup_{2^{k}\leq T\leq 2^{k+1}}|\bR_{[0,T]}|>\delta 2^{k}\right)
\end{equation*}
\begin{equation*}
\leq \frac{E\left[\sup_{2^{k}\leq T\leq 2^{k+1}}|\bR_{[0,T]}|^{2}\right]}{\delta^{2} 2^{2k}}\leq \frac{E\left[|\bR_{[0,2^{k+1}]}|^{2}\right]}{\delta^{2}2^{2k}}
\end{equation*}
\begin{equation*}
\leq \frac{\max\left\{\sup_{t\geq 0}E\left[(X_{t})^{2}\right],K_{\bvarphi}^{2}\right\}2^{k+1}}{\delta^{2}2^{2k}}=O\left(\frac{1}{2^{k}}\right).
\end{equation*}
Applying the Borel–Cantelli Lemma, we obtain $\ds\frac{1}{T}|\bR_{[0,T]}|\xrightarrow[T\to\infty]{a.s.}0$. Then, we can conclude that  $\check\btheta_{T}\xrightarrow[T\to\infty]{a.s.}\btheta$.
The asymptotic normality is a special case of Proposition 1.21 in \cite{Kutoyants2004}, for which $d_{1}=p+1$ and $d_{2}=1$.
As defined in \eqref{Rmatrix},
\begin{equation*}
\bR_{[0,T]}=\begin{bmatrix} \ds\int_{0}^{T}\varphi_{1}(t)dB_{t},~\int_{0}^{T}\varphi_{2}(t)dB_{t},\cdots,\int_{0}^{T}\varphi_{p}(t)dB_{t},-\int_{0}^{T} X_{t}dB_{t},
 \end{bmatrix}^{\top}
\end{equation*}
which is a $p+1$ column vector. We denote $r_{T}^{(i)}(t,\omega)=\ds\frac{1}{T^{1/2} }\varphi_{i}(t)$ for $i=1,2,\cdots,p$ and $r_{T}^{(p+1)}(t,\omega)=\ds\frac{1}{T^{1/2} }{X_{t}}(\omega)$.
For all $i=1,2,\cdots,p$, we have $\left(r_{T}^{(i)}(t,\omega)\right)^{2}=\ds\frac{1}{T}\varphi_{i}^{2}(t)$ and $\left(r_{T}^{(p+1)}(t,\omega)\right)^{2}=\ds\frac{1}{T}(X_{t}(\omega))^{2}$.
By \eqref{ErgodicBound}, $\sup_{t\geq 0} E\left[| X_{t}|^{m}\right]<\infty,$
and $|\varphi_{i}(t)|\leq K_{\bvarphi}<\infty$, which implies that
${\rm P}\left(\ds\int_{0}^{T}\left(r_{T}^{(i)}(t,\omega)\right)^{2}dt<\infty\right)=1$ for all $i=1,2,\cdots,p+1$.
Since
\begin{equation*}
\frac{1}{T}\bQ_{[0,T]}=\begin{bmatrix}
  \ds\frac{1}{{T}}\int_{0}^{T}\bvarphi(t)\bvarphi^\top(t)dt & -\ds\frac{1}{{T}}\int_{0}^{T}\bvarphi(t){X_{t}}dt\\
  -\ds\frac{1}{{T}}\int_{0}^{T}\bvarphi^{\top}(t){X_{t}}dt ~~~& \ds\frac{1}{{T}}\int_{0}^{T}\left({X_{t}}\right)^{2}dt
\end{bmatrix}
\end{equation*}
and $\displaystyle\frac{1}{T}\bQ_{[0,T]}\xrightarrow[T\to\infty]{a.s.}\bSigma$.
Finally, we apply Proposition 1.21 in \cite{Kutoyants2004} to  get
\begin{equation*}
\frac{1}{T^{1/2} } \bR_{[0,T]}\xrightarrow[T\to\infty]{D}\bR^{*}\sim\mathcal{N}_{p+1}\left(0,\bSigma\right).
\end{equation*}
This result also follows from the central limit theorem for martingales in \cite{Remillard/Vaillancourt:2024a}, since the quadratic variation of the martingale $\frac{1}{t^{1/2} }\bR_{[0,t]}$, $t\in [0,T]$, converges to $\Sigma$ by Proposition \ref{ConvQSigma}.
Then, by Slutsky's Theorem,
\begin{equation}
T^{1/2} \left(\check\btheta_{T}-\btheta\right)=\sigma T \bQ_{[0,T]}^{-1}\displaystyle\frac{1}{T^{1/2} }\bR_{[0,T]}\xrightarrow[T\to\infty]{D}\sigma\bSigma^{-1}\bR^{*}.
\nonumber
\end{equation}
Note that $\Sigma^{-1}$ is non-random and symmetric, by the properties of multivariate normal distributions, we have $\rho\sim\mathcal{N}_{p+1}\left(0,\sigma^{2}\Sigma^{-1}\right)$. Then,
\begin{equation*}
T^{1/2} \left(\check\btheta_{T}-\btheta\right)\xrightarrow[T\to\infty]{D}\rho\sim\mathcal{N}_{p+1}\left(0,\sigma^{2}\bSigma^{-1}\right).
\end{equation*}
Finally, let $U^*$ be some positive constant. By Markov's inequality, we have
\begin{equation*}
\begin{aligned}
\dP\left(||\bR_{[0,T]}||\ge T^{1/2} U^*\right)\le &\frac{E\left[||\bR_{[0,T]}||^2\right]}{T(K_*)^2}\le \frac{\int_0^T(||\bvarphi(t)||^2+E\left[X_t^2\right])dt}{T(U^*)^2}\\
\le &\frac{(K_{\bvarphi}^2+\sup_{t\geq 0}E\left[(X_{t})^{2}\right])T}{T(U^*)^2}=\frac{K_{\bvarphi}^2+\sup_{t\geq 0}E\left[(X_{t})^{2}\right]}{(U^*)^2}.
\end{aligned}
\end{equation*}
This completes the proof.
\qed

\subsection{Proof of \Cref{prop:CN}}\label{app:pf0}
Suppose that $C(\Delta_N)>\epsilon_0>0$. The, because $\bvarphi$ is periodic, one can find a  sequence $k_N$ so that $t_{k_N} \to t$ such that
$$
 \frac{1}{\Delta_N} \int_{t_{k_N}}^{t_{k_N+1}}\|\bvarphi(s)-\bvarphi(t_{k_N})\|ds >\epsilon_0.
$$
Take a continuity point $t'$ of $\bvarphi$ so that $|t-t'|<\delta/2$ and $\|\bvarphi(s)-\bvarphi(t')\|< \epsilon_0/4$ whenever $|s-t'|<\delta$. Then, for any $s\in [t_{k_N},t_{k_N+1}]$, $|s-t'|\le \Delta_N +|t_{k_N}-t|+|t-t'|<\delta$
if $N$ is large enough, so
\begin{multline*}
 \epsilon_0 < \frac{1}{\Delta_N}\int_{t_{k_N}}^{t_{k_N+1}}\|\bvarphi(s)-\bvarphi(t_{k_N})\|ds
 \le
  2\sup_{s\in [t_{k_N},t_{k_N+1}]} \|\bvarphi(s)-\bvarphi(t')\|
    \le
 2\frac{\epsilon_0}{4}=\frac{\epsilon_0}{2},
\end{multline*}
which is impossible. Hence, $C(\Delta_N)\to 0$.
\qed

\section{Proofs for the observed discretized process}
\subsection{Proof of \Cref{prop:convdiscont}}\label{app:pf3}

\begin{equation*}
\frac{1}{T}\bZ_N^{\top}\bZ_N= \frac{\Delta_N}{T}
\begin{bmatrix}
\sum_{i=1}^N \bvarphi(t_{i-1}) \bvarphi^\top(t_{i-1})  & -\sum_{i=1}^N \bvarphi(t_{i-1}) X_{t_{i-1}}\\
 -\sum_{i=1}^N \bvarphi^{\top}(t_{i-1}) X_{t_{i-1}} & \sum_{i=1}^N  X_{t_{i-1}}^2\\
\end{bmatrix},
\end{equation*}
where $\Delta_N$ is defined in \Cref{hyp:N}.
We have
\begin{multline*}
\left\|\frac{\Delta_N}{T}\sum_{i=1}^N \bvarphi(t_{i-1}) \bvarphi^\top(t_{i-1})-\frac{1}{T}\int_0^T\bvarphi(t) \bvarphi^\top(t)dt\right\|\\
\leq\frac{1}{T}\sum_{i=1}^N \left\|\int_{t_{i-1}}^{t_{i}}\bvarphi(t_{i-1}) \bvarphi^\top(t_{i-1})dt-\int_{t_{i-1}}^{t_{i}}\bvarphi(t) \bvarphi^\top(t)dt\right\|\\
\leq\frac{1}{T}\sum_{i=1}^N \left\|\int_{t_{i-1}}^{t_{i}}(\bvarphi(t_{i-1}) -\bvarphi(t))\bvarphi^\top (t_{i-1})dt\right\|+\frac{1}{T}\sum_{i=1}^N \left\|\int_{t_{i-1}}^{t_{i}}\bvarphi(t) (\bvarphi(t_{i-1})-\bvarphi(t))^{\top}dt\right\|\\
\leq\frac{1}{T}\sum_{i=1}^N \int_{t_{i-1}}^{t_{i}}\left\|\bvarphi(t_{i-1}) -\bvarphi(t)\right\|\left\|\bvarphi(t_{i-1})\right\|dt+\frac{1}{T}\sum_{i=1}^N \int_{t_{i-1}}^{t_{i}}\left\|\bvarphi(t)\right\| \left\|\bvarphi(t_{i-1})-\bvarphi(t)\right\|dt\\
=\frac{1}{N}\sum_{i=1}^N \frac{1}{\Delta_N}\int_{t_{i-1}}^{t_{i}}\left\|\bvarphi(t_{i-1}) -\bvarphi(t)\right\|\left\|\bvarphi(t_{i-1})\right\|dt+\frac{1}{N}\sum_{i=1}^N \frac{1}{\Delta_N}\int_{t_{i-1}}^{t_{i}}\left\|\bvarphi(t)\right\| \left\|\bvarphi(t_{i-1})-\bvarphi(t)\right\|dt\\
\leq \frac{K_{\bvarphi}}{N}\sum_{i=1}^N \frac{1}{\Delta_N}\int_{t_{i-1}}^{t_{i}}\left\|\bvarphi(t_{i-1}) -\bvarphi(t)\right\|dt+\frac{K_{\bvarphi}}{N}\sum_{i=1}^N \frac{1}{\Delta_N}\int_{t_{i-1}}^{t_{i}}\left\|\bvarphi(t_{i-1})-\bvarphi(t)\right\|dt,
\end{multline*}
where from \eqref{kphi}, $\|\bvarphi(t)\|\le K_{\bvarphi}$ and $\|\bvarphi(t_{i-1})\|\le K_{\bvarphi}$. Furthermore, from \Cref{prop:CN}, we have
$\frac{1}{\Delta_N}\int_{t_{i-1}}^{t_{i}}\left\|\bvarphi(t_{i-1})-\bvarphi(t)\right\|dt\le  C(\Delta_N)$. Then,
\begin{equation}
\label{ineq1}
\left\|\frac{\Delta_N}{T}\sum_{i=1}^N \bvarphi(t_{i-1}) \bvarphi^{\top}(t_{i-1})-\frac{1}{T}\int_0^T\bvarphi(t) \bvarphi^{\top}(t)dt\right\|\leq 2K_{\bvarphi} C(\Delta_N).
\end{equation}
Next,
\begin{multline}
\label{ZQ_2}
\left\|\frac{\Delta_N}{T}\sum_{i=1}^N \bvarphi(t_{i-1}) X_{t_{i-1}}-\frac{1}{T}\int_{0}^{T}\bvarphi(t) X_{t}dt \right\|\\
\leq\frac{1}{T}\sum_{i=1}^N \left\|\int_{t_{i-1}}^{t_{i}}\bvarphi(t_{i-1}) X_{t_{i-1}}dt-\int_{t_{i-1}}^{t_{i}}\bvarphi(t) X_{t}dt \right\|\\
\leq\frac{1}{T}\sum_{i=1}^N \left\|\int_{t_{i-1}}^{t_{i}}(\bvarphi(t_{i-1}) -\bvarphi(t))X_{t_{i-1}}dt\right\|+\frac{1}{T}\sum_{i=1}^N \left\|\int_{t_{i-1}}^{t_{i}}\bvarphi(t) (X_{t_{i-1}}-X_{t})dt\right\|\\
\leq\frac{1}{T}\sum_{i=1}^N \int_{t_{i-1}}^{t_{i}}\left\|(\bvarphi(t_{i-1}) -\bvarphi(t))X_{t_{i-1}}\right\|dt+\frac{1}{T}\sum_{i=1}^N \int_{t_{i-1}}^{t_{i}}\left\|\bvarphi(t) (X_{t_{i-1}}-X_{t})\right\|dt\\
\leq\frac{1}{T}\sum_{i=1}^N \int_{t_{i-1}}^{t_{i}}\left\|\bvarphi(t_{i-1}) -\bvarphi(t)\right\|\left|X_{t_{i-1}}\right|dt+\frac{1}{T}\sum_{i=1}^N \int_{t_{i-1}}^{t_{i}}\left\|\bvarphi(t)\right\| \left|X_{t_{i-1}}-X_{t}\right|dt.
\end{multline}

From \eqref{ZQ_2}, \Cref{prop:CN}, and \eqref{ineqMay10_1} in \Cref{prop:XT}, we have
\begin{multline*}
E\left[\frac{1}{T}\sum_{i=1}^N \int_{t_{i-1}}^{t_{i}}\left\|\bvarphi(t_{i-1}) -\bvarphi(t)\right\|\left|X_{t_{i-1}}\right|dt+\frac{1}{T}\sum_{i=1}^N \int_{t_{i-1}}^{t_{i}}\left\|\bvarphi(t)\right\| \left|X_{t_{i-1}}-X_{t}\right|dt\right]\\
\leq \sup_{0\leq t\leq T}E[|X_t|]\frac{1}{N}\sum_{i=1}^N \frac{1}{\Delta_N}\int_{t_{i-1}}^{t_{i}}\left\|\bvarphi(t_{i-1}) -\bvarphi(t)\right\|dt+K_{\bvarphi}\frac{1}{T}\sum_{i=1}^N \int_{t_{i-1}}^{t_{i}}E\left[\left|X_{t_{i-1}}-X_{t}\right|\right]dt\\
\leq C(\Delta_N)\sup_{0\leq t\leq T}E[|X_t|]+K_{\bvarphi}\frac{1}{T}\sum_{i=1}^N \int_{t_{i-1}}^{t_{i}}\left\{ \sigma\sqrt{(t-t_{i-1})}+O((t-t_{i-1})\right\}dt\\
=C(\Delta_N)\sup_{0\leq t\leq T}E[|X_t|]+K_{\bvarphi}\sigma \frac{1}{T}\sum_{i=1}^N \frac{2}{3}\Delta_N^{3/2}+O(\Delta_N)\\
=C(\Delta_N)\sup_{0\leq t\leq T}E[|X_t|]+\frac{2}{3}K_{\bvarphi}\sigma\Delta_N^{1/2} +O(\Delta_N).
\end{multline*}
This implies that
\begin{multline}
\label{ineq2}
E\left[\left\|\frac{\Delta_N}{T}\sum_{i=1}^N \bvarphi(t_{i-1}) X_{t_{i-1}}-\frac{1}{T}\int_{0}^{T}\bvarphi(t) X_{t}dt \right\|\right]\\
\leq C(\Delta_N)\sup_{0\leq t\leq T}E[|X_t|]+\frac{2}{3}K_{\bvarphi}\sigma\Delta_N^{1/2} +O(\Delta_N).
\end{multline}
In addition, we have
\begin{multline*}
\left|\frac{\Delta_N}{T}\sum_{i=1}^N  X_{t_{i-1}}^2-\frac{1}{T}\int_{0}^{T} X_{t}^2dt\right|=\left|\frac{1}{T}\sum_{i=1}^N  \int_{t_{i-1}}^{t_{i}}X_{t_{i-1}}^2dt-\frac{1}{T}\sum_{i=1}^N \int_{t_{i-1}}^{t_{i}} X_{t}^2dt\right|
\\
=\left|\frac{1}{T}\sum_{i=1}^N  \int_{t_{i-1}}^{t_{i}}(X_{t_{i-1}}^2-X_{t}^2)dt\right|=\left|\frac{1}{T}\sum_{i=1}^N  \int_{t_{i-1}}^{t_{i}}(X_{t_{i-1}}-X_{t})(X_{t_{i-1}}+X_{t})dt\right|\\
\leq\frac{1}{T}\sum_{i=1}^N  \int_{t_{i-1}}^{t_{i}} |X_t - X_{t_{i-1}}|(|X_{t_{i-1}}|+|X_{t}|)dt\\
=\frac{1}{T}\sum_{i=1}^N  \int_{t_{i-1}}^{t_{i}} |X_t - X_{t_{i-1}}||X_{t_{i-1}}|dt+\frac{1}{T}\sum_{i=1}^N  \int_{t_{i-1}}^{t_{i}} |X_t - X_{t_{i-1}}|(|X_{t}|)dt.
\end{multline*}
Then, by Cauchy–Schwartz inequality,
\begin{multline*}
E\left[\frac{1}{T}\sum_{i=1}^N  \int_{t_{i-1}}^{t_{i}} |X_t - X_{t_{i-1}}||X_{t_{i-1}}|dt\right]\leq \frac{1}{T}\sum_{i=1}^N  \int_{t_{i-1}}^{t_{i}} \left\{E\left[|X_t - X_{t_{i-1}}|^2\right]E\left[|X_{t_{i-1}}|^2\right]\right\}^{1/2}dt\\
\leq \sqrt{C_2} \frac{1}{T}\sum_{i=1}^N  \int_{t_{i-1}}^{t_{i}} \left\{E\left[|X_t - X_{t_{i-1}}|^2\right]\right\}^{1/2}dt\leq \sqrt{C_2}\sigma\Delta_N^{1/2} +O(\Delta_N).
\end{multline*}
Similarly,
\begin{multline*}
E\left[\frac{1}{T}\sum_{i=1}^N  \int_{t_{i-1}}^{t_{i}} |X_t - X_{t_{i-1}}|(|X_{t}|)dt\right]\leq \frac{1}{T}\sum_{i=1}^N  \int_{t_{i-1}}^{t_{i}} \left\{E\left[|X_t - X_{t_{i-1}}|^2\right]E\left[|X_{t}|^2\right]\right\}^{1/2}dt\\
\leq \sqrt{C_2} \frac{1}{T}\sum_{i=1}^N  \int_{t_{i-1}}^{t_{i}} \left\{E\left[|X_t - X_{t_{i-1}}|^2\right]\right\}^{1/2}dt\leq \sqrt{C_2}\sigma\Delta_N^{1/2} +O(\Delta_N).
\end{multline*}
Finally, we get
\begin{equation}
\label{ineq3}
E\left[\left|\frac{\Delta_N}{T}\sum_{i=1}^N  X_{t_{i-1}}^2-\frac{1}{T}\int_{0}^{T} X_{t}^2dt\right|\right]\leq 2\sqrt{C_2}\sigma\Delta_N^{1/2} +O(\Delta_N).
\end{equation}
\eqref{ineq1}, \eqref{ineq2}, and \eqref{ineq3} imply that
\begin{multline*}
E\left[\left\|\bSigma_N-\frac{1}{T}\bQ_{[0,T]}\right\|\right]\leq 2K_{\bvarphi} C(\Delta_N)+C(\Delta_N)\sup_{0\leq t\leq T}E[|X_t|]+\frac{2}{3}K_{\bvarphi}\sigma\Delta_N^{1/2} \\
+2\sqrt{C_2}\sigma\Delta_N^{1/2} +O(\Delta_N).
\end{multline*}
Since $\ds C_3=2K_{\bvarphi}+C_1, C
_4=\frac{2}{3}K_{\bvarphi}\sigma+2\sqrt{C_2}\sigma$, we have
\begin{equation*}
E\left[\left\|\bSigma_N-\frac{1}{T}\bQ_{[0,T]}\right\|\right]\leq C_3C(\Delta_N)+C_4\Delta_N^{1/2} +O(\Delta_N).
\end{equation*}
This completes the proof of the first statement.

Furthermore,
\begin{multline*}
\left\|\frac{1}{T^{1/2}}\sum_{i=1}^{N}\bZ_{N,i}\epsilon_{N,i} -\frac{\sigma}{T^{1/2}}\bR_{[0,T]}\right\|\\
=\left\|\frac{1}{T^{1/2}}\sum_{i=1}^N \sigma\int_{t_{i-1}}^{t_{i}}(\bvarphi^\top (t_{i-1}),-X_{t_{i-1}})^\top dB_t
-\frac{1}{T^{1/2}}\sum_{i=1}^N \sigma\int_{t_{i-1}}^{t_{i}}(\bvarphi^\top (t),-X_{t})^\top dB_t\right\|\\
\leq\frac{\sigma}{T^{1/2}}\sum_{i=1}^N \left\|\int_{t_{i-1}}^{t_{i}}\left((\bvarphi^\top (t_{i-1}),-X_{t_{i-1}})^\top-(\bvarphi^\top (t),-X_{t})^\top)\right)dB_t\right\|\\
=\frac{1}{T^{1/2}}\sum_{i=1}^N \sigma\left\|\int_{t_{i-1}}^{t_{i}}\left((\bvarphi(t_{i-1})-\bvarphi(t))^\top,-(X_{t_{i-1}}-X_t)\right)^\top dB_t\right\|.
\end{multline*}
From It\^o's isometry, we have
\begin{multline*}
\left(E\left[\left\|\int_{t_{i-1}}^{t_{i}}\left((\bvarphi(t_{i-1})-\bvarphi(t)),-(X_{t_{i-1}}-X_t)\right)dB_t\right\|\right]\right)^2
\\
\leq E\left[\left\|\int_{t_{i-1}}^{t_{i}}\left((\bvarphi(t_{i-1})-\bvarphi(t)),-(X_{t_{i-1}}-X_t)\right)dB_t\right\|^2\right]\\
=E\left[\int_{t_{i-1}}^{t_{i}}\left\|(\bvarphi(t_{i-1})-\bvarphi(t),-(X_{t_{i-1}}-X_t))^{\top}(\bvarphi(t_{i-1})-\bvarphi(t),-(X_{t_{i-1}}-X_t))\right\|dt\right]\\
\leq \int_{t_{i-1}}^{t_{i}}\left\|(\bvarphi(t_{i-1})-\bvarphi(t))^{\top}(\bvarphi(t_{i-1})-\bvarphi(t))\right\|dt
+E\left[\int_{t_{i-1}}^{t_{i}}\left\|X_{t_{i-1}}-X_t\right\|dt\right]\\
\leq \left(C^2(\Delta_N)+2C(\Delta_N)C^*\sqrt{\Delta_N+C(\Delta_N)}+(C^*\sqrt{\Delta_N+C(\Delta_N)})^2\right)\Delta_N.
\end{multline*}
Then,
\begin{equation}
\label{ineq4}
\begin{aligned}
 &E\left[\left\|\frac{1}{T^{1/2}}\sum_{i=1}^{N}\bZ_{N,i}\epsilon_{N,i} -\frac{\sigma}{T^{1/2}}\bR_{[0,T]}\right\|\right]\\
 &\leq \sigma\sqrt{C^2(\Delta_N)+2C(\Delta_N)C^*\sqrt{\Delta_N+C(\Delta_N)}+(C^*\sqrt{\Delta_N+C(\Delta_N)})^2}.
\end{aligned}
\end{equation}
Finally, using \eqref{eq:rn},
\begin{eqnarray*}
E\left(\left\| \sum_{i=1}^N \bZ_{N,i} r_{N,i}\right\| \right)
&\le &
 \|\bmu\| \sum_{i=1}^N
\int_{t_{i-1}}^{t_{i}} \left\{ \|\bvarphi(t_{i-1})\|+E(|X_{t_{i-1}}|)\right\} \|\bvarphi(s)-\bvarphi(t_{i-1})\| ds\\
&& \qquad +a \sum_{i=1}^N \|\bvarphi(t_{i-1})\|  \int_{t_{i-1}}^{t_{i}}E\left\{|X_s-X_{t_{i-1}}|\right\} ds\\
&& \qquad \qquad +a \sum_{i=1}^N  \int_{t_{i-1}}^{t_{i}}E \left\{|X_{t_{i-1}}||X_s-X_{t_{i-1}}|\right\} ds\\
&\le & T \|\bmu\|( K_\bvarphi+C_1)  C(\Delta_N) + T a K_\bvarphi \left\{O\left(\Delta_N^{1/2} \right)+O\left(\Delta_N \right)\right\} \\
&& \qquad \qquad \quad +a C_2^{1/2} \sum_{i=1}^N \int_{t_{i-1}}^{t_{i}} \left[ E\left\{ \left(X_s-X_{t_{i-1}}\right)^2\right\} \right]^{1/2}ds\\
&\le & T \|\bmu\|( K_\bvarphi+C_1) \left\{C(\Delta_N) + O\left(\Delta_N^{1/2} \right) + O\left(\Delta_N^{3/4} \right)
+ O\left(\Delta_N \right) \right\}.
\end{eqnarray*}
\qed

\subsection{Proof of \Cref{thm:main1}}\label{app:pf-thm1}
Recall that $T=t_N=N\Delta_N$.
First, we prove that if
$$ W(N,K)=\sum_{N<i\le N+K}\hat{\epsilon}_{N,i}-\left(\sum_{N<i\le N+K}\epsilon_{N,i}-\frac{K}{N}\sum_{0<i\le N}\epsilon_{N,i}\right),
$$
then
\begin{equation}
    \label{stat_0}
      \sup_{1\le K<\infty}  \frac{\ds\left|W(N,K) \right|}{g_1(N,K)}=o_P(1).
    \end{equation}
Note that using \eqref{eq:thetaN}, one gets
\begin{multline*}
    \sum_{N<i\le N+K}\hat{\epsilon}_{N,i}=\sum_{N<i\le N+K}(Y_{N,i}-\bZ_{N,i}^\top \hat\btheta_N)
    =\sum_{N<i\le N+K}\left\{Y_{N,i}-\bZ_{N,i}^\top \btheta_0-\bZ_{N,i}^\top \left(\hat\btheta_N-\btheta_0\right)\right\}\\
=\sum_{N<i\le N+K}{\epsilon}_{N,i} + \sum_{N<i\le N+K} r_{N,i} -  \frac{1}{T}\sum_{N<i\le N+K}\bZ_{N,i}^\top \bSigma_N^{-1} \bZ_N^{\top}(\bepsilon_N +\br_N),
\end{multline*}
so
\begin{multline*}
W(N,K)
=
\frac{K}{N}\sum_{0<i\le N}\epsilon_{N,i} + \sum_{N<i\le N+K} r_{N,i}-\frac{1}{T}\sum_{N<i\le N+K}\bZ_{N,i}^\top \bSigma_N^{-1}\bZ_N^{\top}( \bepsilon_N +\br_N ).
\end{multline*}
Let $\mathcal{C}_1$ be the first column of the matrix $\bSigma$. Then, $\mathcal{C}_1^\top \bSigma^{-1}=(1,0,\cdots,0)$.
From \Cref{hyp:2}, $\varphi_1(t)=1$, we  have
$\ds \sum_{0<i\le N}\epsilon_{N,i}= \mathcal{C}_1^\top \bSigma^{-1} \frac{\bZ_N^{\top}\bepsilon_N}{{\Delta_N^{1/2}}}$.
This leads to
$
|W(N,K)| \le |W(N,K,1)| +|W(N,K,2)| $,
where
$$
W(N,K,1) = \left( \frac{\Delta_N^{1/2}}{T} \sum_{N<i\le N+K}\bZ_{N,i}^\top \bSigma_N^{-1}-\frac{K}{N}\mathcal{C}_1^\top \bSigma^{-1} \right) \frac{\bZ_N^{\top}\bepsilon_N}{{\Delta_N^{1/2}}}
$$
and
$$
W(N,K,2) = \frac{1}{T}\sum_{N<i\le N+K}\bZ_{N,i}^\top \bSigma_N^{-1} \bZ_N^{\top} \br_N -\sum_{N<i\le N+K} r_{N,i}.
$$
Recall from \eqref{eq:Z} that
$\bZ_{N,i}= \Delta_N^{1/2}(1,\varphi_{2}(t_{i-1}),\dots,\varphi_{p}(t_{i-1}),-X_{t_{i-1}})^\top $.
Hence, the first column of the matrix
$\ds \bSigma_N = \frac{1}{T}\sum_{i=1}^{N}\bZ_{N,i}\bZ_{N,i}^\top $ is  $\ds \frac{\Delta_N^{1/2}}{T}\sum_{i=1}^{N}\bZ_{N,i} $. From \Cref{ConvQSigma} and \Cref{prop:convdiscont}, we have
$\ds \bSigma_N =\bSigma+o_P(1)$, as $T\to\infty$,
which implies that
$\ds \frac{\Delta_N^{1/2}}{T}\sum_{i=1}^{N}\bZ_{N,i} = \mathcal{C}_1 +o_P(1)$.
In fact, setting $\bZ_{N,i,1}=\Delta_N^{1/2} \bvarphi(t_{i-1})$ and $\ds \bar \bvarphi=\int_0^1 \bvarphi(s)ds=\mathbf{e}_1$ by Remark \ref{rem:a}, we have
\begin{multline*}
    \frac{\Delta_N^{1/2}}{T}\sum_{i=N+1}^{N+K}\bZ_{N,i,1}-\frac{K}{N}\bar \bvarphi\\
    =
    \frac{1}{T}\sum_{i=N+1}^{N+K}\int_{t_{i-1}}^{t_{i-1}} (\bvarphi(t_{i-1})-\bvarphi(s))ds +\frac{1}{T}\int_T^{T(1+K/N)}\bvarphi(s)ds-\frac{K}{N}\bar \bvarphi\\
   =
    \frac{1}{T}\sum_{i=N+1}^{N+K}\int_{t_{i-1}}^{t_{i-1}} (\bvarphi(t_{i-1})-\bvarphi(s))ds +
    \left(\frac{{\lfloor K\Delta_N)\rfloor}}{T} -\frac{K\Delta_N}{T}\right)\bar \bvarphi\\
    +\frac{1}{T}\int_{T+ \lfloor K\Delta_N\rfloor}^{T+K\Delta_N}\bvarphi(s)ds,
  \end{multline*}
  so
  \begin{equation}\label{eq:sumZ1}
  \left\| \frac{\Delta_N^{1/2}}{T}\sum_{i=N+1}^{N+K}\bZ_{N,i,1}-\frac{K}{N}\bar \bvarphi\right\|
  \le
 \frac{K}{N} C(\Delta_N)+ \frac{(\|\bar \bvarphi\|+K_\varphi)}{T}.
  \end{equation}
  Next, if $\bZ_{N,i,2} = - \Delta_N^{1/2}X_{t_{i-1}}$ and $b=-\bmu^\top \bar \bvarphi/a$, we have
  \begin{multline*}
    \frac{\Delta_N^{1/2}}{T}\sum_{i=N+1}^{N+K}\bZ_{N,i,2}-\frac{K}{N} b \\
    =
    - \frac{1}{T}\sum_{i=N+1}^{N+K}\int_{t_{i-1}}^{t_{i-1}} (X_{t_{i-1}}-X_s)ds -\frac{1}{T}\int_T^{T(1+K/N)}X_s ds - \frac{K}{N}b\\
   =  - \frac{1}{ T}\sum_{i=N+1}^{N+K}\int_{t_{i-1}}^{t_{i-1}} (X_{t_{i-1}}-X_s)ds +\frac{1}{a T} \left(X_{T(1+K/N)}-X_T\right)\\
   -\frac{1}{T}\int_T^{T(1+K/N)}\frac{\bmu^\top \bvarphi(s)}{a} ds -\frac{\sigma}{a T}(B_{T(1+K/N)}-B_T)- \frac{K}{N} b\\
   =  - \frac{1}{T}\sum_{i=1}^{N+K}\int_{t_{i-1}}^{t_{i-1}} (X_{t_{i-1}}-X_s)ds +\frac{1}{a T} \left(X_{T(1+K/N)}-X_T\right) -\frac{\sigma}{a T}(B_{T(1+K/N)}-B_T)\\
  +\left(\frac{{\lfloor K\Delta_N)\rfloor}}{T} -\frac{K\Delta_N}{T}\right)b
    +\frac{1}{a T}\int_{T+\lfloor K\Delta_N\rfloor}^{T(1+K/N)}\bmu^\top \bvarphi(s)ds,
  \end{multline*}
   so, using \eqref{ineqMay10_1}, we get
  \begin{multline}\label{eq:sumZ2}
 \mathbf{E}\left[ \left| \frac{\Delta_N^{1/2}}{T}\sum_{i=N+1}^{N+K}\bZ_{N,i,2}-\frac{K}{N}b\right|\right]
  \le \frac{K}{N}C_3 \Delta_N^{1/2}+ \frac{C_3}{a} \left(\frac{K/N}{T}\right)^{1/2}\\
  +   \frac{(|b|+2C_1+\|\bmu\|K_\varphi)/a}{T}.
  \end{multline}
Then,
\eqref{eq:sumZ1} and \eqref{eq:sumZ2} together imply that
 \begin{multline}\label{eq:sumTotal}
 \mathbf{E}\left[ \left\| \frac{\Delta_N^{1/2}}{T}\sum_{i=N+1}^{N+K}\bZ_{N,i}-\frac{K}{N}\mathcal{C}_1\right\|\right]
  \le \frac{K}{N}C_3 \Delta_N^{1/2}+ \frac{C_3}{a} \left(\frac{K/N}{T}\right)^{1/2}\\
  +   \frac{(|b|+2C_1+\|\bmu\|K_\varphi)/a}{T}+\frac{K}{N} C(\Delta_N)+ \frac{(\|\bar \bvarphi\|+K_\varphi)}{T}.
  \end{multline}
Next, we have $W(N,K,1) =W(N,K,1,1) +W(N,K,1,2) $, where
$$
W(N,K,1,1) = \left(\frac{\Delta_N^{1/2}}{T}\sum_{N<i\le N+K}\bZ_{N,i} - \frac{K}{N}\mathcal{C}_1\right)^\top \bSigma_N^{-1}\frac{\bZ_N^{\top}\bepsilon_N}{\Delta_N^{1/2}}
$$
and
$$
W(N,K,1,2) =  \frac{K}{N}\mathcal{C}_1^\top  \left(\bSigma_N^{-1}-\bSigma^{-1}\right)\frac{\bZ_N^{\top}\bepsilon_N}{{\Delta_N^{1/2}}}.
$$
\eqref{ineq4} implies that $N^{-1/2}\left\|\frac{\bZ_N^\top \bepsilon_N}{\Delta_N^{1/2}}\right\| = O_P(1)$. \Cref{ConvQSigma}, \Cref{prop:convdiscont}, and \Cref{prop:inverse} together imply that $\left\|\bSigma_N^{-1}-\bSigma^{-1}\right\|=o_P(1)$. It then follows that
\begin{multline*}
\sup_{1\le K<\infty}\left\|\frac{W(N,K,1,2)}{g_1(N,K)}\right\|=\sup_{K\ge 1} \frac{\left\| \frac{K}{N}\mathcal{C}_1^\top  \left(\bSigma_N^{-1}-\bSigma^{-1}\right)\frac{\bZ_N^{\top}\bepsilon_N}{\Delta_N^{1/2}} \right\| }{g_1(N,K)} \\
\le \sup_{1\le K<\infty}\frac{K/N}{(1+K/N)(\frac{K}{N+K})^{\gamma}}\left\|\bSigma_N^{-1}-\bSigma^{-1}\right\| N^{-1/2}\left\|\frac{\bZ_N^\top \bepsilon_N}{\Delta_N^{1/2}}\right\|\\
= \left\|\bSigma_N^{-1}-\bSigma^{-1}\right\|O_P(1)=o_P(1).\\
\end{multline*}
From  \Cref{eq:sumTotal}, since $0\le \gamma <1/2$,
we have
\begin{multline*}
\sup_{1\le K<\infty}
\left\|\frac{W(N,K,1,1)}{g_1(N,K)}\right\| = \sup_{1\le K<\infty}\frac{\left\|\left(\frac{\Delta_N^{1/2}}{T}\sum_{N<i\le N+K}\bZ_{N,i} - \frac{K}{N}\mathcal{C}_1\right)^\top \bSigma_N^{-1}\frac{\bZ_N^{\top}\bepsilon_N}{\Delta_N^{1/2}}\right\|}{g_1(N,K)}\\
=O_P(1) \sup_{1\le K<\infty}\frac{\left\|\left(\frac{\Delta_N^{1/2}}{T}\sum_{N<i\le N+K}\bZ_{N,i} - \frac{K}{N}\mathcal{C}_1\right)^\top\right\|}{(1+K/N)(\frac{K}{N+K})^{\gamma}}\\
\le O_P(1)\left\{\Delta_N^{1/2} + C(\Delta_N) + T^{-1/2} \right\} + \frac{1}{T} O_P(1) \sup_{k\ge 1} (k/N)^{-\gamma}\\
=
O_P(1)\left\{\Delta_N^{1/2} + C(\Delta_N) + T^{-1/2}  + \frac{N^\gamma}{T}\right\}.
\end{multline*}
\Cref{hyp:N} implies that $2 < \delta <\frac{1}{\gamma}$, and $ \frac{N^\gamma}{T} = \frac{{\lfloor T^{\delta}\rfloor}^\gamma}{T} \le T^{\delta\gamma-1} \to 0$ as $T\to\infty$, since $\delta\gamma <1$.
Hence, one may conclude that
\begin{equation*}
\sup_{1\le K<\infty}\frac{|W(N,K,1)|}{g_1(N,K)}=o_P(1),
\end{equation*}
as $T\to\infty$.
Next, $W(N,K,2) = W(N,K,2,1)+W(N,K,2,2)+W(N,K,2,3)-W(N,K,2,4)$, where
$$
W(N,K,2,1) =
 \left(\frac{\Delta_N^{1/2}}{T}\sum_{N<i\le N+K}\bZ_{N,i} - \frac{K}{N}\mathcal{C}_1\right)^\top \bSigma_N^{-1}\frac{\bZ_N^{\top}\br_N}{\Delta_N^{1/2}},
$$
$$
W(N,K,2,2) =  \frac{K}{N}\mathcal{C}_1^\top  \left(\bSigma_N^{-1}-\bSigma^{-1}\right)\frac{\bZ_N^{\top}\br_N}{{\Delta_N^{1/2}}},
$$
$\ds W(N,K,2,3) = \frac{K}{N}\sum_{i=1}^N r_{N,i}$,
and $\ds W(N,K,2,4) = \sum_{N<i\le N+K} r_{N,i}$. It follows from Proposition \ref{prop:convdiscont} that $\bZ_N^\top r_N = T \left\{C(\Delta_N)+\Delta_N^{1/2}\right\} O_P(1)$. Using \eqref{eq:sumTotal} and the previous calculations, one gets that
$$
\sup_{K\ge 1}\frac{\left|W(N,K,2,1)\right|}{g_1(N,K)} = \frac{T}{N^{1/2}}\left\{C(\Delta_N)+\Delta_N^{1/2}\right\}  \left\{\Delta_N^{1/2}+ T^{-1/2} +N^\gamma T^{-1} +C(\Delta_N) \right\}
O_P(1)=o_P(1),
$$
since $\frac{1}{\gamma}>\delta>2$. Also,
$$
\sup_{K\ge 1}\frac{\left|W(N,K,2,2)\right|}{g_1(N,K)} = T^{1/2}\left\{C(\Delta_N)+\Delta_N^{1/2}\right\}  o_P(1) = o_P(1),
$$
by Assumption \ref{hyp:phiextra}.
Next, from Proposition \ref{prop:convdiscont} and \eqref{eq:rN20},
$$
\sup_{K\ge 1}\frac{\left|W(N,K,2,3)\right|}{g_1(N,K)} \le \left|\frac{1}{N^{1/2}}\sum_{i=1}^N r_{N,i}\right|  =T^{1/2} \left\{C(\Delta_N)+\Delta_N^{1/2} \right\}O(1)= o_P(1),
$$
if $\delta >  2$. Furthermore,
$$
\sup_{K\ge 1}\frac{\left|W(N,K,2,4)\right|}{g_1(N,K)} \le \left|\frac{1}{N^{1/2}}\sum_{i=N+1}^{N+K} r_{N,i}\right| \le
N^{1/2} \Delta_N^{1/2}  \left\{C(\Delta_N)+\Delta_N^{1/2} +\Delta_N \right\}O(1)= o_P(1),
$$
if $\delta > 1+\frac{1}{2\ell_*}\ge 2$.
This completes the proof of \eqref{stat_0}. Finally, for $t\ge 0$,
set
\begin{eqnarray*}
W_N(t) &=& \frac{1}{\sigma N^{1/2}(1+\N{t})} \left\{ \sum_{i=N+1}^{N+\N{t}} \epsilon_{N,i} - \frac{\N{t}}{N}\sum_{i=1}^{N} \epsilon_{N,i}\right\} \\
&= & \frac{  B_N\left(1+ \N{t}/N\right)-B_N(1)-\frac{\N{t}}{N}B_N(1)}{1+\N{t}},
\end{eqnarray*}
where $\ds B_N(t) = \frac{1}{\sigma N^{1/2} }\sum_{i=1}^{\N{t}} \epsilon_{N,i}$. Then, $B_N$ converges in $C[0,\infty)$ to a Brownian motion $B$, so by the continuous mapping theorem, $W_N$ converges in $C[0,\infty)$ to $W$, where $W(t) = \frac{B(1+t)-(1+t)B(1)}{1+t} = \dB_1 \left(\frac{t}{1+t}\right)$, where $\dB_1$ is a Brownian motion. As a result,
$\ds
\sup_{K\ge 1} \frac{\left| \ds \sum_{i=N+1}^{N+\N{t}} \epsilon_{N,i} - \frac{\N{t}}{N}\sum_{i=1}^{N} \epsilon_{N,i}\right|}{\sigma g_1(N,K) }
$ converges in law to $\ds \sup_{t\in (0,1]}\frac{|\dB_1(t)|}{t^\gamma}$. This completes the proof.
\qed

\subsection{Proof of \Cref{ThmAlt}}\label{app:pf-thm1-alt}

For $K >K_*$, and $\bZ_{N,i}^{(2)}$ defined in \eqref{eq:Z_1},
we have
\begin{eqnarray*}
\sum_{i=N+1}^{N+K }\hat{\epsilon}_{N,i}&=& \sum_{i=N+1}^{N+K }\left(\epsilon_{N,i} +r_{N,i}^{(2)} \right) -\left(\sum_{i=N+1}^{N+K_*}\bZ_{N,i}\right)^\top(\hat\btheta_N-\btheta_0)\\
&& \qquad -\left(\sum_{i=N+K_*+1}^{N+K }\bZ_{N,i}^{(2)}\right)^\top(\hat\btheta_N-\btheta_*)\\
  &=& \sum_{i=N+1}^{N+K }\epsilon_{N,i} +   \sum_{i=N+1}^{N+K}r_{N,i}^{(2)}
   -\left(\sum_{i=N+1}^{N+K_*}\bZ_{N,i}\right)^\top(\hat\btheta_N-\btheta_0)\\
   && \qquad -\left(\sum_{i=N+K_*+1}^{N+K }\bZ_{N,i}^{(2)}\right)^\top(\hat\btheta_N-\btheta_0) +\left(\sum_{i=N+K_*+1}^{N+K }\bZ_{N,i}^{(2)}\right)^\top(\btheta_*-{\btheta}_0).
\end{eqnarray*}
From the proof of
\Cref{thm:main1}, we have
$$
\lim_{N\to\infty} \sup_{K> K_*}\frac{\ds \left|\sum_{i=N+1}^{N+K}\epsilon_{N,i}\right|}{g_1(N,K)}
\le \sigma \sup_{t\ge  t_*} \frac{|\dB_1(t)|}{t^\gamma}={\rm O}_P(1),
$$
where $\dB_1$ is a Brownian motion,
and
$$
\lim_{N\to\infty} \sup_{K> K_*}\frac{\ds \left|\sum_{i=N+1}^{N+K}r_{N,i}^{(2)}\right|}{g_1(N,K)}
={\rm o}_P(1).
$$
Next, because of \eqref{eq:sumTotal} and $T^{1/2}\left(\hat \btheta_N-\btheta_0\right) = {\rm O}_P(1)$, we have
\begin{multline*}
\lim_{N\to \infty} \sup_{K>K_*} \frac{\ds T^{-1/2} \left|\sum_{i=N+1}^{N+K_*}\bZ_{N,i}\right|} {g_1(N,K)}
\le \frac{t_* \|\mathcal{C}_1\|}{(1+t_*)^{1-\gamma} t_*^\gamma}+{\rm o}_P(1) ={\rm O}_P(1),
\end{multline*}
proving that
\begin{equation*}
\lim_{N\to\infty} \sup_{K>K_*}\frac{\ds \left| \left( \sum_{i=N+1}^{N+K_*}\bZ_{N,i}\right)^\top(\hat\btheta_N-\btheta_0)\right|}{g_1(N,K)}={\rm O}_P(1).
\end{equation*}
Similarly, using again  \eqref{eq:sumTotal} and $T^{1/2}\left(\hat \btheta_N-\btheta_0\right) = {\rm O}_P(1)$, we have
\begin{multline*}
\lim_{N\to\infty}\sup_{K>K_*}\frac{T^{-1/2} \ds \left| \sum_{i=N+K_*+1}^{N+K}\bZ_{N,i}^{(2)} \right|}{g_1(N,K)}\le  {\rm O}_P(1)  +{\rm o}_P(1)={\rm O}_P(1),
\end{multline*}
proving that
\begin{equation*}
\lim_{N\to\infty}\sup_{K>K_*}\frac{\left|\left( \sum_{i=N+1}^{N+K_*}\bZ_{N,i}^{(2)} \right)^\top(\hat\btheta_N-\btheta_0)\right|}{g_1(N,K)}={\rm O}_P(1).
\end{equation*}
Let $\mathcal{C}_{1*}$ be the first column of the matrix $\bSigma_*$ defined by
\begin{equation}
\label{sigma*}
\bSigma_*=\begin{bmatrix}
    I_p & \bLambda_*\\
    \bLambda^{\top}_* & \omega_*
\end{bmatrix},\quad \bLambda_* = -\int_{0}^{1} \widetilde{h}_*(t)\bvarphi(t)dt, \quad \omega_* = \int_{0}^{1} \widetilde{h}^{2}_*(t)dt + \frac{\sigma^2}{2a_*},
\end{equation}
where
$\displaystyle
\tilde{h}_*(t)=e^{-a_* t}\sum_{k=1}^{p}\mu_{k*}\int_{-\infty}^{t}e^{a_* s}\varphi_{k}(s)ds$, and $\bmu_{*},a_*$ are the parameters after the change-point $K_*$.
Furthermore, \Cref{eq:sumTotal} indicates that for $K/N = t >t_*$,
$$
\sum_{i=N+K_*+1}^{N+K }\bZ_{N,i}^{(2)}- T \Delta_N^{-1/2} (t -t_*)\mathcal{C}_{1*}
= \left( T t + N^{1/2} t^{1/2} + \Delta_N^{-1/2} + C(\Delta_N) t T \Delta_N^{-1/2}\right) {\rm O}_P(1)  ,
$$
so
\begin{equation*}
\lim_{N\to\infty} T^{-1/2} \frac{\ds \sum_{i=N+K_*+1}^{N+\lfloor Nt\rfloor }\bZ_{N,i}^{(2)} }{g_1(N,\lfloor Nt\rfloor)} = \frac{(t-t_*)}{(1+t)\ell_\gamma(t)}\mathbf{\mathcal{C}}_{1*} .
\end{equation*}
Finally,
setting
\begin{equation}\label{eq:kappastar0}
\kappa_{\hat Q,t_*,t} = \frac{(t-t_*)}{(1+t)\ell_\gamma(t)}  (\btheta_*-\btheta_0)^\top   \mathbf{\mathcal{C}}_{1*} ,
\end{equation}
one gets that by hypothesis, $\ds \lim_{t\to\infty} \kappa_{\hat Q,t_*,t}  = \kappa_{\hat Q,\btheta_0,\btheta_*}  \neq 0$.
This completes the proof.
\qed

\subsection{Proof of \Cref{thm:main12}}\label{app:pf-thm2}
By \eqref{eq:thetaN}, we have
$$
T^{1/2}\left(\hat \btheta_{N+\N{t}}-\hat \btheta_N\right) = T^{1/2}\left(\hat \btheta_{N+\N{t}}-\btheta_0\right)-T^{1/2}\left(\hat \btheta_N-\btheta_0\right).
$$
Setting $A_{N+\N{t}} = \frac{1}{t_{N+\N{t}}} \bZ_{N+\N{t}}^\top \bZ_{N+\N{t}}$, it follows from \Cref{prop:convdiscont} that
$ \ds A_{N+\N{t}} = \frac{ \bQ_{\left[0,t_{N+\N{t}}\right]}}{t_{N+\N{t}}}+o_P(1)$
and
$$
   T^{1/2}\left(\hat \btheta_{N+\N{t}}-\btheta_0\right)
= \frac{T}{t_{N+\N{t}}} A_{N+\N{t}}^{-1}T^{-1/2} \bR_{[0,t_{N+\N{t}}]}+o_P(1).
$$
By using the Central Limit Theorem for continuous martingales in \cite{Remillard/Vaillancourt:2024a}, we have
\begin{equation*}
T^{1/2}\left(\hat \btheta_{N+\N{t}}-\btheta_0\right) \xrightarrow[N\to\infty]{D}\bSigma^{-1/2} \frac{\dB_{p+1}(1+t)}{1+t},
\end{equation*}
where $\dB_{p+1}$  is a standard $(p+1)$-dimensional Brownian motion, and $D$ denotes convergence in the space of continuous functions on $[0,\infty)$. As a result,
$$
\hat \bGamma (N,\N{t}) \xrightarrow[N\to\infty]{D} \frac{\dB_{p+1}(1+t) - (1+t)\dB_{p+1}(1)}{1+t} = \dB_{p+1}\left(\frac{t}{1+t}\right),
$$
where $ \dB_{p+1}$ is a standard $(p+1)$-dimensional Brownian motion on $[0,1]$.
Finally, from the properties of $D$-convergence to Brownian motion and the continuous mapping theorem, it follows that
$\ds \sup_{K\ge 1} \frac{\left\|\hat \bGamma(N,K)\right\|}{\ell_\gamma\left(K/N\right)}$ converges in law to $\ds \sup_{t\in (0,1]}\frac{\left\|\dB_{p+1}(t)\right\|}{t^\gamma}$, completing the proof.
\qed

\subsection{Proof of \Cref{ThmAlt2}}\label{app:pf-thm2-alt}
Under the alternative hypothesis \eqref{eq:H1}, there exists $K_*\ge 1$, such that $K_*/N\to t_*$, and
for $K_*<K = \N{t}$,
\begin{multline*}
 \hat\bGamma(N,\N{t}) = A_N^{1/2}T^{1/2}\left(\hat\btheta_{N+\N{t}}-\hat\btheta_N\right)    =
 A_N^{1/2}T^{1/2}\left(\hat\btheta_{N+\N{t}}-\hat\btheta_{N+\N{t_*} } \right) \\
 + A_N^{1/2}T^{1/2}\left(\hat\btheta_{N+\N{t_*} } -\hat\btheta_{N}\right)\\
 =  A_N^{1/2}T^{1/2}\left(\hat\btheta_{N+\N{t}}-\hat\btheta_{N+\N{t_*} } \right)  +O_P(1).
\end{multline*}
Furthermore, for $t>t_*$,
\begin{multline*}
\frac{1}{T }A_{N+\N{t}} = \frac{1}{T } \sum_{i=1}^{N+K_*}\bZ_{N,i}\bZ_{N,i}^\top +\frac{1}{T }\sum_{i=N+K_*+1}^{N+\N{t}}\bZ_{N,i}^{(2)}{\bZ_{N,i}^{(2)}}^\top\\
=(1+t_*)\bSigma + (t-t_*) \bSigma_*+o_P(1) = (t-t_*)\left(\frac{ (1+t_*)}{(t-t_*)}\bSigma+\bSigma_*\right) +o_P(1),
\end{multline*}
where $\bSigma_*$ is defined in \eqref{sigma*}.
As a result,
\begin{multline*}
\left(\hat\btheta_{N+\N{t}}-\hat\btheta_{N+\N{t_*} } \right)
=   \left(\frac{1}{T }A_{N+\N{t}}\right)^{-1} T^{-1}\left(  \sum_{i=1}^{N+K_*}\bZ_{N,i}Y_{N,i}+  \sum_{i=N+K_*+1}^{N+\N{t}}\bZ_{N,i}^{(2)}Y_{N,i}^{(2)}\right)\\
-  T^{-1}\left(\frac{1}{T }A_{N+\N{t_*}}\right)^{-1}\left(  \sum_{i=1}^{N+K_*}\bZ_{N,i}Y_{N,i}\right)\\
=
 \left(\frac{1}{T }A_{N+\N{t}}\right)^{-1}T^{-1} \left(  \sum_{i=1}^{N+K_*}\bZ_{N,i} \bZ_{N,i}^\top \btheta_0 +  \sum_{i=1}^{N+K_*}\bZ_{N,i} \epsilon_{N,i}+ \sum_{i=1}^{N+K_*}\bZ_{N,i} r_{N,i}\right)\\
 +  \left(\frac{1}{T }A_{N+\N{t}}\right)^{-1}  T^{-1} \left(  \sum_{i=N+K_*+1}^{N+\N{t}}\bZ_{N,i}^{(2)}{\bZ_{N,i}^{(2)}}^\top \btheta_* + \sum_{i=N+K_*+1}^{N+\N{t}}\bZ_{N,i}^{(2)}\epsilon_{N,i}^{(2)}
 +\sum_{i=N+K_*+1}^{N+\N{t}}\bZ_{N,i}^{(2)}r_{N,i}^{(2)}\right)\\
- T^{-1} \left(\frac{1}{T }A_{N+\N{t_*}}\right)^{-1}\left( \sum_{i=1}^{N+K_*}\bZ_{N,i}\bZ_{N,i}^\top \btheta_0 + \sum_{i=1}^{N+K_*}\bZ_{N,i}\epsilon_{N,i} +\sum_{i=1}^{N+K_*}\bZ_{N,i}r_{N,i}\right)\\
\stackrel{Pr}{\longrightarrow} \left(\frac{ (1+t_*)}{(t-t_*)}\bSigma+\bSigma_*\right)^{-1} \bSigma_* (\btheta_*-\btheta_0).
\end{multline*}
Finally,  setting $\bSigma_{t_*,t} = \bSigma\left(\frac{ (1+t_*)}{(t-t_*)}\bSigma+\bSigma_*\right)^{-1} \bSigma_*$, one
may conclude that
$$
T^{-1/2} \left\|\hat{\bGamma}(N,\N{t})\right\|/\ell_\gamma\left(\frac{\N{t}}{N}\right) \stackrel{{\rm Pr}}{\longrightarrow}
\frac{1}{\ell_\gamma(t)}\left\|\bSigma_{t_*,t} (\btheta_*-\btheta_0)\right\| =\frac{\kappa_{t_*,t}}{\ell_\gamma(t)} > 0,
$$
and
\begin{equation}\label{eq:newkappa}
\lim_{t\to\infty} \kappa_{t_*,t}  =\kappa_{\hat \bGamma,\btheta_0,\btheta_*} = \left\| \bSigma (\btheta_*-\btheta_0)\right\|>0.
\end{equation}
As a result,
$$
 \sup_{0 \le t < \infty}\left\{\left\|\hat{\bGamma}(N,\N{t})\right\|/h\left(\frac{\N{t}}{N}\right)\right\}\stackrel{\rm{Pr}}{\longrightarrow}\infty,
$$
as $N\to\infty$.
This completes the proof.\qed

\section{Computation of $\tilde h$}\label{app:tildeh}

Note that for for any $k\ge 1$ and any $t\in[0,1]$,
$$
\int_{-\infty}^t e^{a s}\varphi_k(s)ds  = \frac{e^{-a}}{1-e^{-a}}\int_0^1 e^{a s}\varphi_k(s)ds +
\int_0^t e^{a s}\varphi_k(s)ds.
$$
In particular, if $k=1$, then
\begin{equation}\label{eq:htilde1}
e^{-a t}  \int_{-\infty}^t e^{a s}\varphi_1(s)ds =   \frac{1}{a}.
\end{equation}
Next, if $\bvarphi_2(s) = 2^{1/2} \cos{2\pi s}$, then, for any $t\in[0,1]$,
\begin{equation}\label{eq:htilde2}
 e^{-a t}  \int_{-\infty}^t e^{a s}\varphi_2(s)ds =   \frac{2^{1/2}}{a^2+4\pi^2}\left\{
a  \cos{2\pi t} + 2\pi \sin{2\pi t}\right\} = a\varphi_2(t)+b\psi_2(t),
\end{equation}
where $\psi_2(t) = 2^{1/2} \sin{2\pi t}$. Recall that $\left\{1,2^{1/2}\cos{2k\pi t},2^{1/2}\sin{2k \pi t}:\; k\ge 1\right\}$ is a complete orthonormal system on $L^2([0,1])$, so setting $\varphi_k(t) = 2^{1/2}\cos{2(k-1)\pi t}$, $k\ge 1$,  and $\psi_k(t) = 2^{1/2}\sin{2(k-1)\pi t}$, $k\ge 1$, one gets
that  for $k\ge 1$,
$$
e^{-a t}  \int_{-\infty}^t e^{a s}\varphi_k(s)ds =   a_k \varphi_k(t) + b_k \psi_k(t),
$$
with $a_k = \frac{a}{a^2+4\pi^2 (k-1)^2}$ and $b_k = \frac{2\pi(k-1)}{a^2+4\pi^2 (k-1)^2}$. Note that $a_1=\frac{1}{a}$, $b_1=0=\psi_1(t)$, and $a_k^2+b_k^2  = \frac{1}{a^2+4\pi^2 (k-1)^2}$.
As a result, for $\btheta=(\mu_1,\ldots,\mu_p,a)$, one gets
\begin{equation}\label{eq:htilde}
 \tilde h(t) =   \sum_{k=1}^p \mu_k \{a_k \varphi_k(t) +b_k \psi_k(t)\}.
\end{equation}
It then follows that $\ds \int_0^1 \tilde h(t)\varphi_k(t)dt = \mu_k a_k$, $k\ge 1$, and
$\ds \int_0^1 \left\{\tilde h(t)\right\}^2dt=  \sum_{k=1}^p \frac{\mu_k^2}{a^2+4(k-1)^2 \pi^2} $. Hence,
\begin{equation}\label{eq:Lambda}
 \Lambda^\top = -\left[ \mu_1 a_1, \ldots, \mu_p a_p \right].
\end{equation}
Finally, $\ds \omega =  \sum_{k=1}^p \frac{\mu_k^2}{a^2+4(k-1)^2 \pi^2} +\frac{\sigma^2}{2a}$. Taking for example $p=2$, so $\btheta=(\mu_1,\mu_2,a)$, it follows from \eqref{sigma} that
\begin{equation}\label{eq:Sigmaex}
\bSigma = \left( \begin{array}{ccc}
1 & 0 & -\frac{\mu_1}{a} \\
0 & 1 &  -\frac{\mu_2a}{a^2+4\pi^2}\\
-\frac{\mu_1}{a} & -\frac{\mu_2a }{a^2+4\pi^2} & \frac{\mu_1^2}{a^2}+\frac{\mu_2^2}{a^2+4\pi^2}+\frac{\sigma^2}{2a}
\end{array}\right).
\end{equation}
In particular, if $\btheta=(1,2,1)$, then
\begin{equation}\label{eq:Sigmaex1}
\bSigma = \left( \begin{array}{ccc}
1 & 0 & -1 \\
0 & 1 &  -\frac{2}{1+4\pi^2}\\
-1 & -\frac{2}{1+4\pi^2} & \frac{5+4\pi^2}{1+4\pi^2}+\frac{\sigma^2}{2}
\end{array}\right).
\end{equation}

Now, using $\btheta_*$ instead of $\btheta_0$, one gets

\begin{equation}\label{eq:Sigmastar}
\bSigma_* = \left( \begin{array}{ccc}
1 & 0 & -\frac{\mu_{1*}}{a_*} \\
0 & 1 &  -\frac{\mu_{2*}a_*}{a_*^2+4\pi^2}\\
-\frac{\mu_{1*}}{a_*} & -\frac{\mu_{2*}a_* }{a_*^2+4\pi^2} & \frac{\mu_{1*}^2}{a_*^2}+\frac{\mu_{2*}^2}{a_*^2+4\pi^2}+\frac{\sigma^2}{2a_*}
\end{array}\right).
\end{equation}

\section{Exact simulation of the generalized OU process}\label{app:sim}

First, suppose that
$\bvarphi(t) = \left(\phi_1(t),\ldots, \phi_{p_1}(t), \psi_2(t),\ldots, \psi_{p+1-p_1}(t)\right)^\top$, with $\phi_1(t)\equiv 1$, and for $k\ge 2$,
$\phi_k(t) = 2^{1/2} \cos\{2\pi(k-1)t\}$, and $\psi_k(t) = 2^{1/2} \sin\{2\pi(k-1)t\}$.
Then, for $\Delta>0$ and for $k\ge 1$,
\begin{eqnarray}\label{eq:expphi2}
A_k &=& e^{-a \Delta }\int_{0}^{\Delta} e^{a s}\cos\{2\pi(k-1)s\}ds \\
&= & a_k \cos\{2\pi(k-1)\Delta\}-a_k e^{-a \Delta }  +b_k\sin\{2\pi(k-1)\Delta\},\nonumber
\end{eqnarray}
and
\begin{eqnarray}\label{eq:expphi1}
B_k &=& e^{-a \Delta }\int_{0}^{\Delta} e^{a s}\sin\{2\pi(k-1)s\}ds\\
& =&   b_k e^{-a \Delta } -b_k \cos\{2\pi(k-1)\Delta\} +a_k \sin\{2\pi(k-1)\Delta\},\nonumber
\end{eqnarray}
where  $a_k = \frac{a}{a^2+4\pi^2 (k-1)^2}$ and $b_k = \frac{2\pi(k-1)}{a^2+4\pi^2 (k-1)^2}$.
As a result, for any $k\ge 1$,
\begin{equation}\label{eq:expphi}
 e^{-a t_i}\int_{t_{i-1}}^{t_i} e^{a s}\phi_k(s)ds = A_k \phi_k(t_{i-1}) - B_k\psi_k(t_{i-1}),
\end{equation}
and
\begin{equation}\label{eq:exppsi}
 e^{-a t_i}\int_{t_{i-1}}^{t_i} e^{a s}\psi_k(s)ds = B_k \phi_k(t_{i-1}) +A_k\psi_k(t_{i-1}).
\end{equation}
Next, suppose that for $t\in [0,T]$,
$dX_t = (\bmu^\top \bvarphi(t) - a X_t)dt+\sigma dB_t$. Then, for any $0\le i \le N+K_*$ and $t_i=i\Delta N=i\frac{T}{N}$, one gets
\begin{equation}\label{eq:OUsol}
X_{t_{i}} = e^{-a \Delta_N} X_{t_{i-1}}+ e^{-a t_i}\int_{t_{i-1}}^{t_i} e^{a s} \bmu^\top \bvarphi(s)ds + \sigma
\varepsilon_{N,i} ,
\end{equation}
where $\ds \varepsilon_{N,i} = e^{-a t_i}\int_{t_{i-1}}^{t_i} e^{a s} dB_s \sim N\left(0, \frac{1-e^{-2a \Delta_N}}{2a}\right)$.
If the parameters change after time $T(1+t_*) = N+K_*$, it follows that for
$i> N+K_*$,
\begin{equation}\label{eq:OUsol*}
X_{t_{i}} = e^{-a_* \Delta_N} X_{t_{i-1}}+ e^{-a_* t_i}\int_{t_{i-1}}^{t_i} e^{a_* s} \bmu_*^\top \bvarphi(s)ds + \sigma
\varepsilon_{N,i,*} ,
\end{equation}
where $\ds \varepsilon_{N,i,*} = e^{-a_* t_i}\int_{t_{i-1}}^{t_i} e^{a_* s} dB_s \sim N\left(0, \frac{1-e^{-2a_* \Delta_N}}{2a_*}\right)$.

\subsection{OU process}
Here, we take $p=1$ and $q=0$, so $\btheta^\top = (\mu_1,a)$. Setting $b=\mu_1/a$ and $\phi_N=e^{-a \Delta_N}$, one gets that
$$
X_{N,t_i} = b (1-\phi_N) + \phi_N X_{N,t_{i-1}}+ \sqrt{\frac{1-\phi_N^2}{2a}}\varepsilon_i,\qquad i\ge 1,
$$
$\varepsilon_i$ iid  $N(0,1)$, and $X_{N,0}=0$. Note that,  as $N\to\infty$,  $(1-\phi_N)/\Delta_N \to a$ and  $\frac{1-\phi_N^2}{2a\Delta_N} \to 1 $.

\addcontentsline{toc}{section}{References}
\bibliographystyle{apalike}

\section{Supplementary material}

\input{sup_mat}


\end{document}

%% file: sup_mat.tex
\setcounter{figure}{0} 

\subsection{Other boxplots}

\begin{figure}[ht!]
    \centering
             \includegraphics[width=6cm,height=3cm]{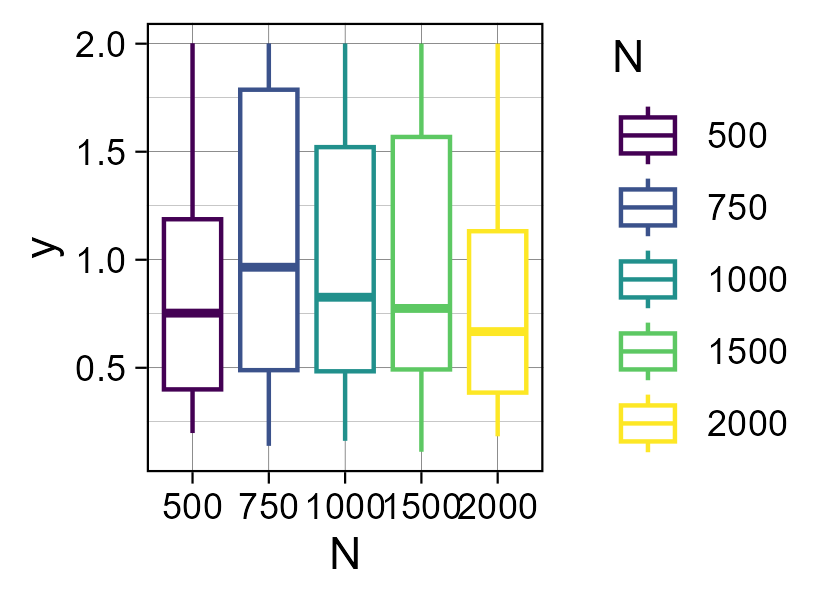}
    \includegraphics[width=6cm,height=3cm]{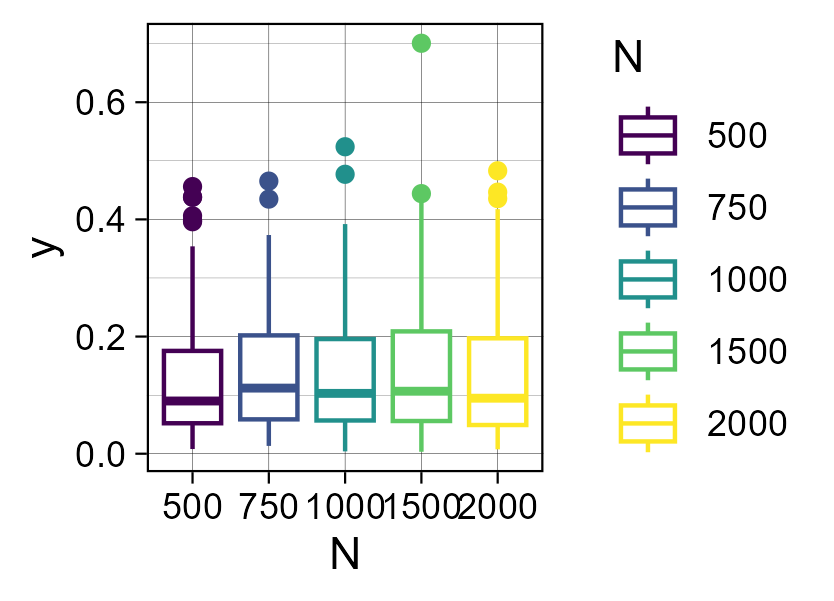}

            \includegraphics[width=6cm,height=3cm]{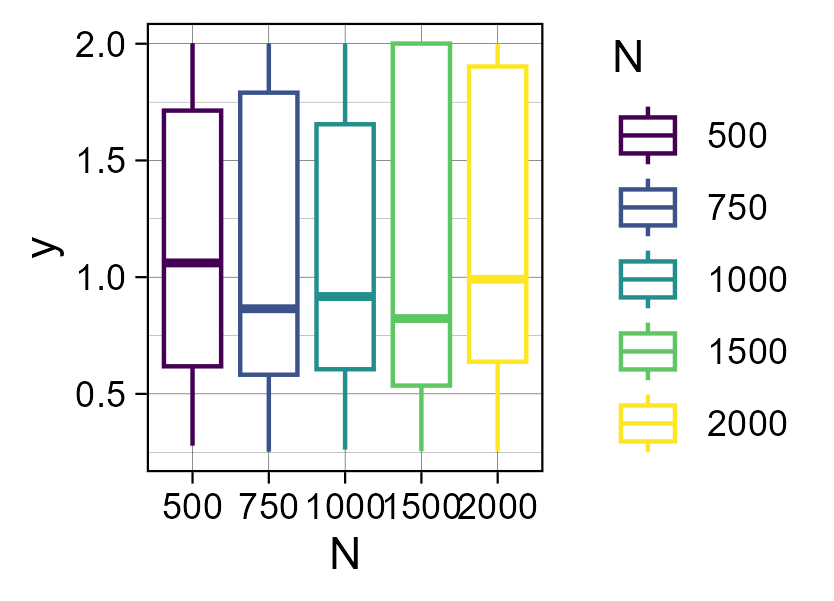}
    \includegraphics[width=6cm,height=3cm]{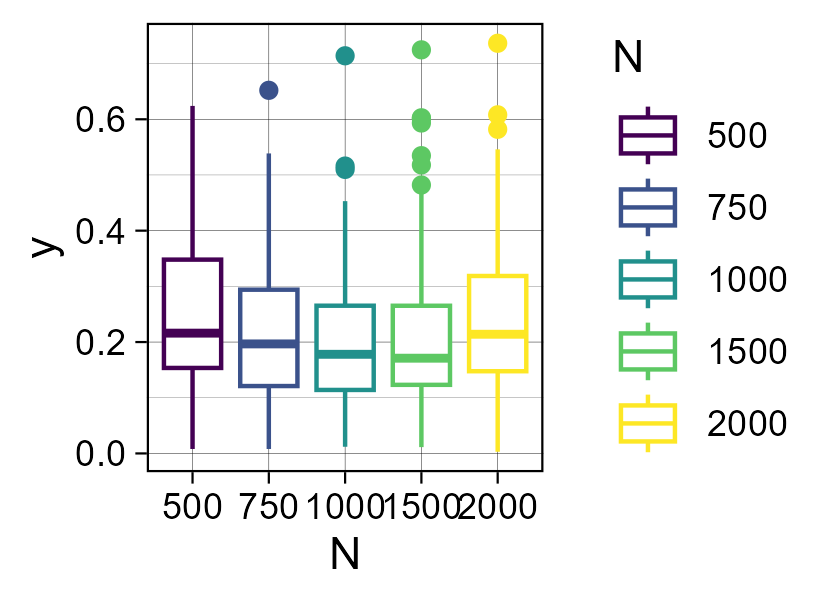}

            \includegraphics[width=6cm,height=3cm]{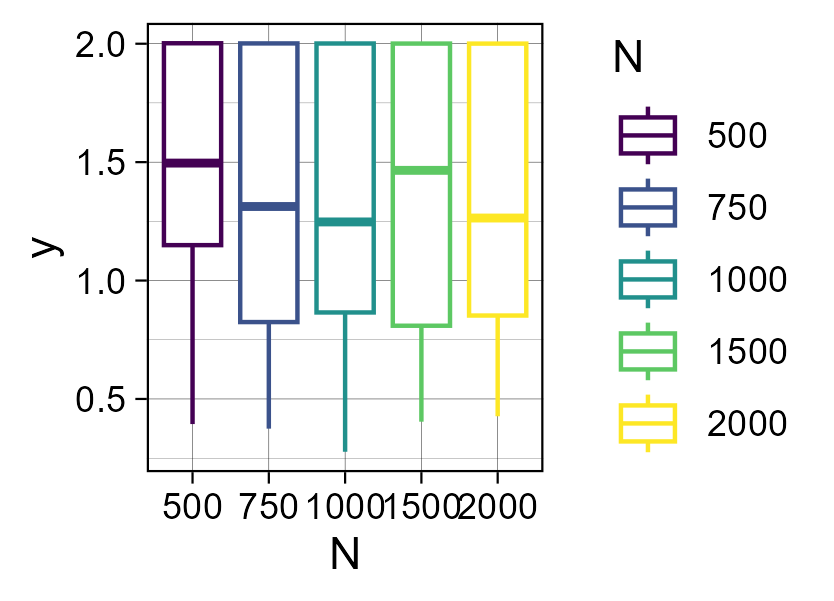}
    \includegraphics[width=6cm,height=3cm]{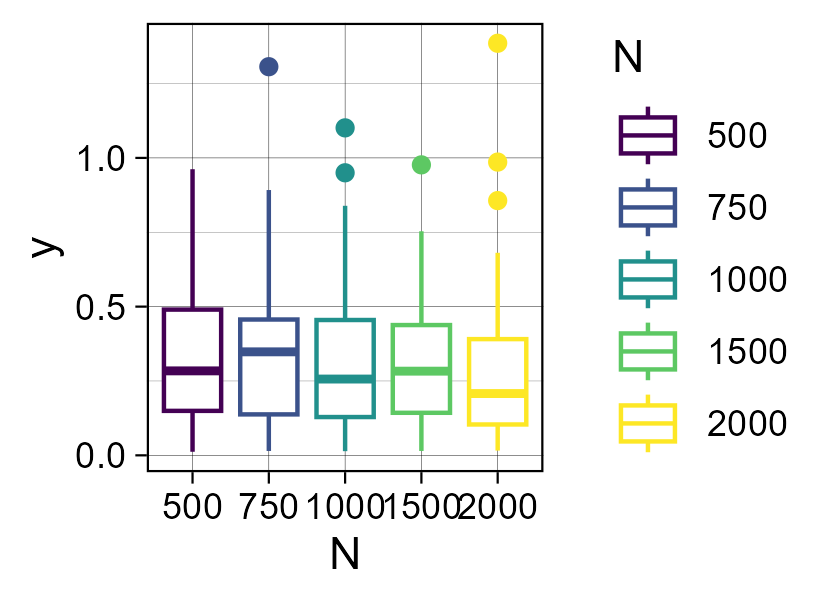}

            \includegraphics[width=6cm,height=3cm]{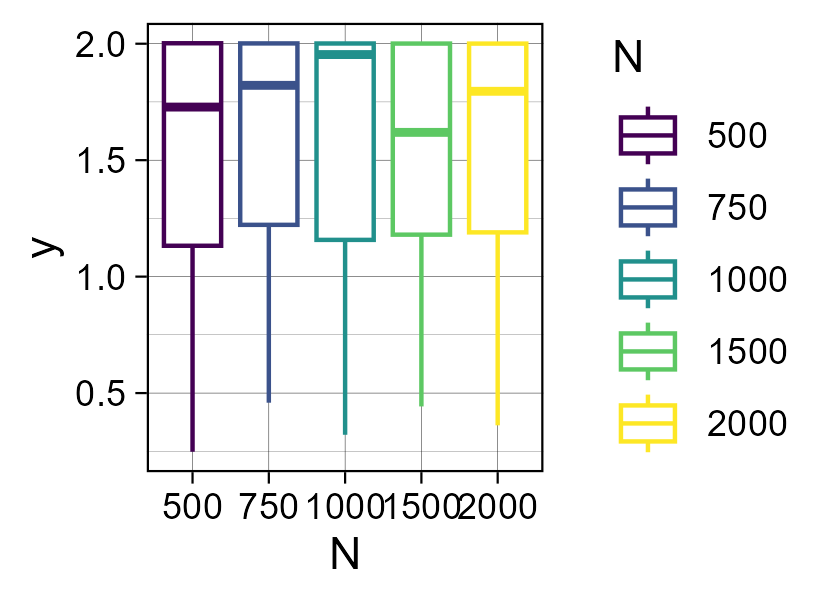}
    \includegraphics[width=6cm,height=3cm]{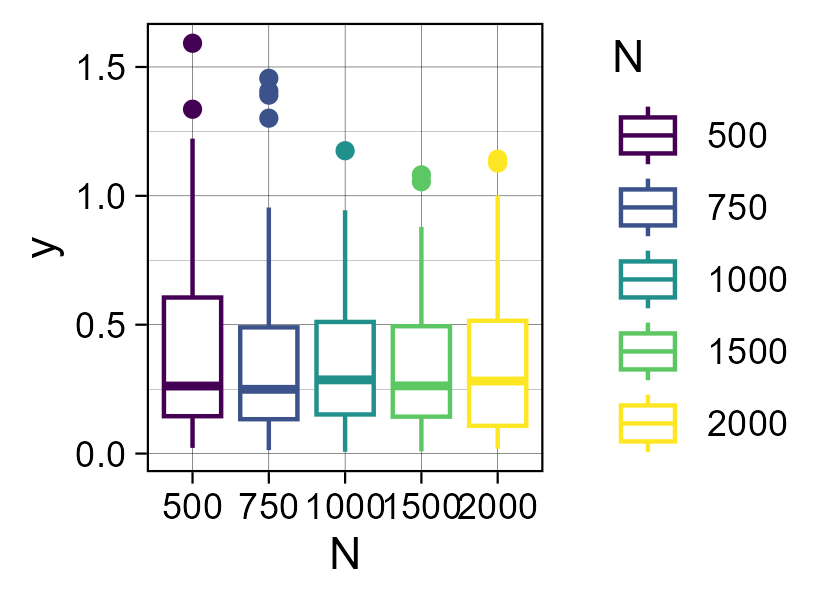}

            \includegraphics[width=6cm,height=3cm]{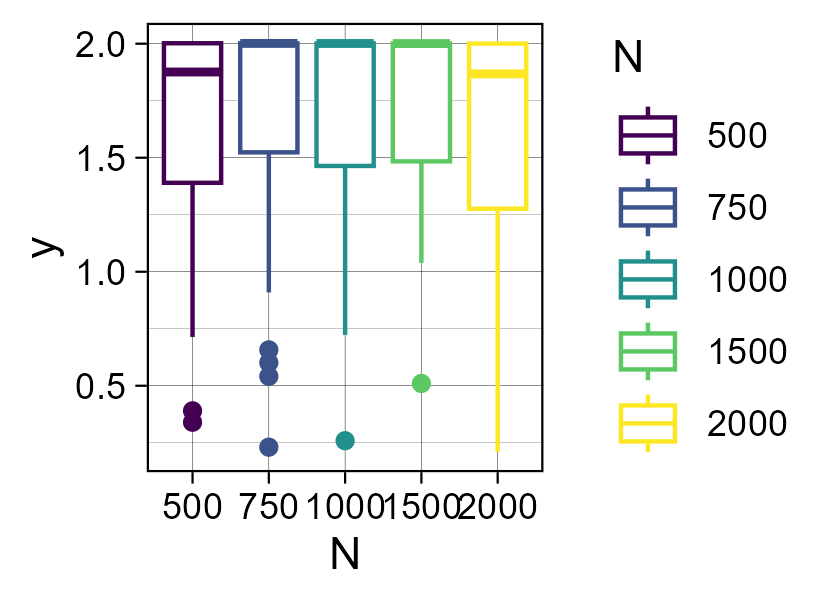}
    \includegraphics[width=6cm,height=3cm]{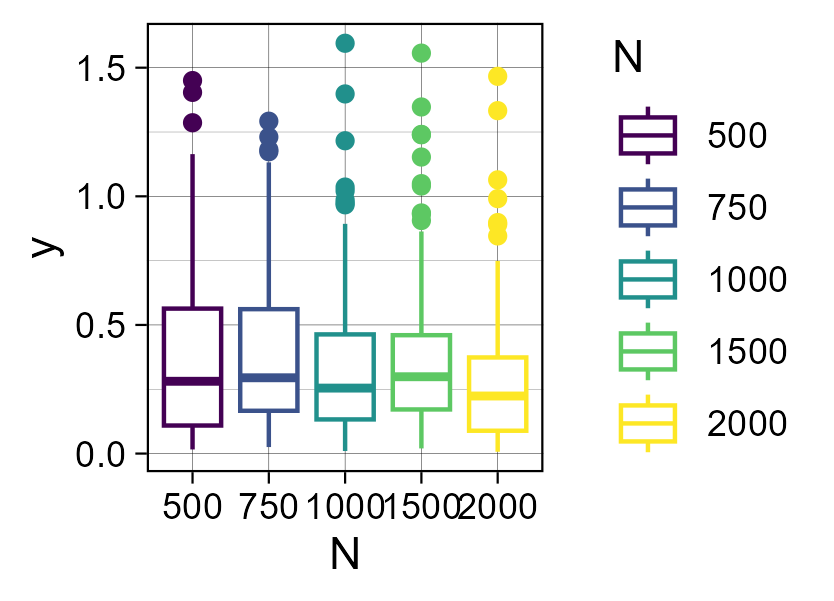}
        \caption{Estimation of $t_*$ when $t_*\in\{0,0.1,0.3,0.5,0.7\}$ (top to bottom), $\btheta_*=(15,3,4)$, $\kappa_{\hat Q,\btheta_0,\btheta_*}=2.75$ and $\kappa_{\hat \bGamma,\btheta_0,\btheta_*}=11.37$, based on $\widehat Q$ (left panel) and
     $\widehat \bGamma$ (right panel).}\label{fig:tauStar_t7_1534}
\end{figure}

\begin{figure}[ht!]
    \centering
                  \includegraphics[width=6cm,height=3cm]{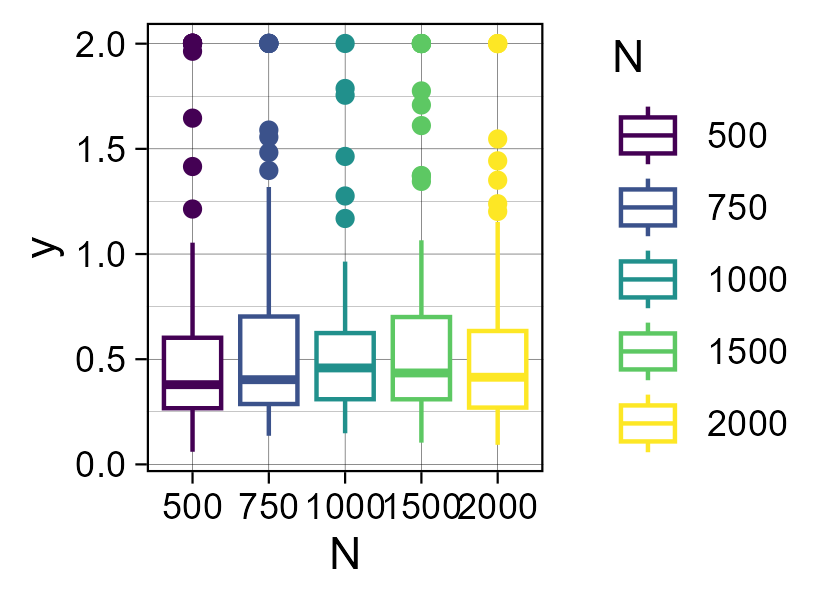}
    \includegraphics[width=6cm,height=3cm]{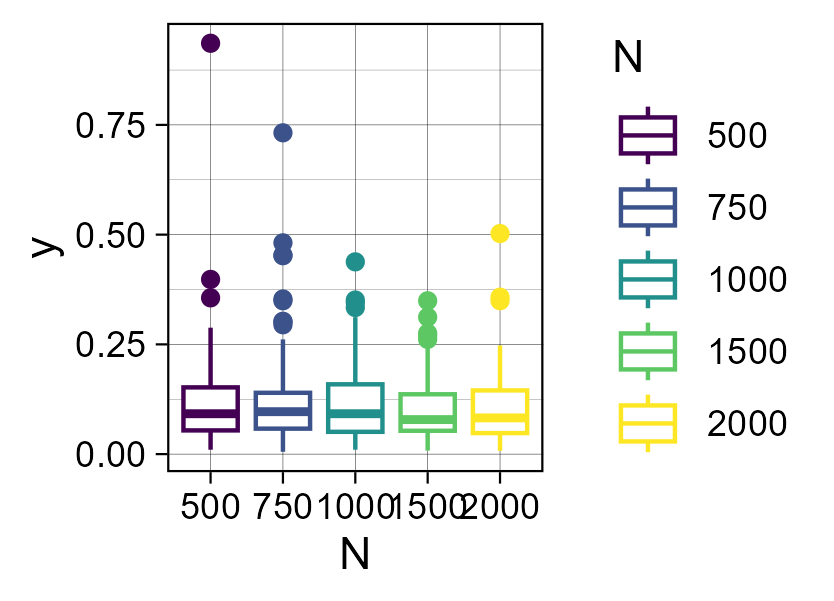}

      \includegraphics[width=6cm,height=3cm]{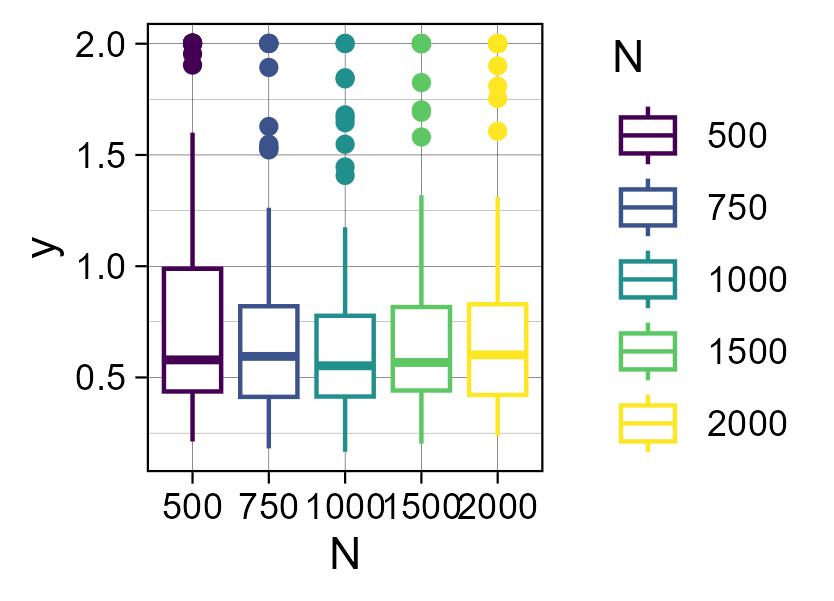}
    \includegraphics[width=6cm,height=3cm]{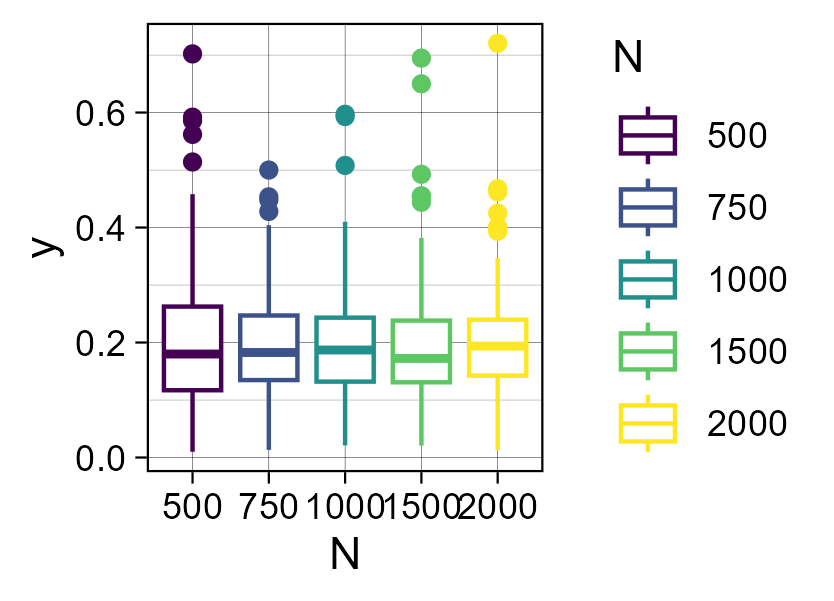}

                  \includegraphics[width=6cm,height=3cm]{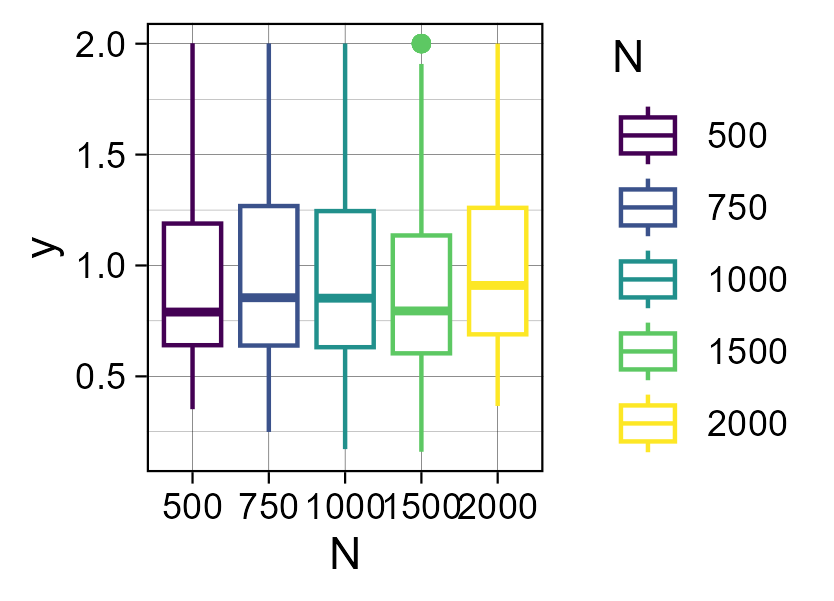}
    \includegraphics[width=6cm,height=3cm]{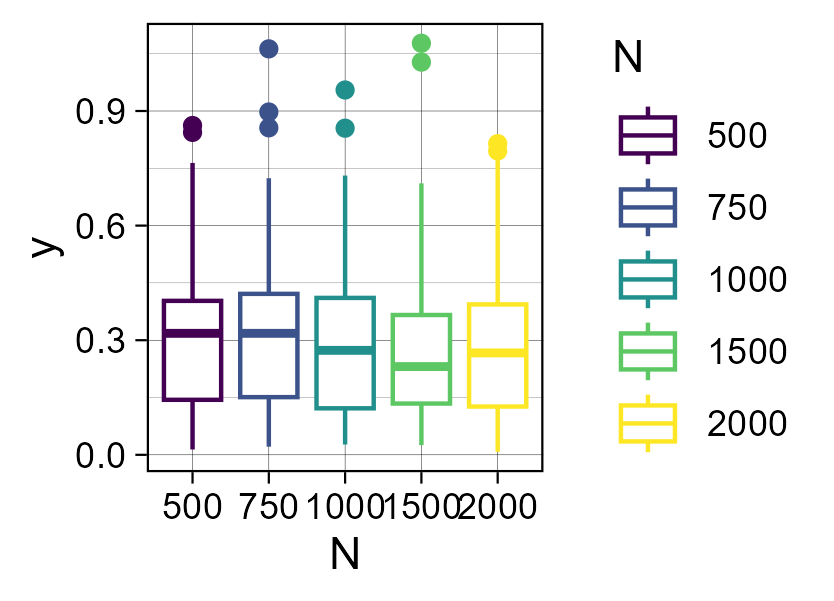}

                  \includegraphics[width=6cm,height=3cm]{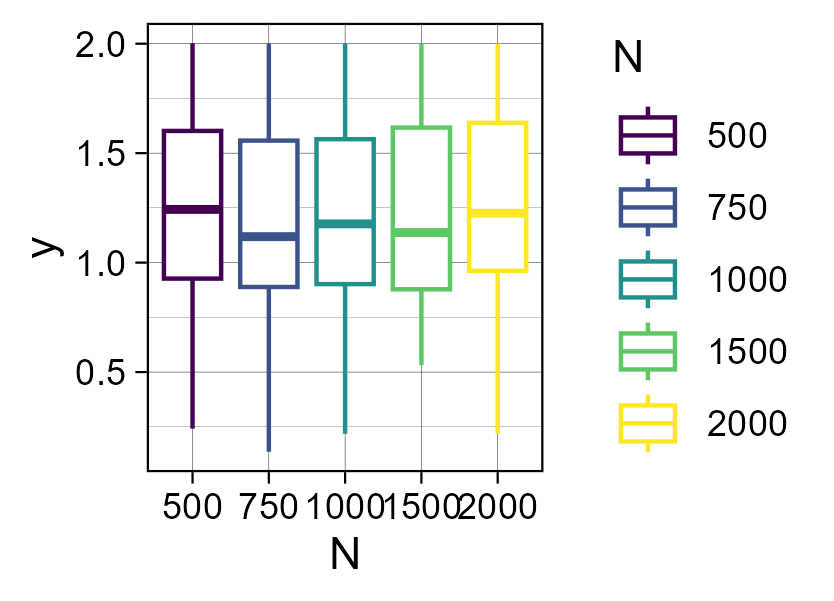}
    \includegraphics[width=6cm,height=3cm]{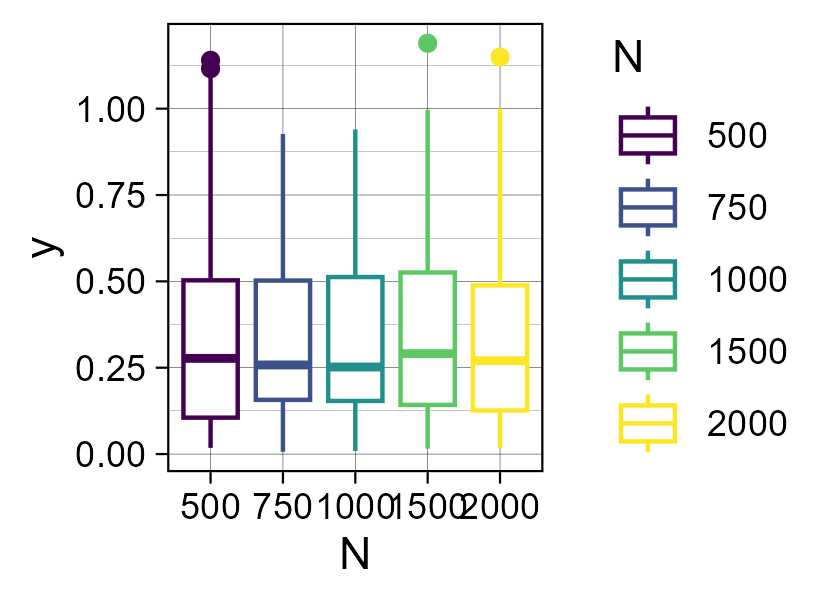}

                  \includegraphics[width=6cm,height=3cm]{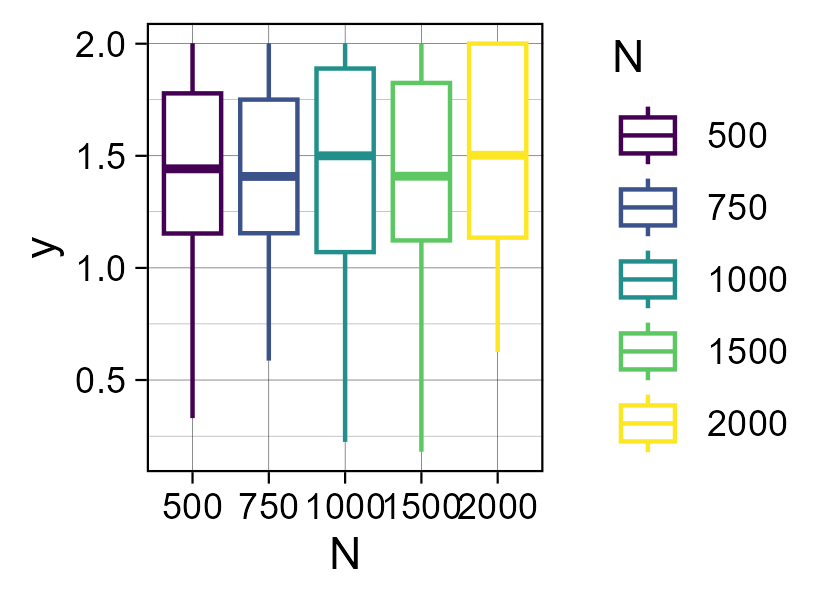}
    \includegraphics[width=6cm,height=3cm]{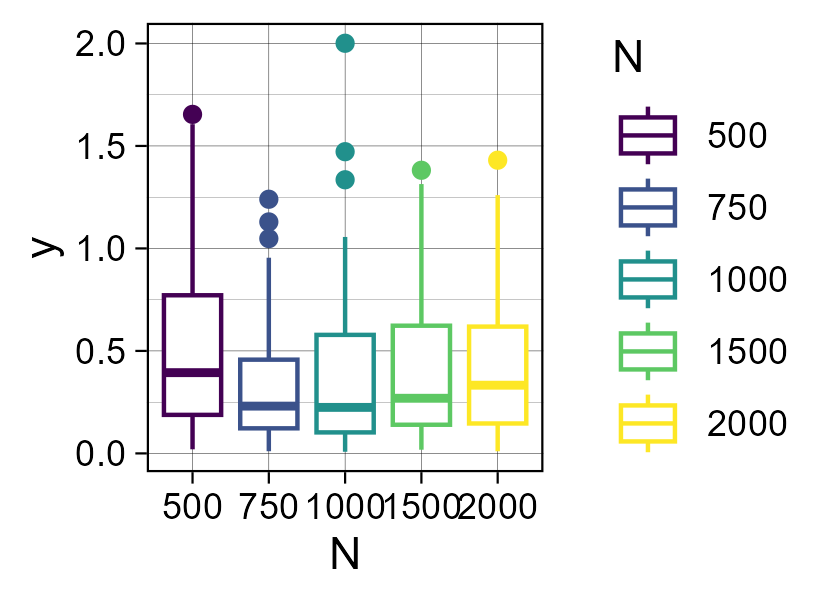}
        \caption{Estimation of $t_*$ when $t_*\in\{0,0.1,0.3,0.5,0.7\}$ (top to bottom), $\btheta_*=(5,3,1)$, $\kappa_{\hat Q,\btheta_0,\btheta_*}=4$ and $\kappa_{\hat \bGamma,\btheta_0,\btheta_*}=5.78$, based on $\widehat Q$ (left panel) and
     $\widehat \bGamma$ (right panel).}\label{fig:tauStar_t7_531}
\end{figure}

\subsection{Power curves}

\begin{figure*}[ht!]
    \centering
     \includegraphics[width=6cm,height=3cm]{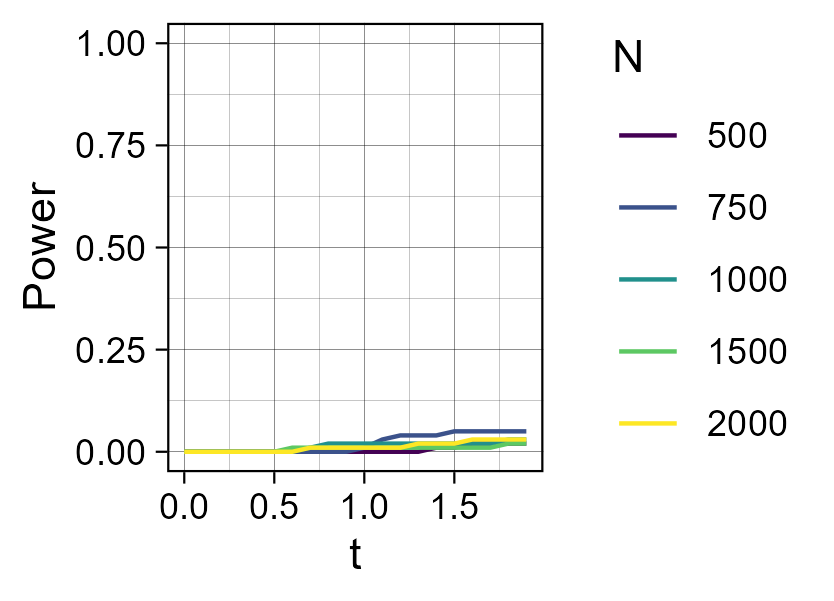}
    \includegraphics[width=6cm,height=3cm]{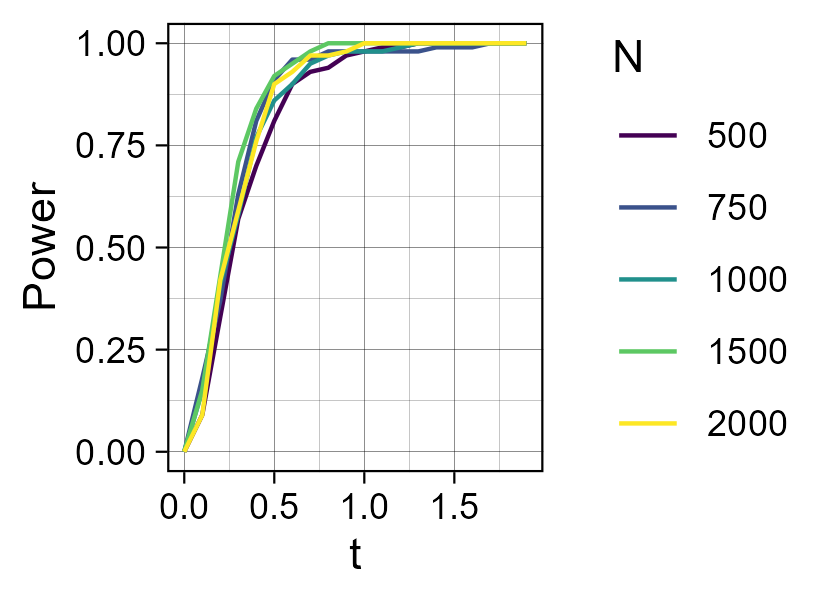}

   \includegraphics[width=6cm,height=3cm]{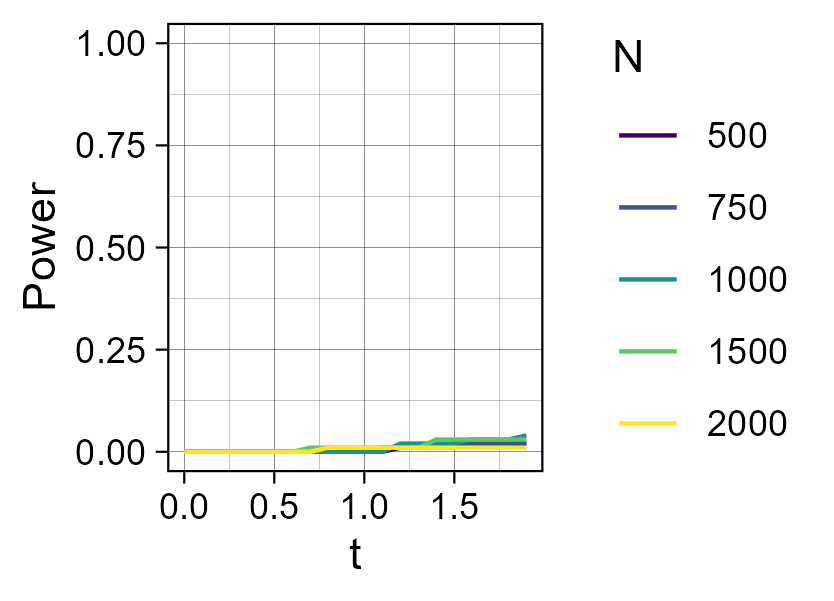}
    \includegraphics[width=6cm,height=3cm]{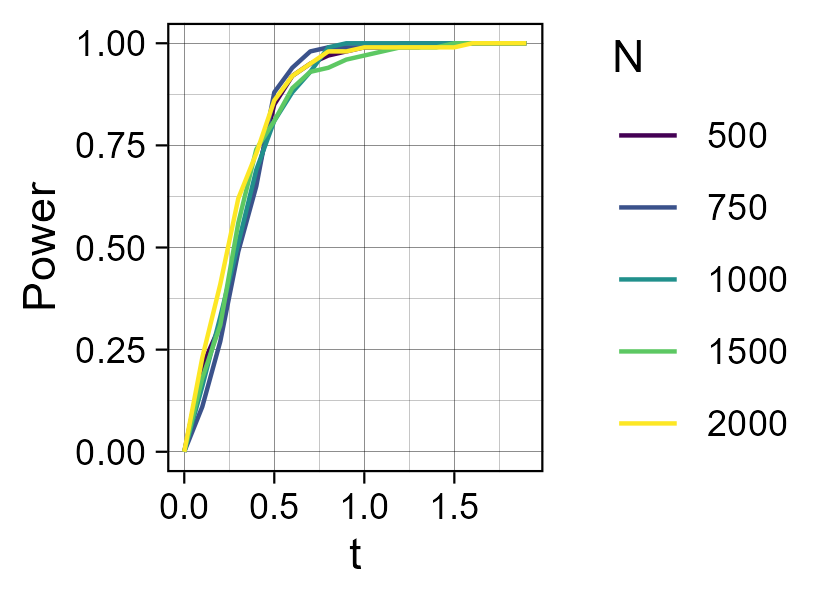}

      \includegraphics[width=6cm,height=3cm]{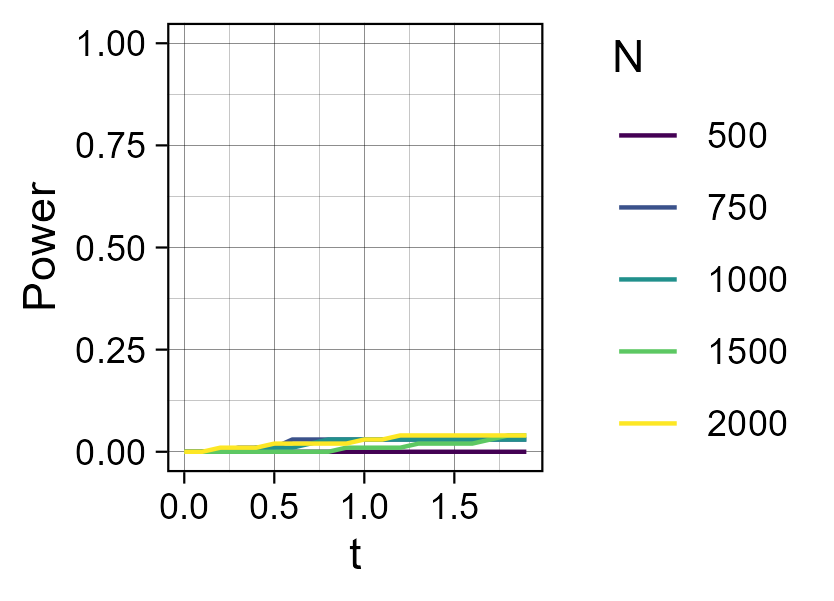}
    \includegraphics[width=6cm,height=3cm]{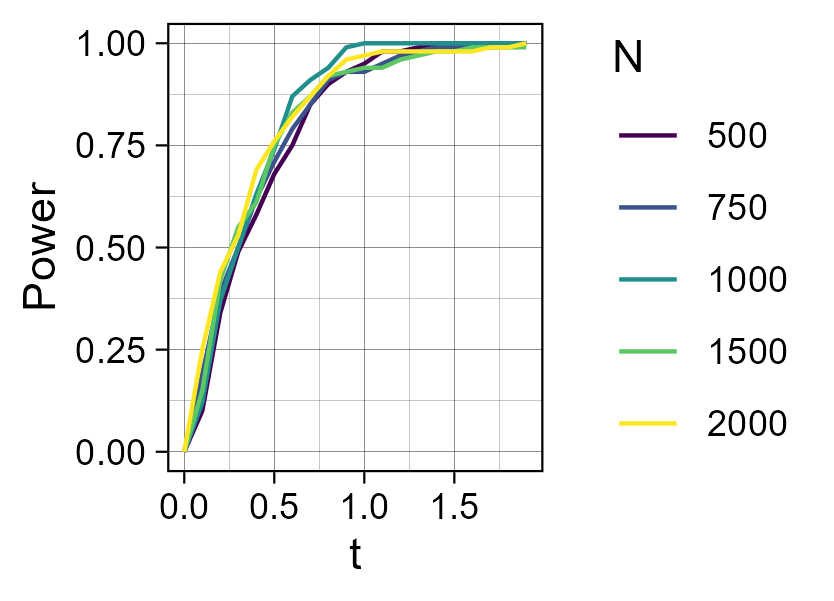}

     \includegraphics[width=6cm,height=3cm]{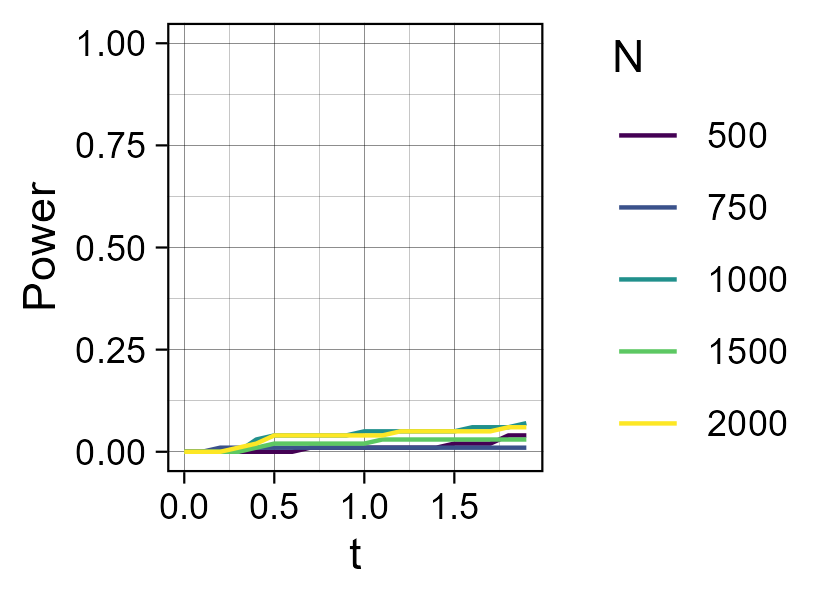}
    \includegraphics[width=6cm,height=3cm]{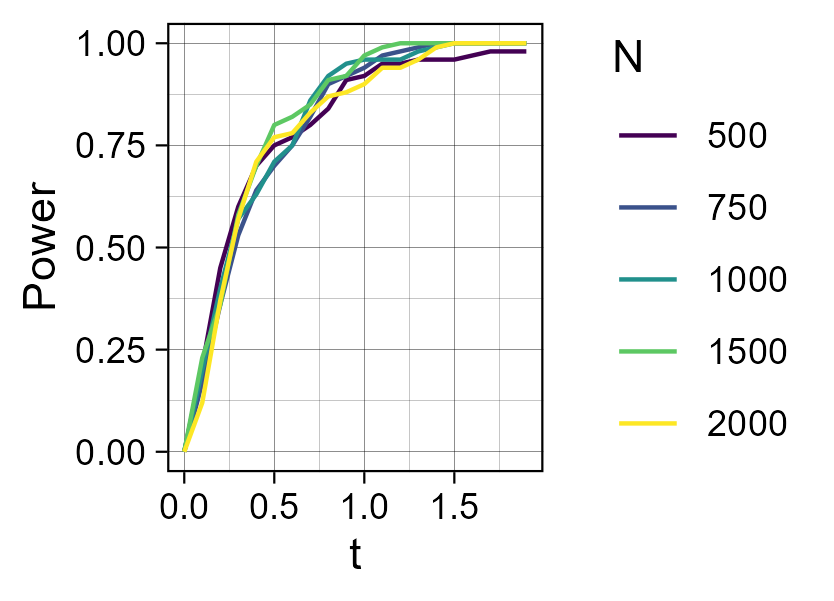}

     \includegraphics[width=6cm,height=3cm]{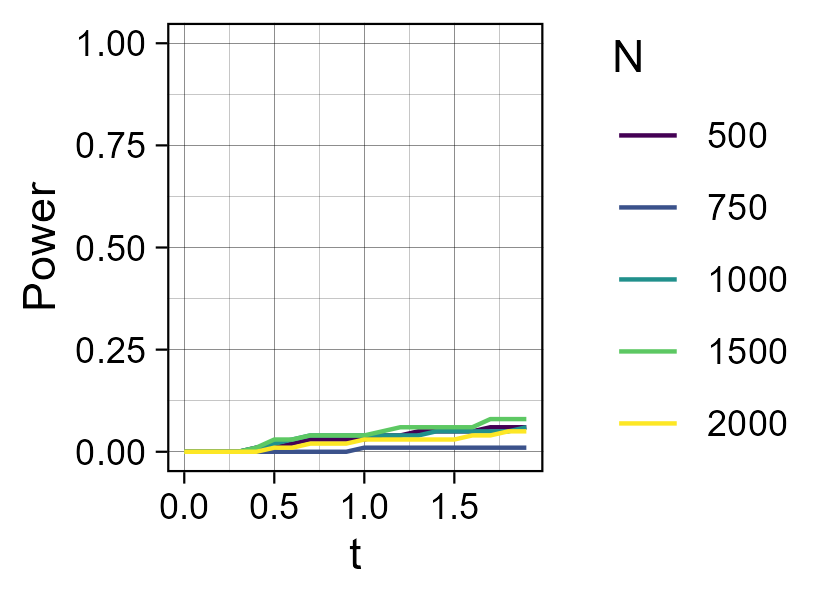}
    \includegraphics[width=6cm,height=3cm]{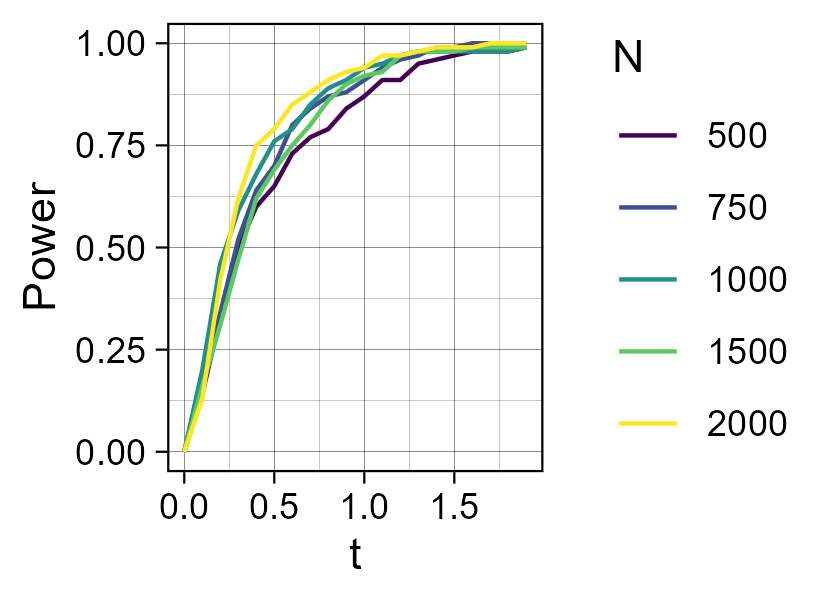}
    \caption{Empirical power of the estimation of  $t_*$ when $t_*\in\{0,0.1,0.3,0.5,0.7\}$ (top to bottom),  $\btheta_*=2\btheta_0=(2,4,2)$,  $\kappa_{\hat Q,\btheta_0,\btheta_*}=0$, based on $\widehat Q$ (left panel) and
     $\widehat \bGamma$ (right panel).}\label{fig:pow_tauStar_t7_242}
\end{figure*}

\begin{figure}[ht!]
    \centering

   \includegraphics[width=6cm,height=3cm]{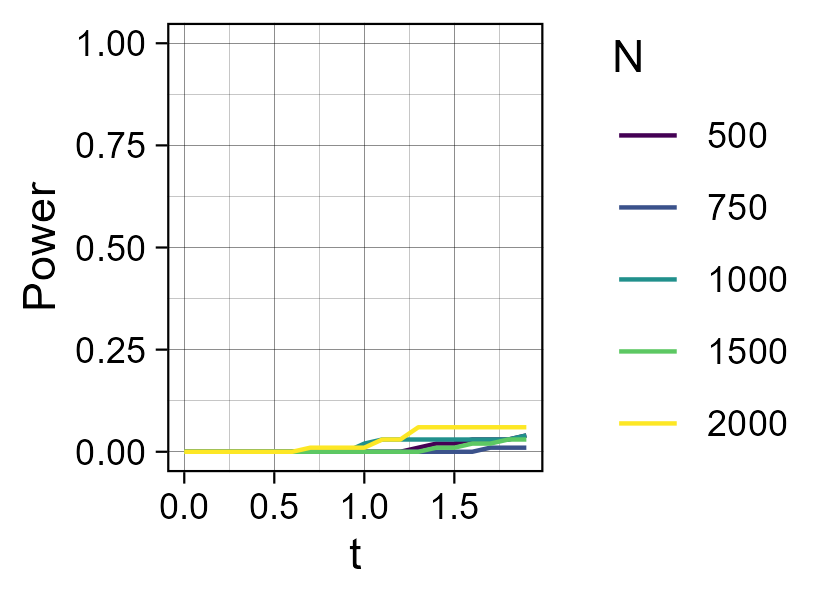}
    \includegraphics[width=6cm,height=3cm]{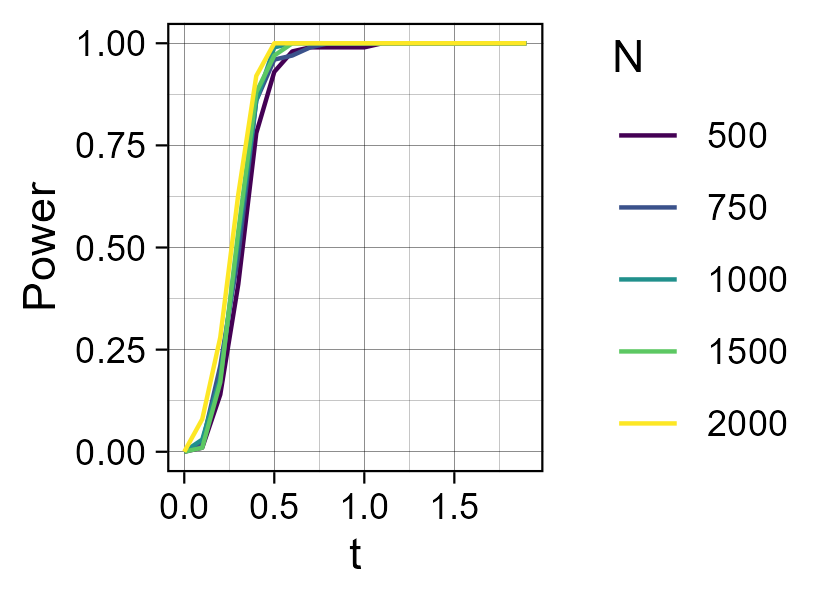}

 \includegraphics[width=6cm,height=3cm]{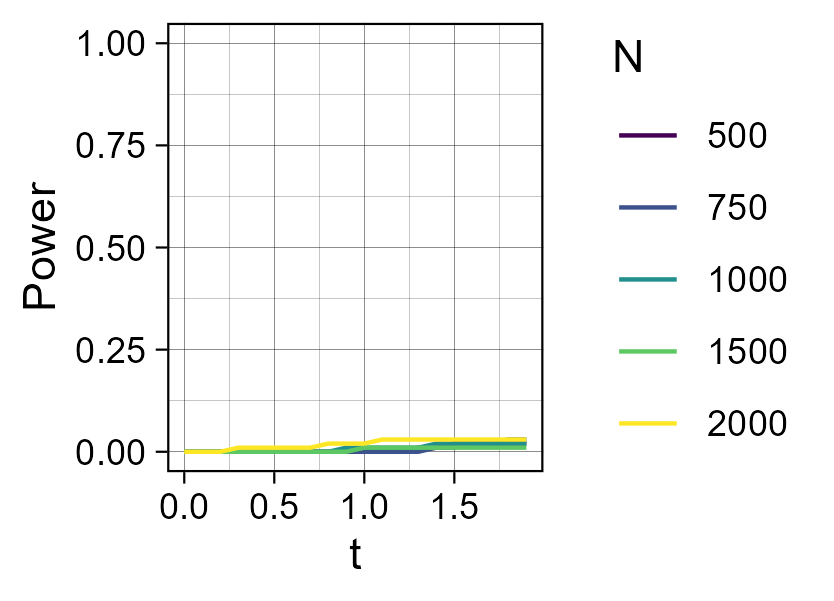}
    \includegraphics[width=6cm,height=3cm]{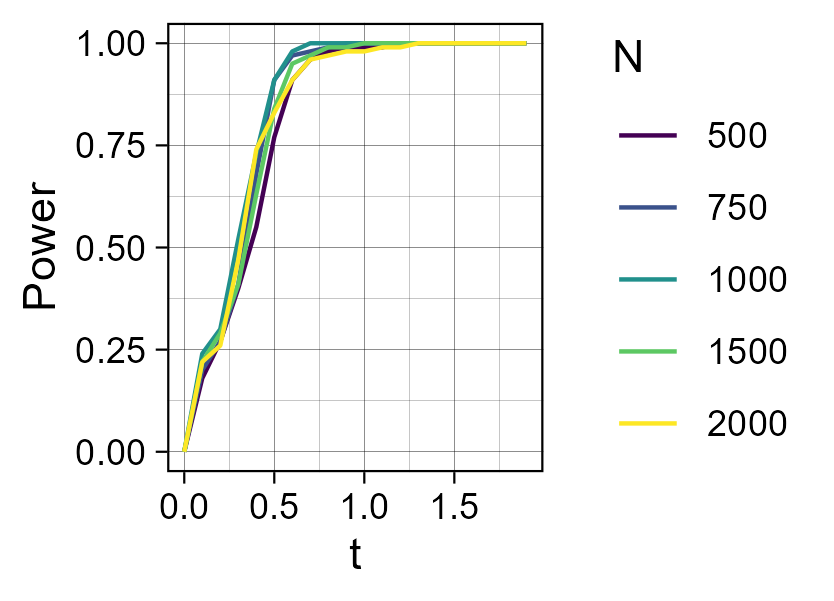}

\includegraphics[width=6cm,height=3cm]{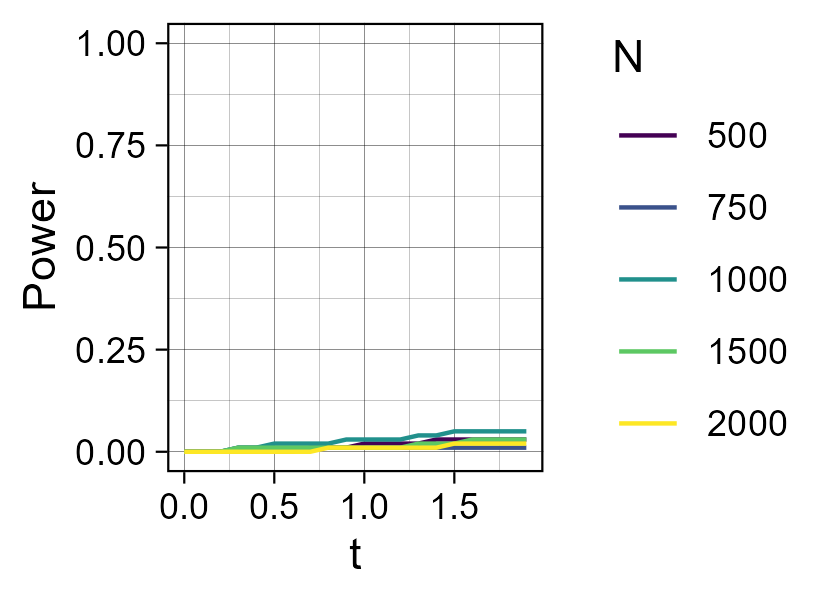}
    \includegraphics[width=6cm,height=3cm]{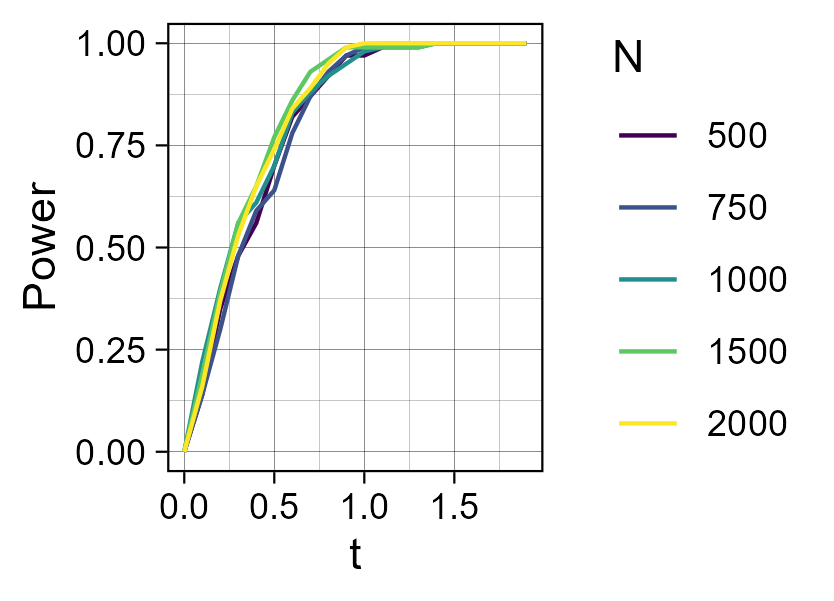}

\includegraphics[width=6cm,height=3cm]{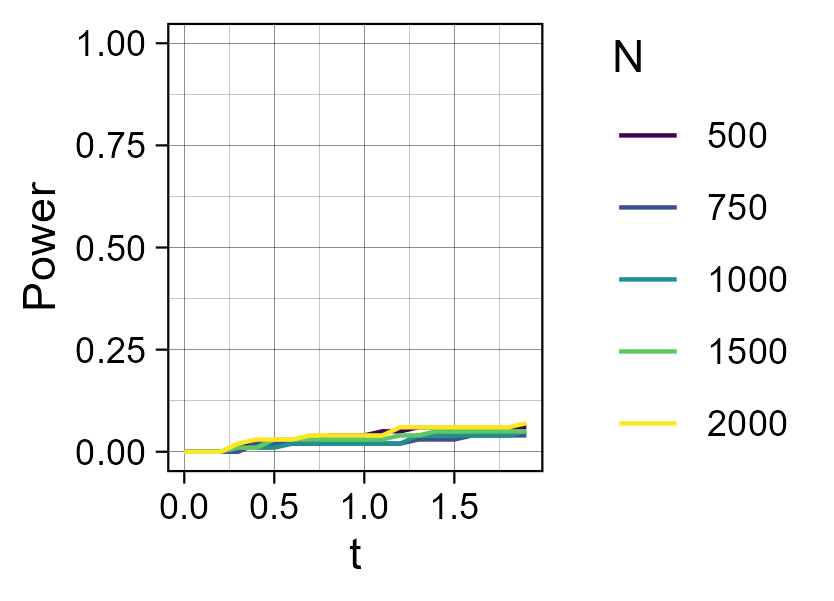}
    \includegraphics[width=6cm,height=3cm]{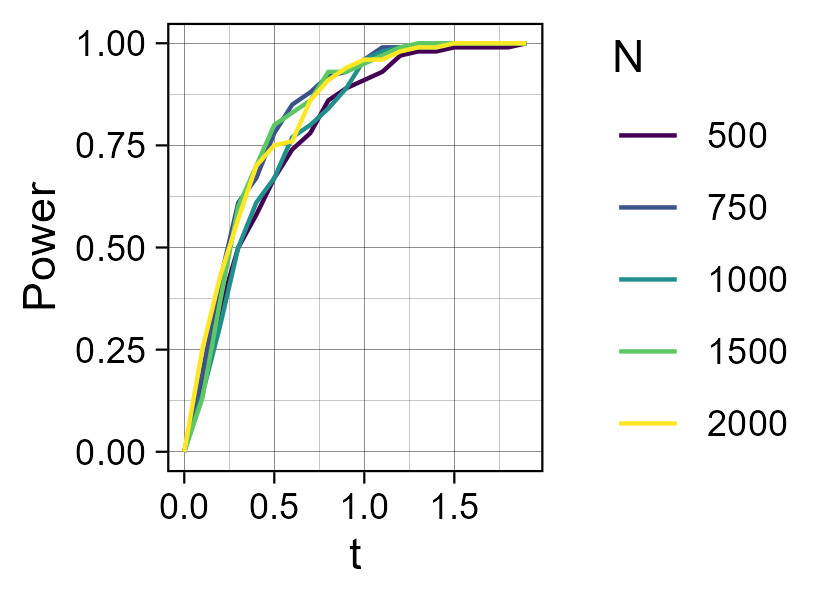}

 \includegraphics[width=6cm,height=3cm]{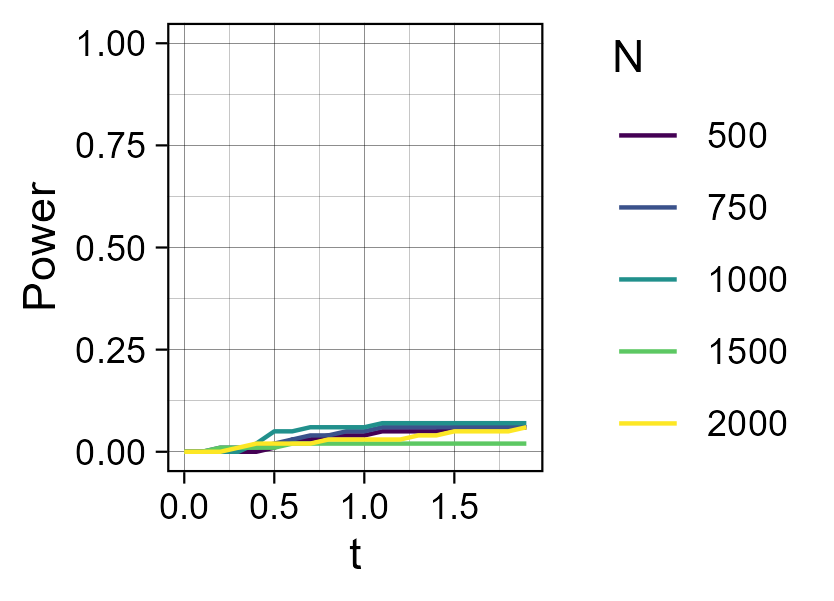}
    \includegraphics[width=6cm,height=3cm]{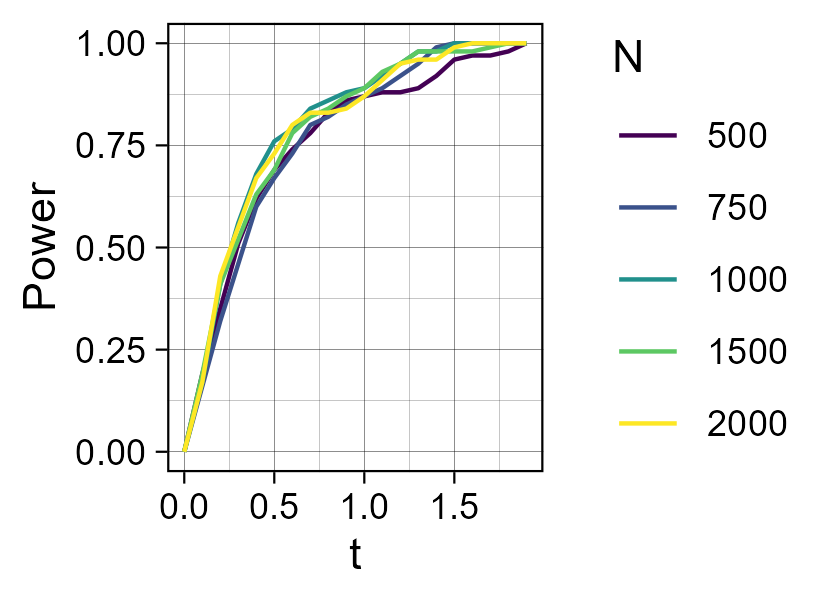}
        \caption{Empirical power of the estimation of  $t_*$ when $t_*\in\{0,0.1,0.3,0.5,0.7\}$ (top to bottom), $\btheta_*=(5,3,4)$,  $\kappa_{\hat Q,\btheta_0,\btheta_*}=0.25$ and $\kappa_{\hat \bGamma,\btheta_0,\btheta_*}=10.31$, based on $\widehat Q$ (left panel) and
     $\widehat \bGamma$ (right panel).}\label{fig:pow_tauStar_t0_534}
\end{figure}

\begin{figure}[ht!]
    \centering

      \includegraphics[width=6cm,height=3cm]{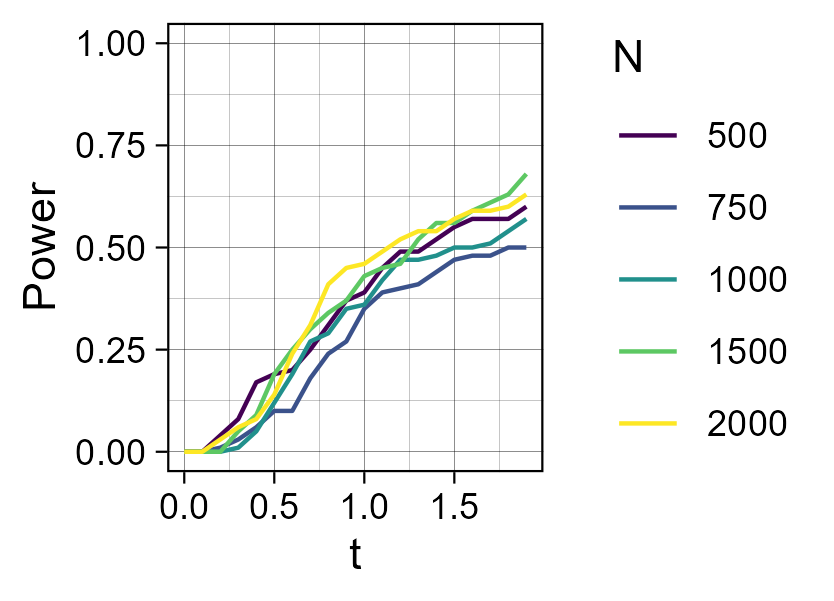}
    \includegraphics[width=6cm,height=3cm]{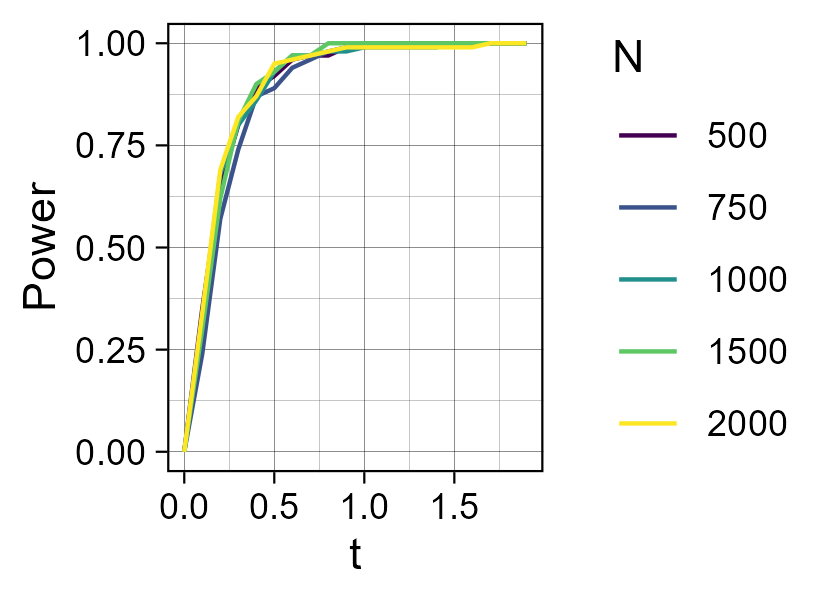}

       \includegraphics[width=6cm,height=3cm]{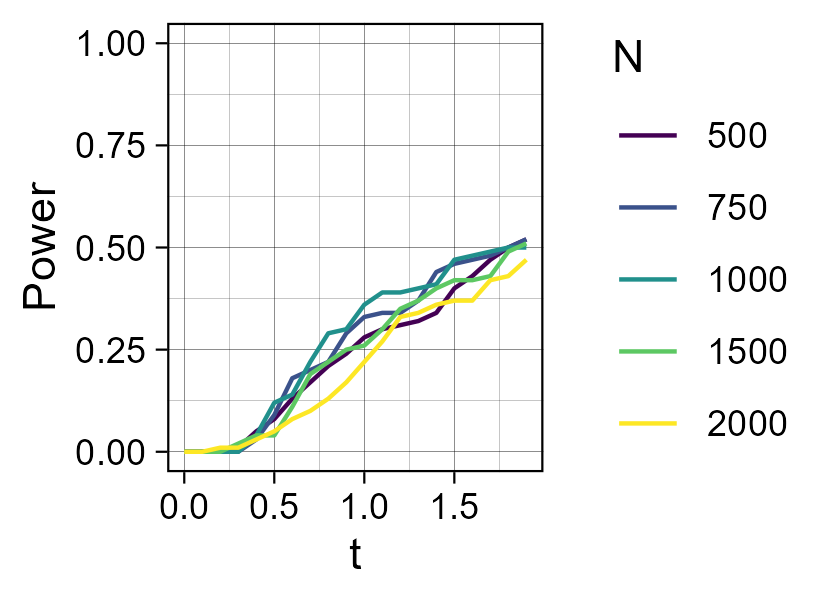}
    \includegraphics[width=6cm,height=3cm]{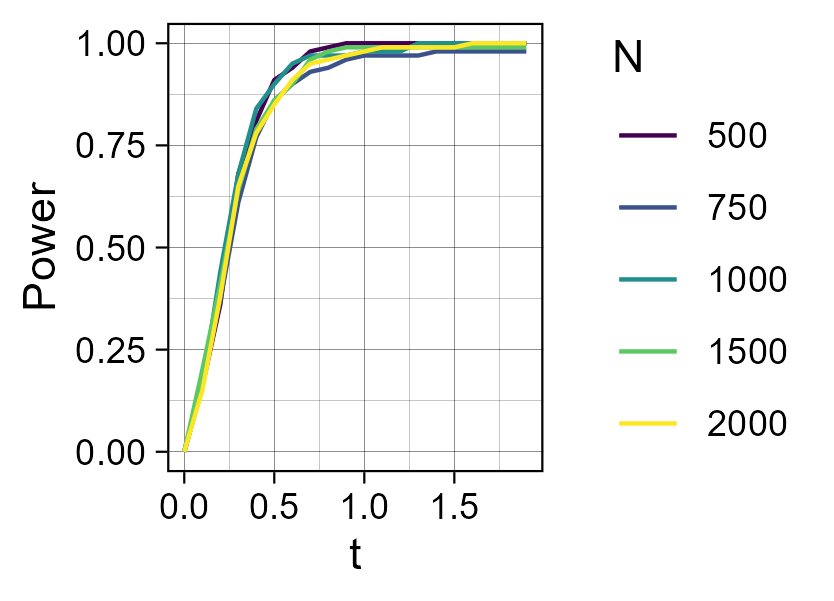}

       \includegraphics[width=6cm,height=3cm]{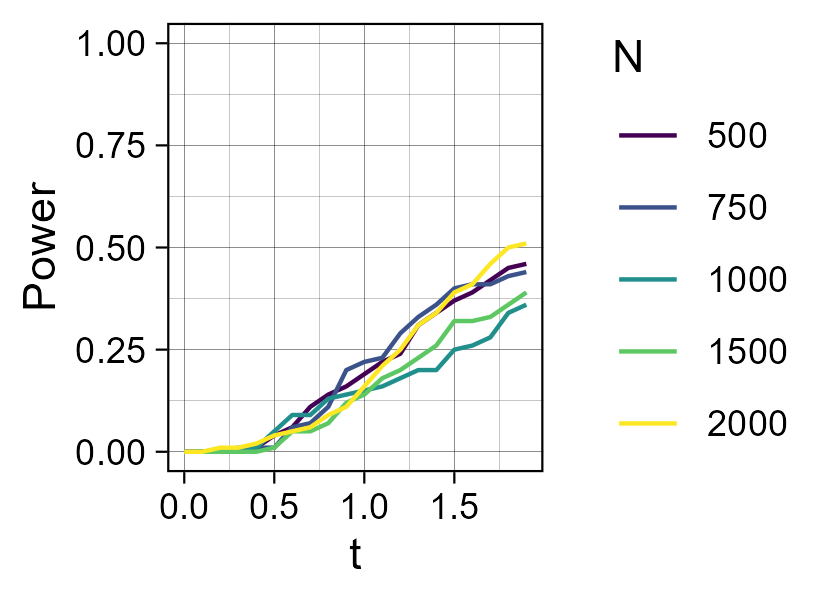}
    \includegraphics[width=6cm,height=3cm]{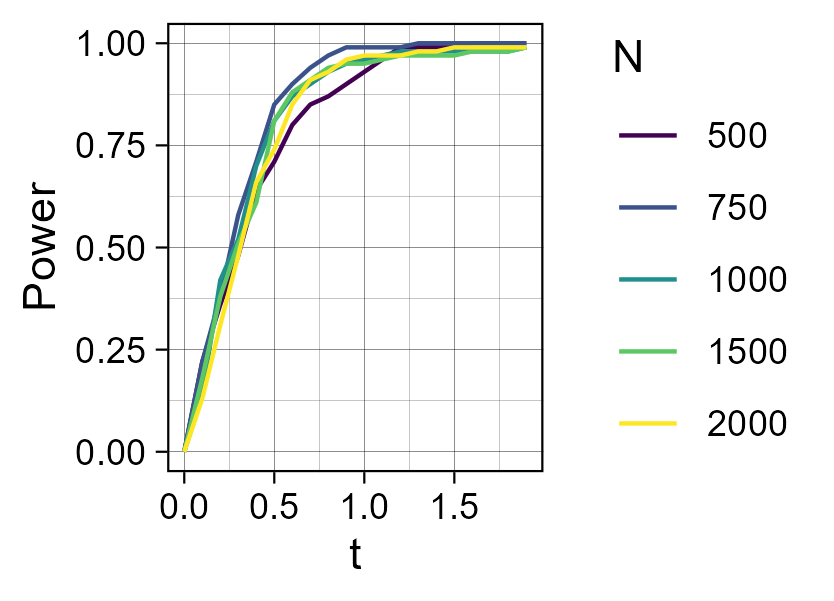}

       \includegraphics[width=6cm,height=3cm]{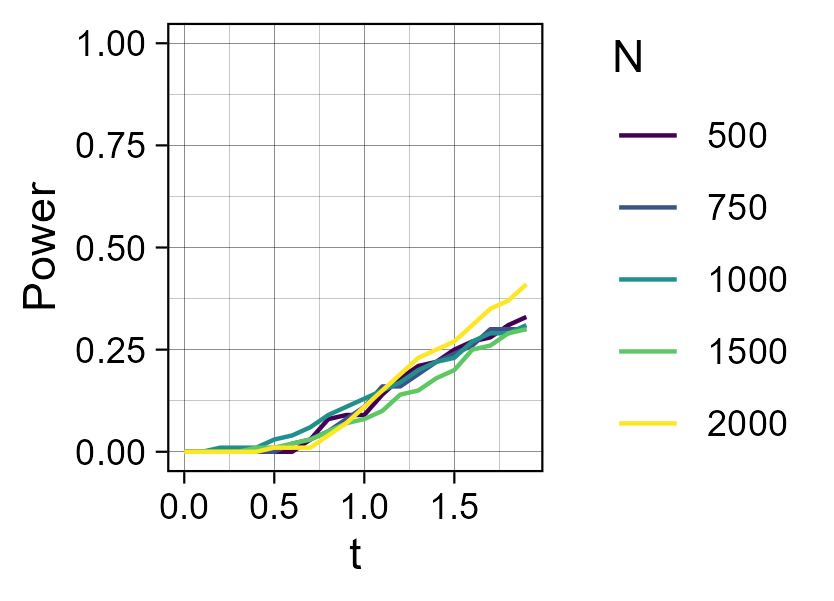}
    \includegraphics[width=6cm,height=3cm]{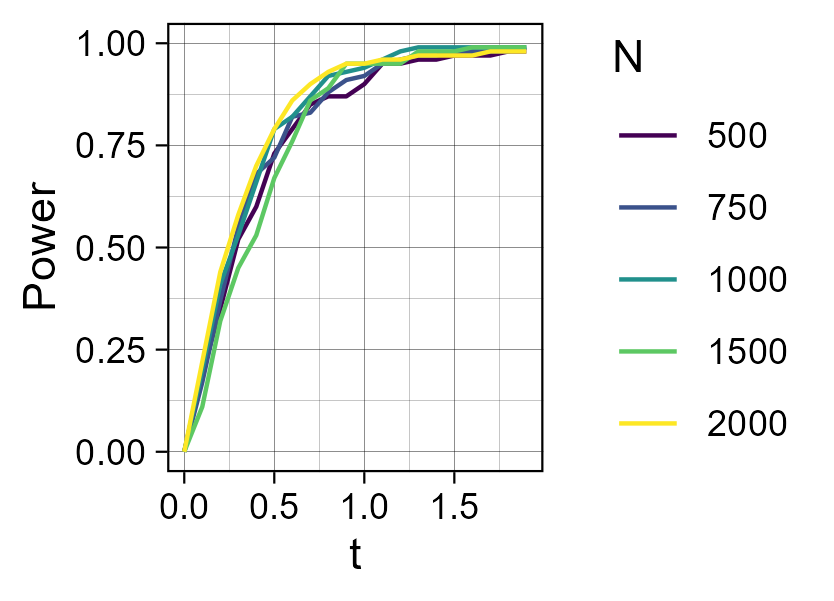}

       \includegraphics[width=6cm,height=3cm]{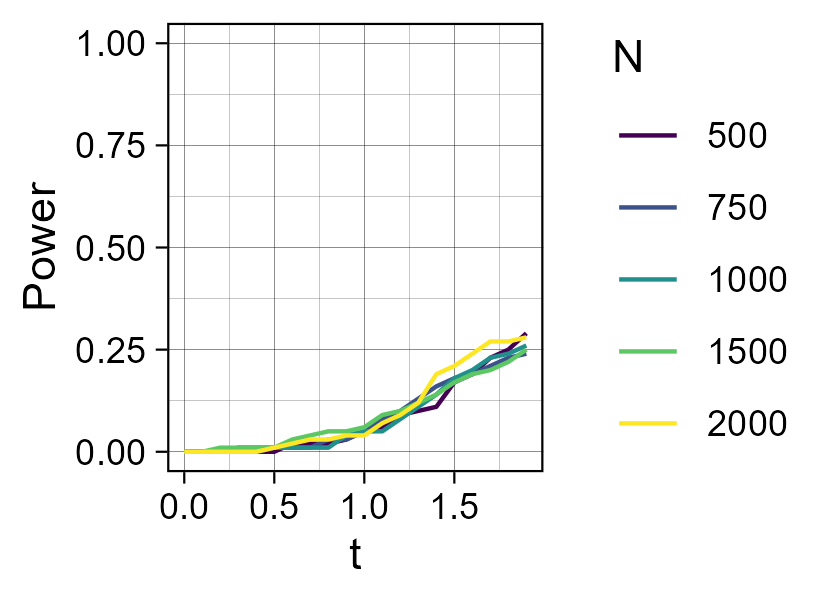}
    \includegraphics[width=6cm,height=3cm]{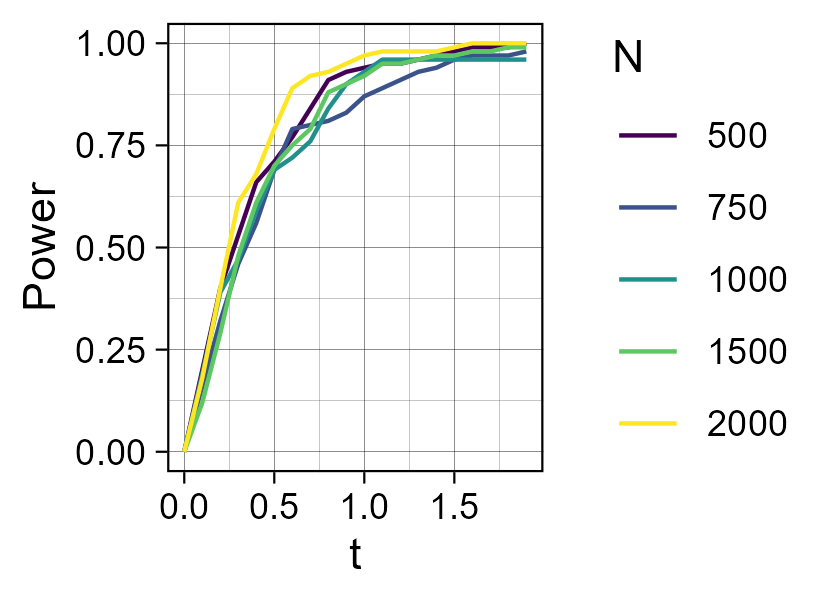}
        \caption{Empirical power of the estimation of  $t_*$ when $t_*\in\{0,0.1,0.3,0.5,0.7\}$ (top to bottom), $\btheta_*=(3,3,1)$,  $\kappa_{\hat Q,\btheta_0,\btheta_*}=2$ and $\kappa_{\hat \bGamma,\btheta_0,\btheta_*}=3.03$, based on $\widehat Q$ (left panel) and
     $\widehat \bGamma$ (right panel).}\label{fig:pow_tauStar_t7_331}
\end{figure}

\begin{figure}[ht!]
    \centering

                  \includegraphics[width=6cm,height=3cm]{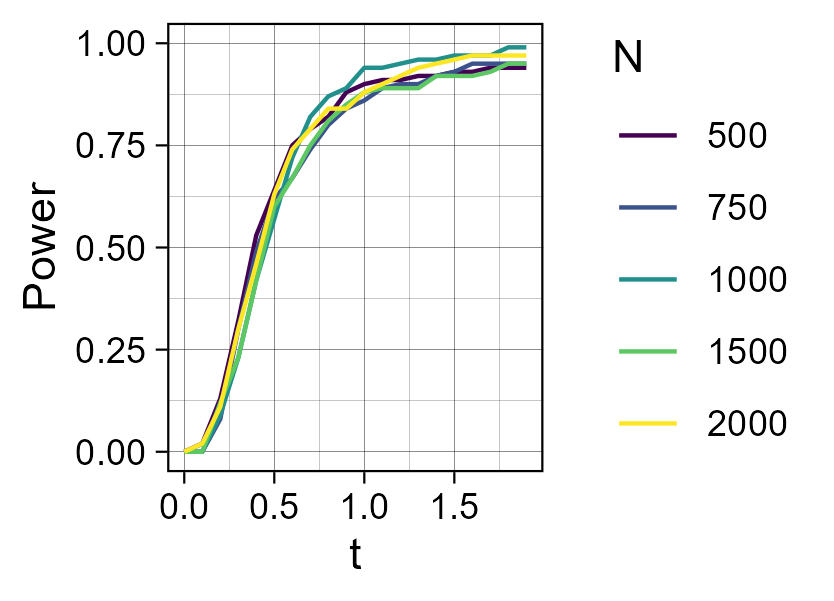}
    \includegraphics[width=6cm,height=3cm]{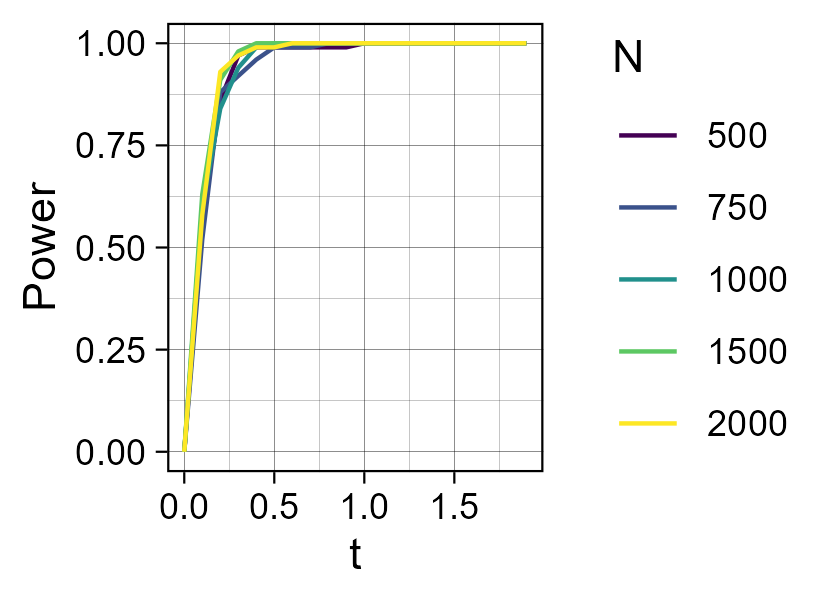}

      \includegraphics[width=6cm,height=3cm]{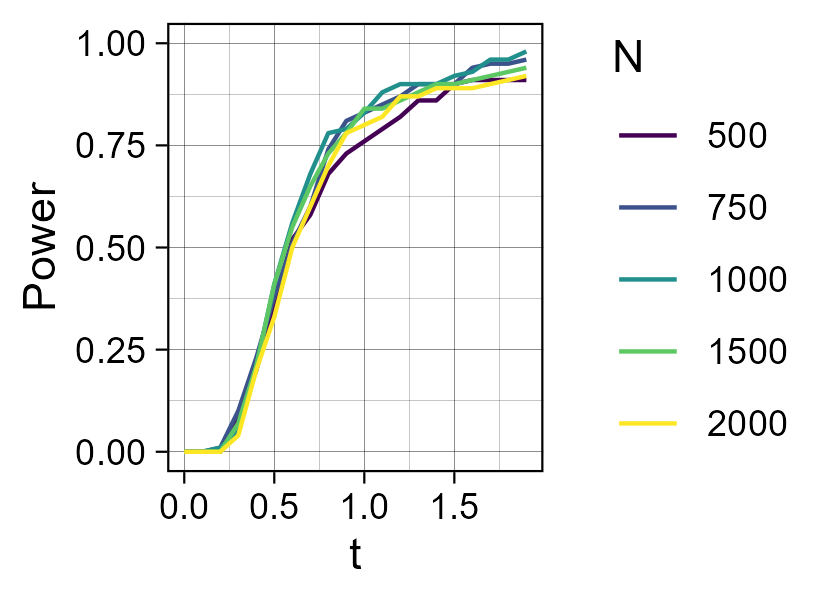}
    \includegraphics[width=6cm,height=3cm]{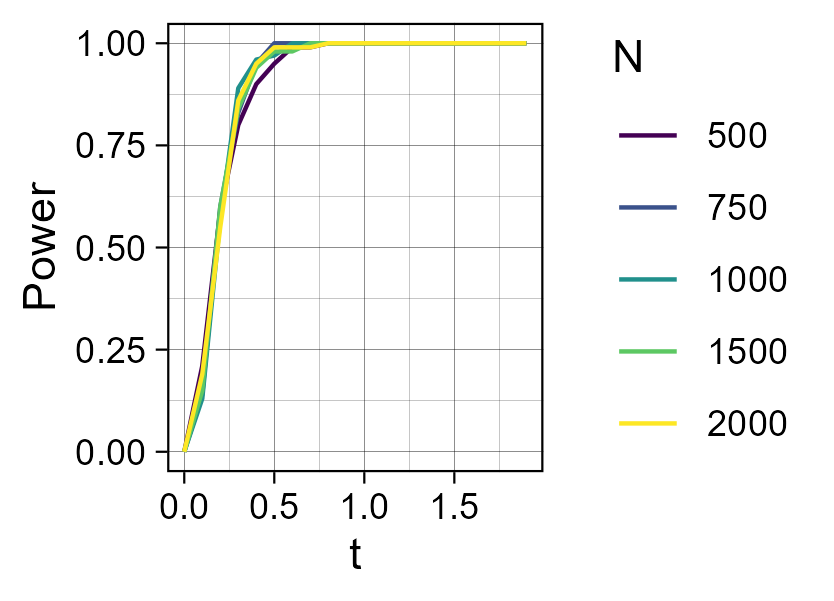}

                  \includegraphics[width=6cm,height=3cm]{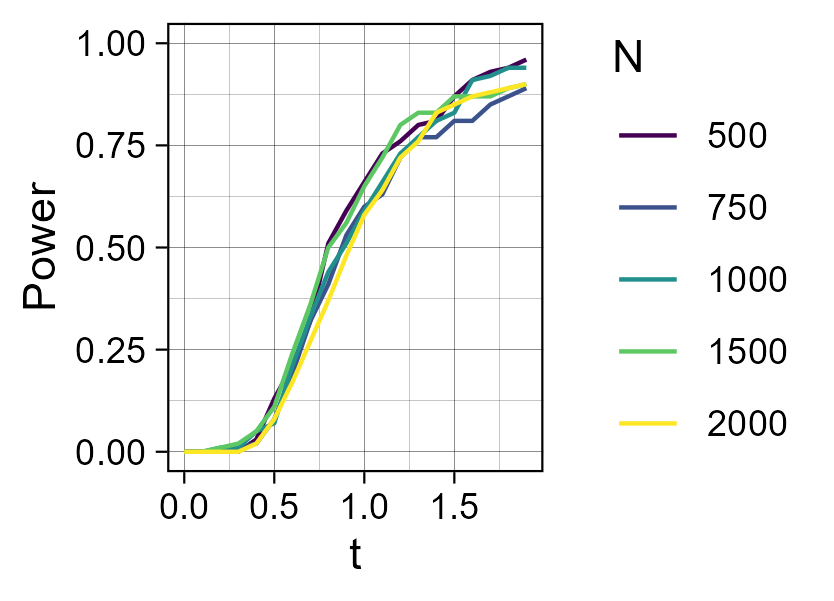}
    \includegraphics[width=6cm,height=3cm]{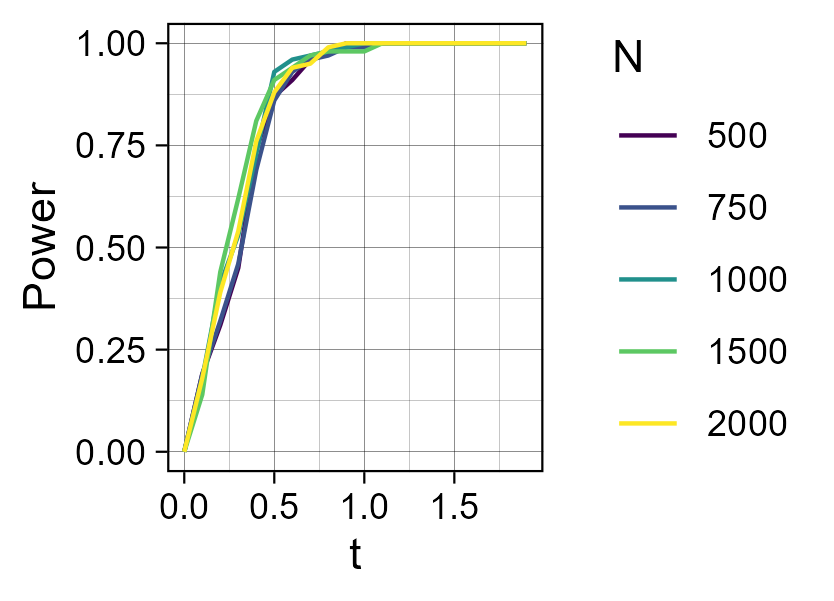}

                  \includegraphics[width=6cm,height=3cm]{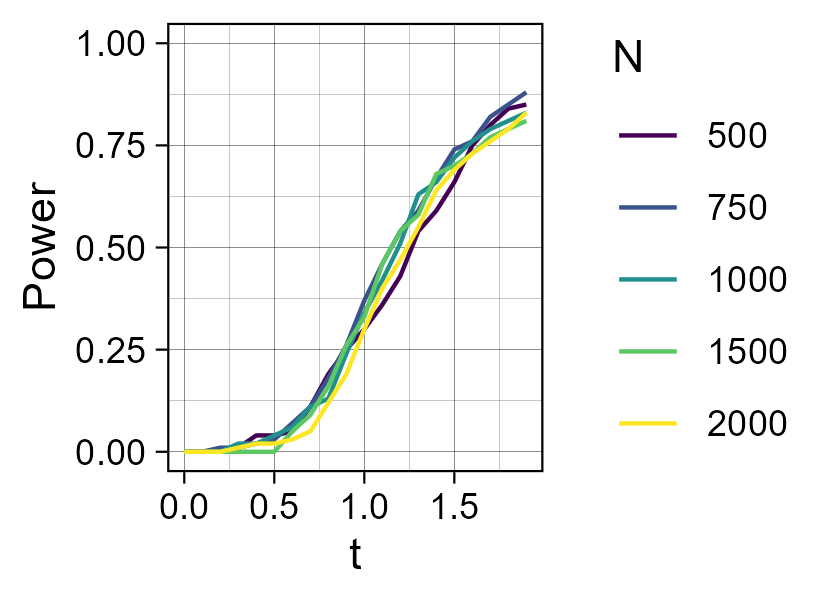}
    \includegraphics[width=6cm,height=3cm]{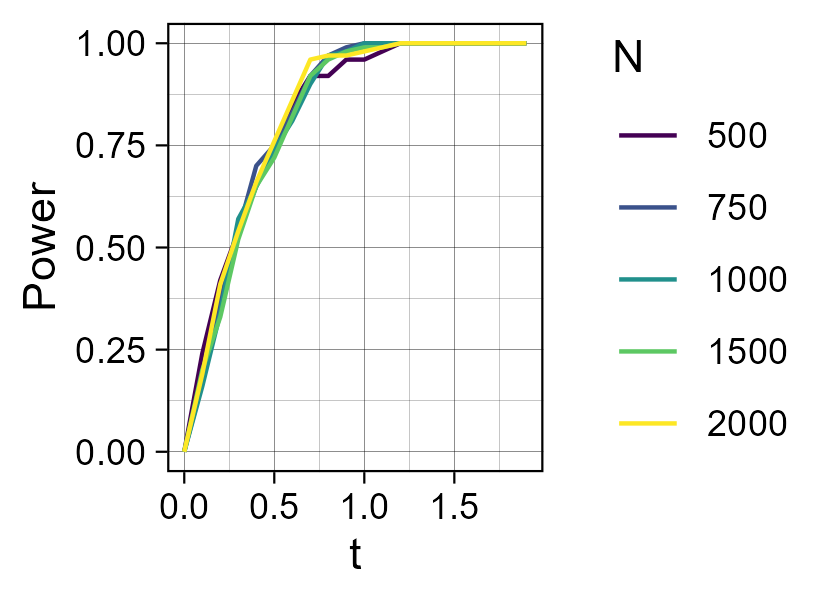}

                  \includegraphics[width=6cm,height=3cm]{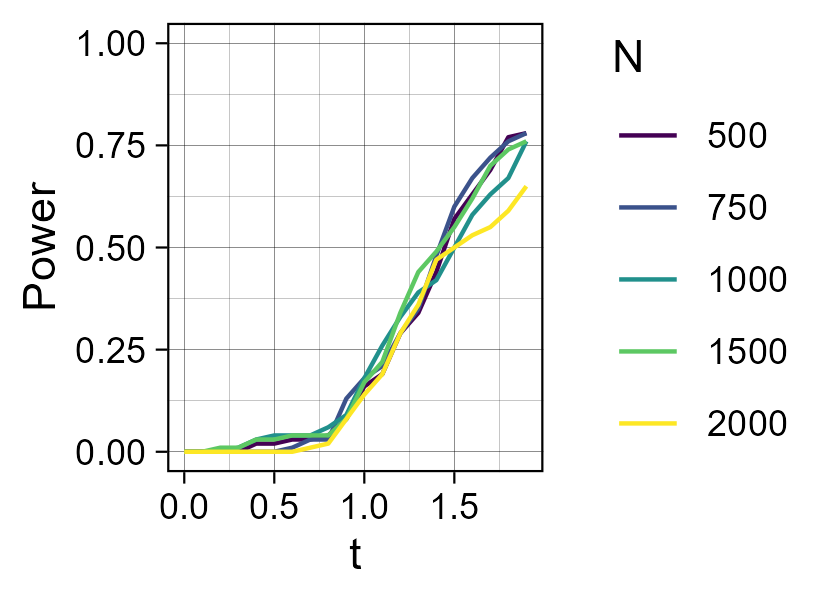}
    \includegraphics[width=6cm,height=3cm]{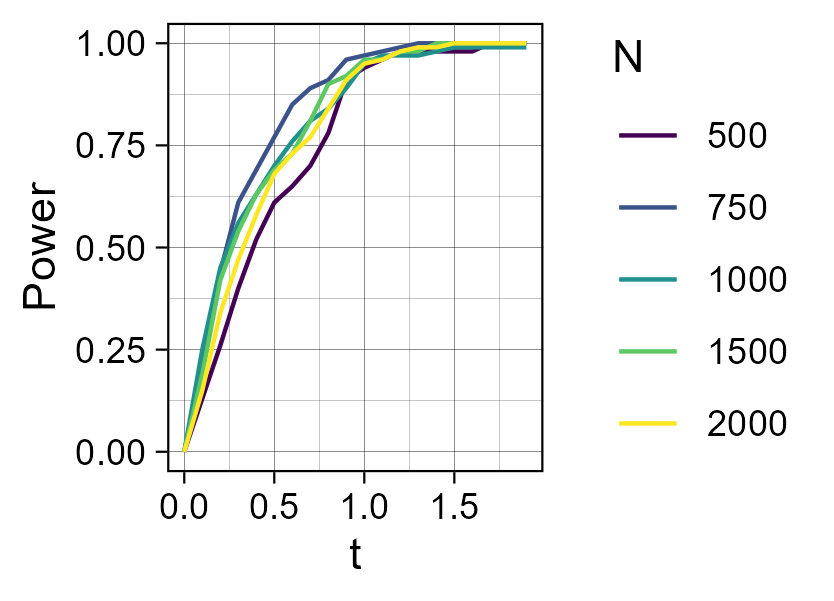}
        \caption{Empirical power of the estimation of  $t_*$ when $t_*\in\{0,0.1,0.3,0.5,0.7\}$ (top to bottom), $\btheta_*=(5,3,1)$, $\kappa_{\hat Q,\btheta_0,\btheta_*}=4$ and $\kappa_{\hat \bGamma,\btheta_0,\btheta_*}=5.78$, based on $\widehat Q$ (left panel) and
     $\widehat \bGamma$ (right panel).}\label{fig:pow_tauStar_t7_531}
\end{figure}

%% file: main_arxiv_Bouchra.bbl
\begin{thebibliography}{}

\bibitem[Berkes et~al., 2004]{Berkes/Gombay/Horvath/Kokoszka:2004}
Berkes, I., Gombay, E., Horváth, L., and Kokoszka, P. (2004).
\newblock {Sequential change-point detection in GARCH(p, q) models}.
\newblock {\em Econometric Theory}, 20(6):1140--1167.

\bibitem[Dehling et~al., 2010]{dehling2010drift}
Dehling, H., Franke, B., and Kott, T. (2010).
\newblock Drift estimation for a periodic mean reversion process.
\newblock {\em Statistical Inference for Stochastic Processes}, 13:175--192.

\bibitem[Ghoudi and Rémillard, 1998]{Ghoudi/Remillard:1998}
Ghoudi, K. and Rémillard, B. (1998).
\newblock Empirical processes based on pseudo-observations.
\newblock In {\em Asymptotic methods in probability and statistics ({O}ttawa,
  {ON}, 1997)}, pages 171--197. North-Holland, Amsterdam.

\bibitem[Ghoudi and Rémillard, 2015]{Ghoudi/Remillard:2015}
Ghoudi, K. and Rémillard, B. (2015).
\newblock {Diagnostic tests for innovations of ARMA models using empirical
  processes of residuals}.
\newblock In {\em Asymptotic laws and methods in stochastics}, volume~76 of
  {\em Fields Inst. Commun.}, pages 239--282. Fields Inst. Res. Math. Sci.,
  Toronto, ON.

\bibitem[Horv{\'a}th et~al., 2004]{horvath2004monitoring}
Horv{\'a}th, L., Hu{\v{s}}kov{\'a}, M., Kokoszka, P., and Steinebach, J.
  (2004).
\newblock Monitoring changes in linear models.
\newblock {\em Journal of Statistical Planning and Inference}, 126(1):225--251.

\bibitem[Kloeden and Platen, 1999]{kloeden1999numerical}
Kloeden, P.~E. and Platen, E. (1999).
\newblock {\em Numerical Solution of Stochastic Differential Equations}.
\newblock Springer-Verlag Berlin Heidelberg.

\bibitem[Kutoyants, 2004]{Kutoyants2004}
Kutoyants, Y.~A. (2004).
\newblock {\em Statistical Inference for Ergodic Diffusion Processes}.
\newblock Springer-Verlag, London.

\bibitem[Lilliefors, 1967]{Lilliefors:1967}
Lilliefors, H.~W. (1967).
\newblock {On the Kolmogorov-Smirnov test for normality with mean and variance
  unknown}.
\newblock {\em J. Amer. Statist. Assoc.}, 62:399--402.

\bibitem[Liptser and Shiryaev, 2001]{liptser2001statistics}
Liptser, R. and Shiryaev, A. (2001).
\newblock {\em {Statistics of Random Process I: General Theory}}.
\newblock Springer Berlin.

\bibitem[Lyu et~al., 2025]{Lyu/Nasri/Remillard:2025a}
Lyu, Y., Nasri, B.~R., and Rémillard, B.~N. (2025).
\newblock {\em {GenOU: Sequential Change-Point Tests for Generalized
  Ornstein-Uhlenbeck Processes}}.
\newblock R package version 0.2.1.

\bibitem[Lyu and Nkurunziza, 2023]{lyu2023inference}
Lyu, Y. and Nkurunziza, S. (2023).
\newblock Inference in generalized exponential {O-U} processes.
\newblock {\em Statistical Inference for Stochastic Processes}, 26(3):581--618.

\bibitem[Nasri and Rémillard, 2019]{Nasri/Remillard:2019}
Nasri, B.~R. and Rémillard, B.~N. (2019).
\newblock Copula-based dynamic models for multivariate time series.
\newblock {\em J. Multivariate Anal.}, 172:107--121.

\bibitem[Nkurunziza and Fu, 2019]{nkurunziza2019improved}
Nkurunziza, S. and Fu, K. (2019).
\newblock Improved inference in generalized mean-reverting processes with
  multiple change-points.
\newblock {\em Electronic Journal of Statistics}, 13(1):1400--1442.

\bibitem[Nkurunziza and Zhang, 2018]{nkurunziza2018estimation}
Nkurunziza, S. and Zhang, P.~P. (2018).
\newblock Estimation and testing in generalized mean-reverting processes with
  change-point.
\newblock {\em Statistical Inference for Stochastic Processes}, 21:191--215.

\bibitem[Page, 1954]{page1954continuous}
Page, E.~S. (1954).
\newblock Continuous inspection schemes.
\newblock {\em Biometrika}, 41(1/2):100--115.

\bibitem[Rémillard and Vaillancourt, 2024]{Remillard/Vaillancourt:2024a}
Rémillard, B. and Vaillancourt, J. (2024).
\newblock {Central limit theorems for martingales-I : continuous limits}.
\newblock {\em Electronic Journal of Probability}, 29:1--18.

\bibitem[Stout, 1974]{Stout1974}
Stout, W.~F. (1974).
\newblock {\em Almost Sure Convergence}.
\newblock Academic Press, New York.

\end{thebibliography}
